\documentclass[EJP,preprint]{ejpecp} % replace ECP by EJP if needed.
\usepackage[utf8]{inputenc}

\usepackage{amsmath}
\usepackage{amssymb}
\usepackage{amsfonts}
\usepackage{amsthm}
\usepackage{graphicx}
\usepackage{enumerate}
\usepackage[shortlabels]{enumitem}
\usepackage[english]{babel}
\usepackage{bm}
\usepackage{latexsym}
\usepackage{caption}
\usepackage{soul}
\usepackage{bbm}
\SHORTTITLE{Sampling probabilities, diffusions, ancestral graphs, and duality under strong selection}

\TITLE{Sampling probabilities, diffusions, ancestral graphs, and duality under strong selection} %  Insert line breaks with \\

\AUTHORS{
  Martina~Favero
  \footnote{Department of Statistics, University of Warwick, Coventry CV4 7AL, UK}
  \footnote{Current address: Department of Mathematics, Stockholm University.\\ %106 91 Sweden. 
 \EMAIL{martina.favero@math.su.se}}%\orcid{0000-0002-7069-9721}
  \and 
  Paul~A.~Jenkins 
  \footnotemark[1]
  \footnote{Department of Computer Science, University of Warwick, Coventry CV4 7AL, UK}
  }
 % Separate items with ;

\AMSSUBJ{60J90} % Edit. Separate items with ;
\AMSSUBJSECONDARY{60J60, 92D15} % Optional, separate items with ;

\VOLUME{0}
\YEAR{2025}
\PAPERNUM{0}
\DOI{10.1214/YY-TN}

\ABSTRACT{
Wright--Fisher diffusions and their dual ancestral graphs occupy a central role in the study of allele frequency change and genealogical structure, and they provide expressions, explicit in some special cases but generally implicit, for the sampling probability, a crucial quantity in inference. 
Under a finite-allele mutation model, with possibly parent-dependent mutation, we consider the asymptotic regime where the selective advantage of one allele grows to infinity, while the other parameters remain fixed. 
In this regime, we show that the Wright--Fisher diffusion can be approximated either by a Gaussian process or by a process whose components are independent continuous-state branching processes with immigration, aligning with analogous results for Wright--Fisher models but employing different methods. While the first process becomes degenerate at stationarity, 
the latter does not and provides a simple, analytic approximation for the leading term of the sampling probability. 
Furthermore, using another approach based on a recursion formula, we characterise all remaining terms to provide a full asymptotic expansion for the sampling probability. 
Finally, we study the asymptotic behaviour of the rates of the block-counting process of the conditional ancestral selection graph and establish an asymptotic duality relationship between this and the diffusion. 
}

\newcommand{\E}[1]{\mathbb{E}\left[#1\right]}

\newcommand{\ind}[1]{\mathbb{I}_{\left\{#1\right\}}}
\newcommand{\N}{\boldsymbol{N}}
\newcommand{\M}{\boldsymbol{M}}
\newcommand{\Q}{\boldsymbol{Q}}
\newcommand{\R}{\boldsymbol{R}}
\newcommand{\X}{\boldsymbol{X}}
\newcommand{\Y}{\boldsymbol{Y}}
\newcommand{\W}{\boldsymbol{W}}
\newcommand{\U}{\boldsymbol{U}}
\newcommand{\Z}{\boldsymbol{Z}}
\newcommand{\x}{\boldsymbol{x}}
\newcommand{\z}{\boldsymbol{z}}
\newcommand{\xii}{\boldsymbol{\xi}}
\newcommand{\nuu}{\boldsymbol{\nu}}
\newcommand{\n}{\boldsymbol{n}}
\newcommand{\m}{\boldsymbol{m}}
\newcommand{\e}{\boldsymbol{e}}
\newcommand{\Ex}[1]{\mathbb{E}_{\x}\left[#1\right]}
\newcommand{\norm}[1]{\|#1\|_1}
\newcommand{\Nth}{^{(N)}}
\newcommand{\sth}{^{(\sigma)}}
\DeclareMathOperator{\cov}{Cov}
\begin{document}

%%%%%%%%%%%%%%%%%%%%%%%%%%%%%%%%%%%%%%%%%%%%%%%%%%%%%%%%%%%%%%%%%%%
%%                                                               %%
%% No need for \maketitle.                                       %%
%%                                                               %%
%%%%%%%%%%%%%%%%%%%%%%%%%%%%%%%%%%%%%%%%%%%%%%%%%%%%%%%%%%%%%%%%%%%

%%%%%%%%%%%%%%%%%%%%%%%%%%%%%%%%%%%%%%%%%%%%%%%%%%%%%%%%%%%%%%%%%%%
%%                                                               %%
%% Please replace what follows by the body of your article       %%
%% (up to the bibliography):                                     %%
%%                                                               %%
%%%%%%%%%%%%%%%%%%%%%%%%%%%%%%%%%%%%%%%%%%%%%%%%%%%%%%%%%%%%%%%%%%%
\section{Introduction}

%%%%%%%% WF DIFF, ASG & SAMPLING PROB  %%%%%%%%%%%%%%%%%%%%%%%%%%%%%
We are interested in evolutionary models for a large, randomly mating population in which genetic frequencies undergo stochastic changes over time owing to random genetic drift and other evolutionary processes. We study two closely related stochastic models, the \emph{Wright--Fisher (WF) diffusion} with mutation and selection, which models genetic frequencies evolving forwards in time in the population, and its dual, the \emph{Ancestral Selection Graph (ASG)}, which evolves backwards in time and models the ancestral history of a sample of genetic sequences taken from some individuals in the same population.
Interest in these models is often motivated by the desire to calculate the \emph{sampling probability}, the probability of observing a certain configuration of alleles when sampled from the population, which is typically the quantity of interest for statistical inference. Under certain duality conditions, discussed later, the sampling probability for the two models are equivalent. However, except for the special case of a parent-independent mutation (PIM) model, the implicit nature of the expressions poses a significant obstacle. 

%%%%%%%% ASYMPTOTIC REGIME & UNDERLYING ASSUMPTIONS %%%%%%%%%%%%%%%%%%%%%%%%%%%%%

Let us give a more detailed description of the models under consideration. We assume a finite number $d$ of genetic types, or alleles, and a general mutation mechanism with mutation rate parametrised by $\theta$ and mutation probability matrix $P$; that is, when a type $i$ undergoes mutation its offspring is of type $j$ with probability $P_{ij}$. 
We make the standard assumption of $P$ being an irreducible matrix to ensure the existence of a stationary density for the WF diffusion (see Section \ref{sect:stationary} for more details). It should be noted that this assumption is not required for the results of Section \ref{sect:diffusions}. 
We consider single-locus haploid models where the selective pressure on the population is expressed by the selection parameters $\sigma_{i}, i=1,\dots, d$ representing the fitness of allele $i$. 
We consider the scenario in which the selective advantage of one allele (without loss of generality allele $1$ throughout) grows to infinity, while all the other parameters remain fixed. 
That is, we analyse the asymptotic regime given by $\sigma_1=\sigma\to\infty$. We use the superscript $(\sigma)$ to emphasise the dependence on this parameter. 
Under this asymptotic regime, we study the behaviour of two processes, the WF diffusion and the ASG, and provide a full asymptotic expansion of the sampling probability. Furthermore, we establish an asymptotic duality relationship between two limiting processes. 

%%%%%%%%  MOTIVATION  %%%%%%%%%%%%%%%%%%%%%%%%%%%%%

Our asymptotic analysis, besides contributing to the theory of WF diffusions and ancestral processes, is also relevant for, and motivated by,  applied problems.
First, the full expansion of the sampling probability  provides, in the strong selection regime, an approximation to the true sampling probability which is unknown. That it is unknown poses a significant obstacle in inference.
%as well as a challenge in our analysis. 
Furthermore, we derive a diffusion approximation  that  can be simulated exactly, as discussed later, and thus provides an alternative way of  simulating the WF diffusion under strong selection in those cases in which an exact simulation algorithm is not available.
Exact simulation algorithms for the  WF diffusion are available in the multi-allele PIM case \cite{jenkins2017, sant2023}; however, a generalisation to the case with a general mutation mechanism is neither available, nor straightforward. 
Finally, the asymptotic analysis of the ASG and its duality complete the picture, providing a genealogical interpretation. We now describe more precisely the main results and the outline of the paper.

%%%%%%%% SECT:SAMPLING_PROB %%%%%%%%%%%%%%%%%%%%%%%%%%%%%

A complete asymptotic analysis of the sampling probability, in Section \ref{sect:sampling_prob},  constitutes the main contribution of this paper.
As a first step we derive a \emph{Gamma approximation} of the sampling probability (Proposition \ref{thm:gamma_sampling}) 
to replace the classical \emph{Gaussian approximation} which fails in this setting, see e.g.\ \cite{jenkins2015} for a strong-recombination setting in which the Gaussian approximation \emph{can} be applied.
The role of the Gamma approximation is to guide the choice of the correct asymptotic scaling by indicating that each unfit allele in the sample introduces an additional multiplicative factor $\sigma^{-1}$ to the decay of the sampling probability, unlike the $\sigma^{-\frac{1}{2}}$ one would expect with the Gaussian approximation originating from the classical central limit theorem. 
However, this approximation provides only the first order of the asymptotic expansion of the sampling probability.

To provide a full asymptotic expansion of the sampling probability in inverse powers of $\sigma$ 
we use instead a recursion formula for the sampling probability. This allows us to derive a general characterisation of the coefficients in the expansion to all orders (Theorem \ref{thm:characterisation_asyptotic_coefficients}). 
We also provide more explicit expressions for the first two coefficients and, in the special case of parent-independent mutation, for all of the coefficients (Proposition \ref{thm:characterisation_asyptotic_coefficients_PIM}). 

%%%%%%%% SECT:DIFFUSIONS %%%%%%%%%%%%%%%%%%%%%%%%%%%%%

In order to understand the origin of the Gamma approximation and the lack of a Gaussian approximation in this setting, the asymptotic and stationary behaviour of the WF diffusion must be studied. 
In Section \ref{sect:diffusions}, we analyse the asymptotic behaviour of the WF diffusion (summarised in Proposition \ref{thm:diffusion_limits}) through a law-of-large-number type result, i.e.\ as the diffusion noise vanishes we obtain a logistic trajectory with an equilibrium point on the boundary of the state space; and two different central-limit-theorem type results: (i) a Gaussian process as a limit of the scaled fluctuations around this logistic trajectory, and (ii) a special type of continuous-state branching processes with immigration (CBI processes), or Cox--Ingersoll--Ross (CIR) process, as a limit of the (differently) scaled fluctuations around the equilibrium point. 
The equilibrium point being on the boundary is precisely the reason that this asymptotic regime is characterised by two different central-limit-theorem type results, instead of only by the classical Gaussian central limit theorem. 

The limiting processes we obtain naturally correspond to those obtained from the analogous asymptotic analysis of the corresponding---this time discrete in both time and space---WF model as the population size goes to infinity and parameters are scaled appropriately; this is sketched in the Appendix. 
More precisely, the logistic trajectory and the Gaussian limits are known in the literature of WF models as part of rather general %(and thus not straightforward to interpret) 
results concerning various types of asymptotic regimes \cite{Norman1972,Norman1975,Nagylaki1986}. Results of this type have found use in modelling selective sweeps \cite{Feder2014,lacerda2014}; see Section \ref{sect:diffusions} for more details on the literature. 
The CBI limit, on the other hand, which appears in \cite{Ethier1988} and \cite{Nagylaki1990} in a study concerning rare alleles, seems nonetheless to be less well known. 

Section \ref{sect:diffusions} has a twofold nature. Because some of our results have their analogues with existing results for strong selection in \emph{discrete} WF models, in part it provides a short review of the literature of WF models in the specific asymptotic regime of interest.   
Second, we focus on the WF diffusion itself and employ different methods that,  being applied to a sequence of WF diffusions, rather than a corresponding sequence of WF models, allow more sleek calculations and a lighter notation,  which facilitate understanding and provide a clear picture of the strong-selection asymptotic regime. 
Another way of putting this is that asymptotic results for WF models  require a joint consideration of both population size and/or selection parameter going to infinity at appropriate rates; 
on the contrary for the corresponding diffusion model the population size has effectively  been taken to infinity already, enabling us to concentrate directly on the effects of strong selection. 
The former approach---arguably more natural from a modelling perspective but involving less transparent calculations---is the focus of the Appendix.

%%%%%%%% SECT:STATIONARITY %%%%%%%%%%%%%%%%%%%%%%%%%%%%%

Section \ref{sect:stationary} is dedicated to the asymptotic analysis of stationary distributions. 
Because the equilibrium point of the logistic trajectory is on the boundary of the state space, the Gaussian process becomes degenerate at stationarity. 
This explains the lack of the classical Gaussian approximation of the sampling probability and it also strongly motivates the study of the CBI limit, which instead has a non-degenerate stationary distribution with independent Gamma-distributed components and which leads to the alternative Gamma approximation. % and thus should not be neglected. 

%%%%%%%% SECT:GENEALOGY %%%%%%%%%%%%%%%%%%%%%%%%%%%%%

To complete the asymptotic analysis of strong selection we study the asymptotic behaviour of the ASG, in Section \ref{sect:genealogy}. In particular, we consider the block-counting process of a  \emph{conditional} ASG with typed lineages (rather than, as is common, with types superimposed afterwards), and study its asymptotic behaviour. After `reducing' the graph, we obtain two limiting processes acting on a slow and a fast timescale, as was shown in \cite{wakeley2008} for a model with two alleles. Our analysis generalises the analysis of \cite{wakeley2008} from the two- to the multi-allele case, and, in addition, 

%%%%%%%% SECT:DUALITY %%%%%%%%%%%%%%%%%%%%%%%%%%%%%

in Section \ref{sect:duality}, we also derive an asymptotic duality relationship between the limiting CBI process and the limiting ancestral processes acting on the fast timescale (Theorem \ref{thm:duality}).  
This is, to the best of our knowledge, the first asymptotic duality result of this kind.   
The derived duality relationship provides additional structure relating the various asymptotic objects and further confirms the major role of the CBI limit, compared to the Gaussian limit, in this particular asymptotic regime.

%%%%%%%% Comment on Fan,Wakeley2023 %%%%%%%%%%%%%%%%%%%%%%%%%%%%%

At the time of finalising the writing of this paper, we became aware of the very recent work of Fan and Wakeley \cite{fan2024},  which partly overlaps some of the problems treated here. 
More precisely, for the special case of two alleles, which falls under the PIM assumption,  \cite{fan2024} also calculate the first coefficient of the asymptotic expansion of the sampling probability and the CIR limit of the WF diffusion.  This constitutes the overlap between \cite{fan2024} and this paper. Nevertheless, the two papers focus on different directions. 
Our work provides a generalisation of the above-mentioned results to the general multi-allele and  parent-dependent mutation framework. In addition, we establish a new asymptotic duality relationship and derive a full asymptotic expansion of the sampling probability by characterising the coefficients of all orders. We use a significantly different approach which does not rely on the explicit expressions that are only available under the PIM assumption. 
Instead, \cite{fan2024} focus on other points which are not studied here: while their work is restricted to two alleles and to the first-order of the expansion of the sampling probability, they provide new results on the number of latent mutations, the limit of infinite sample size, and they also study the complication of a changing population size.

%%%%%%%% LARGE RECOMBINATION %%%%%%%%%%%%%%%%%%%%%%%%%%%%%

Finally, we note that this  work on large selection can be placed  in a wider framework of `large parameter' regimes. 
In particular, an exhaustive asymptotic analysis for strong recombination rates has been performed in \cite{alberti2023, bs2012, jenkins2015, jenkins2009,jenkins2010,jenkins2012}, which provide asymptotic expansions for the sampling probability as well as limiting Gaussian and ancestral processes and also inspired our analysis.

%%%%%%%%%%%%%%%%%%%%%%%%%%%%%%%%%%%%%%%%%%%%%%%%%%%%%%%%%%%%

\section{Two different diffusion limits}
\label{sect:diffusions}

The main object of this section is the  WF diffusion $\X^{(\sigma)}=\{ \X^{(\sigma)}(t)\}_{t\geq 0}$, defined below, with mutation and selection. More precisely the model comprises $d$ alleles, mutation rate $\theta$, mutation probability matrix $P$, and selection parameters $\sigma_i$, $i=1,\dots,d$. We study its asymptotic behaviour as $\sigma_1=\sigma\to\infty$. 
The diffusion $\X^{(\sigma)}$ 
evolves  on the simplex $\Delta= \{\x\in [0,1]^d: \sum_{i=1}^d x_i=1 \}$
and is characterised by its infinitesimal generator 
\begin{align}
	\nonumber
	&\mathcal{L}^{(\sigma)}= \frac{1}{2}
	\sum_{i,j=1}^d 
	d_{ij}(\x)     \frac{\partial^2}{\partial x_i \partial x_j }
	+ \sum_{i=1}^d \left[\mu_i(\x)+ s_i(\x) \right]
	\frac{\partial}{\partial x_i },
	\\
	\label{generatorWFdiff}
	& \ \text{with} \quad 
	d_{ij}(\x)=x_i (\delta_{ij}-x_j) ,
	\quad 
	\mu_i(\x)= 
	\frac{\theta}{2}  \left( \sum_{j=1}^d P_{ji}x_j - x_i   \right),
	\quad
	s_i(\x)= 
	\frac{x_i}{2}  \left(\sigma_i - \sum_{j=1}^d \sigma_j x_j \right),
\end{align}
and domain $ C^2(\Delta)$; or equivalently, by the SDE
\begin{align}
	\label{eq:WF_SDE}
	d\X\sth(t)&=[ \mu(\X\sth(t)) + s(\X\sth(t)) ] dt + D^{1/2}(\X\sth(t)) d\W(s), 
	& \X\sth(0) &= \x\sth(0),
\end{align}
where $\mu(\x)=(\mu_i(\x))_{i=1}^d$, $s(\x)=(s_i(\x))_{i=1}^d$ and $D(\x)=(d_{ij}(\x))_{i,j=1}^d$ being a positive semidefinite matrix, with a square root $D^{1/2}$ satisfying $D = D^{1/2}(D^{1/2})^\top$. 
See for example \cite[Ch.\ 10]{Ethier1986} for more details on the WF diffusion.

The diffusion $\X\sth$ arises as the limit of a sequence of WF models as the population size grows to infinity under the assumption that the strength of mutation, selection and genetic drift balance each other in this limit. 
For this reason, our analysis on the asymptotic behaviour of the WF diffusion yields analogous results to the asymptotic analysis of its finite-population size counterpart; see the Appendix for details. 
As mentioned in the introduction, some of these asymptotic results for WF models are well known, i.e.\ the deterministic limit and the Gaussian limit, while others, i.e.\ the CBI limit, are lesser known and even sometimes overlooked, which is why we believe it is valuable to review them here by matching them to the analogous results that we derive for the WF diffusion and by citing the relevant literature at each step.
We choose to work with the WF diffusion (and to confine the re-derivation of results from  WF models to the Appendix for comparison) both because it allows sleeker calculations, which in turn facilitate understanding and intuition regarding the limiting behaviours; and because it relies on a different (and convenient) theoretical framework based on powerful tools, such as SDEs, infinitesimal generators and martingales.
To facilitate the reading, we first concisely state the three main convergence results in the following proposition. 

Here, and throughout the paper, we denote convergence in distribution by the arrow  $\xrightarrow[]{d}$. In particular, 
as is customary \cite{billingsley1999, Ethier1986},
the convergence in distribution of a sequence of stochastic processes refers to the weak convergence of the corresponding sequence of probability measures, their  distributions or laws, in the Skorokhod space, i.e.\ in the space of c\`adl\`ag (right continuous with left limits)  functions on the time interval $[0,\infty)$  taking values in the state space of the processes, equipped with the classical Skorokhod topology. 

\begin{proposition}
\label{thm:diffusion_limits}
Let $\X\sth$ be the WF diffusion defined as the solution to \eqref{eq:WF_SDE}. Then the following convergence results hold in the strong selection limit. 
\\
\textbf{(a) $\X^{(\sigma)}$ converges to a deterministic logistic trajectory}. 
That is, assuming  $\X\sth(0)\xrightarrow[]{d}\xii(0),$ as $\sigma\to\infty$, we have
    \begin{align*}
    \left\{\X^{(\sigma)}(t/\sigma)\right\}_{t\geq 0}
    \xrightarrow[]{d}  \left\{\xii(t)\right\}_{t\geq 0} ,
    \quad \text{as} \quad 
    \sigma\to\infty,
    \end{align*}
where $\xii$ is the logistic trajectory defined by  the following  ODE 
    \begin{align}
    \label{eq:ODE_deterministic_trajectory}
    \frac{d}{dt} \xii (t)= 
    \boldsymbol{\omega}(\xii(t)), 
    \quad \text{with} \quad 
    \boldsymbol{\omega}(\x)=  \frac{x_1}{2} (\e_1 -  \x ),
    %=  \frac{\xi_1(t)}{2} (\e_1 -  \xii(t) ),
    \end{align}
	%with 
	%    \begin{align*}
		%    \omega_i(\x)= 
		%    \frac{x_1}{2}
		%     (\delta_{1 i } -  x_i ),
		%      \comm{= \frac{x_i}{2}
			%     (\delta_{1 i } -  x_1 )}
		%    \end{align*}
or, more explicitly,  
	%$\xii$ is the logistic trajectory with components 
\begin{align*}
    \xi_i (t)= 
    \frac{\xi_i(0) e^{\frac{t}{2}\delta_{1 i }}}{ \xi_1(0)e^{\frac{t}{2}}+ 1-\xi_1(0)}
    %\comm{=
    %\begin{cases}
    %\frac{\xi_1(0)}{ \xi_1(0)+ e^{-t/2}(1-\xi_1(0))}, &\quad\text{for } i=1; 
    %\\
    %\frac{\xi_i(0)}{ e^{t/2}\xi_1(0)+ 1-\xi_1(0)}, &\quad\text{for } i\neq 1,
    %\end{cases}}
    \end{align*}
with equilibrium point  $\xii(\infty)=\e_1 := (\delta_{1i})_{i=1}^d$, if $\xi_1(0)>0$; and $\xii(\infty)=\xii(t)=\xii(0)$, if $\xi_1(0)=0$. 
\\
\textbf{(b) The $\sqrt{\sigma}$-scaled fluctuations of $\X\sth$ around the logistic trajectory $\xii$ converge to a Gaussian process}.  
That is, assuming  
$\sqrt{\sigma}\left[\X^{(\sigma)}(0) - \xii(0) \right]\xrightarrow[]{d}\U(0),$ as $\sigma\to\infty$, we have
    \begin{align*}
    \left\{\sqrt{\sigma}\left[\X^{(\sigma)}(t/\sigma) - \xii(t) \right]\right\}_{t\geq 0}
    \xrightarrow[]{d} 
    \left\{\U(t)\right\}_{t\geq 0} ,
    \quad \text{as} \quad 
    \sigma\to\infty,
    \end{align*}
where $\U$ is  the Gaussian process
defined as the solution to the following linear SDE with time-dependent coefficients:
    \begin{align}
    \label{eq:SDE_Gaussian}
    d\U(t)&=
    J\boldsymbol{\omega} (\xii(t))  \U(t)  dt
    + 
    D^{1/2}(\xii (t)) d\W(t),
    \end{align}
with $\boldsymbol{\omega}$ as in \eqref{eq:ODE_deterministic_trajectory} and $J$ denoting the Jacobian matrix, i.e.\ $[J\boldsymbol{\omega} (\xii(t))]_{ij} = \frac{d}{dx_j}\omega_i(\xii(t))$, 
and thus, more explicitly, 
    \begin{align*}
    J\boldsymbol{\omega} (\xii(t))  \U(t) =
    - \frac{1}{2} \left\{
    U_1(t) \left[ \xii (t) - \e_1  \right]
    +
    \xi_1(t) \U(t) \right\}  .
    \end{align*}
\textbf{(c) The $\sigma$-scaled fluctuations of $\X\sth$ around the equilibrium point $\e_1$ converge to CBI processes}. That is, assuming  
${\sigma}\left[\X^{(\sigma)}(0) - \e_1 \right]\xrightarrow[]{d}\Z(0),$ as $\sigma\to\infty$, we have
    \begin{align*}
    \left\{\sigma\left[\X^{(\sigma)}(t/\sigma) - \e_1 \right]\right\}_{t\geq 0}
    \xrightarrow[]{d}  
    \left\{\Z(t)\right\}_{t\geq 0}  ,
    \quad \text{as} \quad 
    \sigma\to\infty,
    \end{align*}
where for $i=2,\dots d$ the components $Z_i$ are the independent CBI processes defined by the following SDEs:
    \begin{align}
    \label{eq:SDE_CBI}
    d Z_i(t)&=  \frac{1}{2} (\theta P_{1i} - Z_i(t))dt + \sqrt{Z_i(t)} d W_i(t),
    \end{align}   
and $Z_1(t)= - \sum_{i=2}^d Z_i(t)$. 
\end{proposition}
	\begin{proof}
		We show (a) at the end of this section; (b)  in Section \ref{sect:diffusion_gaussian}; and (c)  in Section \ref{sect:diffusion_CBI}.
	\end{proof}
	In a more compact expression, $\Z$ solves the SDE
	\begin{align}
		\label{eq:fluctuationsSDE}
		d \Z(t)=  
		\frac{1}{2}[\theta(\boldsymbol{P}_{1\boldsymbol{\cdot}}- \e_1)  -  \Z(t) ]dt
		% \frac{1}{2} [\theta (P\boldsymbol{\cdot}\boldsymbol{\cdot}_{1\bullet}-e_1 ) - Z(t)] dt 
		+    B(\Z(t)) d \W (t)
	\end{align}
	where 
	\begin{align*}
		% b(\z)= 
		%\theta(P_{1\boldsymbol{\cdot}}- e_1)  - \z;
		%\quad \quad
		B(\z)= 
		\begin{pmatrix}
			0      & -\sqrt{z_2} & -\sqrt{z_3}& \cdots & -\sqrt{z_d} \\
			0      & \sqrt{z_2} & 0         & \cdots & 0           \\
			\vdots & \ddots     &\sqrt{z_3} & \ddots &  \vdots    \\
			\vdots &            &  \ddots   & \ddots &  0           \\
			0  &   \cdots   &  \cdots   & 0      & \sqrt{z_d}\\
		\end{pmatrix},
	\end{align*}
and $\boldsymbol{P}_{1\boldsymbol{\cdot}} := (P_{1i})_{i=1}^d$. Before moving on to a more detailed description of the limiting processes and their derivations, let us make some remarks. 
It is interesting to note that the equilibrium point $\e_1$ of the logistic trajectory $\xii$ is on the boundary of the state space $\Delta$. This is a key feature of the asymptotic regime of this paper. In fact, it is exactly this feature that causes the two different asymptotic behaviours (b) and (c) in Proposition \ref{thm:diffusion_limits}. In other asymptotic regimes, in which the equilibrium point is in the interior of the state space, only the Gaussian behaviour is present, e.g.\ \cite[Sect. 9]{Feller1951} for the WF model. 
Furthermore, note that in (c) the effect of mutation (from the fit allele to unfit alleles) is retained in the limit, whereas in (a) and (b) this effect is totally absent. 
In Section \ref{sect:stationary}, we show how this makes (c) a more powerful result at stationarity: as $t$ grows, the logistic trajectory gets closer to its equilibrium point on the boundary, and the effect of mutation becomes crucial to approximate the WF diffusion.

\subsubsection*{Derivation of the deterministic limit---proof of Proposition \ref{thm:diffusion_limits}(a)}
	
%Since $x_1= -\sum_{i=2}^dx_i$, we omit the first component, with a slight abuse of notation.

From \eqref{generatorWFdiff},
and by rewriting
    \begin{align*}
    s_i(\x)= 
    \sigma \frac{x_1}{2} (\delta_{1i} -  x_i )
    + \frac{1}{2}\sigma_i x_i (1-\delta_{1i} ) 
    -\frac{x_i}{2}\sum_{j=2}^d \sigma_j x_j  ,
    \end{align*}
it is straightforward to obtain  
	\begin{align*}
    &\frac{1}{\sigma}\mathcal{L}^{(\sigma)} 
    =
    %\mathcal{L}_{\xi}=
    \sum_{i=1}^d 
    \omega_i(\x)
    \frac{\partial}{\partial x_i }
    + \frac{1}{\sigma} \sum_{i=1}^d h_i(\x)\frac{\partial}{\partial x_i }
    + \frac{1}{2 \sigma}
	\sum_{i,j=1}^d 
	d_{ij}(\x)     \frac{\partial^2}{\partial x_i \partial x_j }
    ,
	\end{align*}
where 
	$
	\omega_i (\x)=  \frac{x_1}{2} (\delta_{1i} -  x_i ) 
	$ 
is defined in \eqref{eq:ODE_deterministic_trajectory}, and where the function
	\begin{align}
    \label{eq:fct_h}
    h_i(\x)=\mu_i(\x) 
    + \frac{1}{2}\sigma_i x_i (1-\delta_{1i} ) 
    -\frac{x_i}{2}\sum_{j=2}^d \sigma_j x_j  
	\end{align} 
does not depend on $\sigma$ and is bounded on $\Delta$.
This implies the convergence of the generators uniformly in $\Delta$, i.e.\  for all twice continuously differentiable functions $f$,
\begin{align*}
    &\lim_{\sigma\to\infty} \sup_{\x \in \Delta}
    \left| 
    \frac{1}{\sigma}\mathcal{L}^{(\sigma)} f(\x)
    -\sum_{i=1}^d 
    \omega_i(\x) \frac{\partial}{\partial x_i }f(\x)
    \right|
    %\xrightarrow[\sigma\to\infty]{} 
    %\mathcal{L}_{\xi}=
    =0.
	\end{align*}
Note that $\sum_{i=1}^d \omega_i(\x)\frac{\partial}{\partial x_i }$  is the generator of the deterministic trajectory $\xii$. Then,
assuming $\X\sth(0)\xrightarrow[]{d}\xii(0),$ as $\sigma \to\infty, $
classical arguments imply weak convergence \cite[Thm 6.1, Ch 1, and Thm 2.5, Ch 4]{Ethier1986}.
%which in this case yields almost sure convergence, since the limiting process is deterministic. 
\qed

\subsection{Gaussian fluctuations around the logistic trajectory}
\label{sect:diffusion_gaussian}

In this section we describe and  derive the Gaussian limit  (b) of Proposition \ref{thm:diffusion_limits}  through the convergence of infinitesimal generators and a martingale approach. 
First, we recall that analogous limits can be found in the literature for WF models, 
assuming that selection on the first allele is stronger than selection on the other alleles as well as on mutation and genetic drift; see the Appendix and e.g.\ \cite{Norman1972,Norman1975} and \cite{Nagylaki1986} which cover not only the asymptotic regime we consider but also more general large-parameter regimes.
However, the results of Norman, being rather general, require  additional calculations and verification of conditions in order to be explicated in the various cases of interest, making the interpretation not immediate. 

Let us now describe in detail the limiting Gaussian process $\U$, then, in the next section, we show how to derive it from the WF diffusion $\X\sth$. 
%To make the SDE \eqref{eq:SDE_Gaussian} more explicit, 	note that  
%$$ 
%[J\omega (\xii(t))]_{ij} = \frac{d}{dx_j}\omega_i(\xii(t))=
%\frac{1}{2}[\delta_{1i}\delta_{1j}-\delta_{1j}\xi_i(t)-\delta_{ij}\xi_1(t)],
%$$
%and thus
%\comm{\tiny{
        %   $[J\omega (\xii(t))U(t)]_i
        %   =\langle \nabla \omega_i (\xi(t)), \U(t) \rangle
        %   =     \sum_{j=1}^d 
        %    \frac{d}{dx_j}\omega_i(\xi(t)) U_j(t)
        %    =
        %    \frac{1}{2} \left\{
        %    \left[\delta_{1i} - \xi_i (t) \right] U_1(t) 
        %   - 
        %   \xi_1(t) U_i(t) \right\}
        %    $
        %} }
%\begin{align*}
%    J\omega (\xii(t))  \U(t) =
%    - \frac{1}{2} \left\{
%    U_1(t) \left[ \xii (t) - \e_1  \right]
%    +
%    \xi_1(t) \U(t) \right\}  .
%\end{align*}
	%\comm{\tiny{
			%    \begin{align*}
				%     dU(t)=
				%    -\frac{1}{2} \left\{
				%     U_1(t) \left[ \xii (t)  - \e_1 \right]
				%   + 
				%   \xi_1(t) \U(t) \right\}  dt
				%   + 
				%     D^{1/2}(\xii (t)) d\W(t).
				%    \end{align*}
			%}}
	%\comm{\footnotesize{Note: here there is a difference between the Jacobian considering all components, and the Jacobian considering  components $2, \dots, d$, thus the matrix $ J\boldsymbol{\omega} (\xii(t))$ is different, but, for $i\neq 1$, $[ J\boldsymbol{\omega} (\xii(t))U(t)]_i$ is the same for both Jacobians  }}\\
	It is straightforward, by using the textbook idea of variation of constants, e.g.\ \cite{baldi2017},  to derive an explicit expression for the solution of a linear SDE with time-dependent coefficients, in this case, 
	\begin{align*} 
		\U(t)= 
		e^{\Lambda(t)}\U(0)
		+ e^{\Lambda(t)} \int_0^t e^{-\Lambda(s)}D^{1/2}(\xii (s)) d\W(s),
	\end{align*}
	where $\Lambda(t)= \int_0^t J\boldsymbol{\omega} (\xii(s))ds$. This Gaussian process has respective mean and covariance functions
	\begin{align*}
		&m(t)=   e^{\Lambda(t)}\E{\U(0)},
		\\
		&C(t)=  \int_0^t e^{\Lambda(t)-\Lambda(s)}D(\xii (s)) e^{\Lambda(t)^\top-\Lambda(s)^\top} ds,
	\end{align*}
	which satisfy 
	\begin{align}
		\nonumber
		& \frac{d}{dt } m(t)= J\boldsymbol{\omega} (\xii(t)) m(t) ,
		\\ \label{eq:mean_cov_PDE}
		&  \frac{d}{dt }  C(t)= 
		D(\xii (t)) + J\boldsymbol{\omega} (\xii(t)) C(t) + C(t) J\boldsymbol{\omega} (\xii(t))^\top. 
	\end{align}
	
\subsubsection*{Derivation of the Gaussian limit---proof of Proposition \ref{thm:diffusion_limits}(b) }
	
Let $\Y\sth(t)=\X^{(\sigma)}(t/\sigma)$ and
$\U^{(\sigma)}(t)= b_\sigma[\Y^{(\sigma)}(t) - \xii(t) ]$ where $b_\sigma$ is a scaling factor yet to be determined. We derive  the limit $\U$ of the sequence $\U^{(\sigma)}$, as $\sigma\to\infty,$  and show that the appropriate choice of the scaling sequence is $b_\sigma=\sqrt{\sigma}$.  
First, we calculate the  mean and covariance of the infinitesimal increments  of $\Y\sth$. 
For $i=1,\dots,d,$
    \begin{align*}
    \lim_{\epsilon\to 0 }
    \frac{1}{\epsilon} \E{ Y_i\sth (t+ \epsilon) -x_i\mid \Y\sth(t)=\x }
    &= \frac{1}{\sigma}\lim_{\epsilon'\to 0 }\frac{1}{\epsilon'} \E{ X_i\sth (t'+ \epsilon') -x_i\mid \X\sth(t')=\x }
    \\ &= \frac{1}{\sigma} \left[ \mu_i(\x) + s_i(\x) \right]
    %\\& 
    = \omega_i  (\x) + \frac{1}{\sigma} h_i(\x)
    =: \omega_i \sth (\x)
    ,
    \end{align*}
where 
    $
    \omega_i (\x)=  \frac{x_1}{2} (\delta_{1i} -  x_i ) 
    $ 
is defined in \eqref{eq:ODE_deterministic_trajectory}, and 
$h_i(\x)$ defined in \eqref{eq:fct_h}
% $ h_i(\x)=\mu_i(\x) 
%+ \frac{1}{2}\sigma_i x_i (1-\delta_{1i} ) -\frac{x_i}{2}\sum_{j=2}^d \sigma_j x_j 
% $
does not depend on $\sigma$. 
We denote by $\boldsymbol{\omega}\sth$ the vector function with components $\omega_1\sth,\dots,\omega_d\sth$.
Furthermore, for $i,j=1,\dots,d,$
    \begin{align*}
    & \lim_{\epsilon\to 0 }
    \frac{1}{\epsilon} 
    \E{
        \left( Y_i\sth (t+ \epsilon) -x_i\right)\left( Y_j\sth (t+ \epsilon) -x_j\right)
        \mid \Y\sth(t)=\x }
    \\ 
    &\quad  =\frac{1}{\sigma} \lim_{\epsilon'\to 0 }\frac{1}{\epsilon'} 
    \E{ 
        \left(X_i\sth (t'+ \epsilon') -x_i\right)\left(X_j\sth (t'+ \epsilon') -x_j\right)
        \mid \X\sth(t')=\x }
    %\\ &
    = \frac{1}{\sigma} d_{ij}(\x) .
    \end{align*}
To derive the limit of $\U^{(\sigma)}$, we use a martingale approach based on rewriting 
    \begin{align}
    \nonumber
    \U^{(\sigma)}(t)= \ &
    b_\sigma[\Y^{(\sigma)}(0) - \xii(0) ] 
    %\\    & 
    +  b_\sigma \int_0^t \left[\boldsymbol{\omega}\sth(\Y\sth (s))-\boldsymbol{\omega} (\xii(s))\right]ds
    \\ \label{eq:martingale_approach}
    &  +  b_\sigma \left[ \Y\sth(t)-\Y\sth(0) - \int_0^t \boldsymbol{\omega}\sth(\Y\sth (s))ds \right],
    \end{align}
and on studying the limit as $\sigma\to\infty$ of each of the three terms on the right-hand side of \eqref{eq:martingale_approach} separately. To make this argument rigorous we aim to meet sufficient conditions for the limit of each term to exist, be sufficiently well behaved, and be possible to identify. A collection of such conditions is provided by Theorem 2.11 of \cite{kang2014}, which we summarise as follows (see also \cite{jenkins2015} for an application to the strong recombination limit).
\begin{enumerate}[(i),noitemsep]
\item $\Y\sth$ converges as $\sigma\to\infty$ to an identifiable, deterministic limit.
\item The process $\R\sth(t)$ satisfying 
    \begin{align*}
    d \R\sth(t)&=
    d\Y\sth(t)  - \boldsymbol{\omega}\sth(\Y\sth (t))dt 
    % \\ &=
    % \boldsymbol{\omega}\sth(\Y\sth(t))dt +  \frac{1}{\sqrt{\sigma}} D^{1/2}(\Y\sth (t)) d\W(t)- \boldsymbol{\omega}\sth(\Y\sth (t))dt, 
    \end{align*} 
    is a local martingale, and the matrix of covariations processes $[\Y\sth]_t$ with entries $[Y\sth_i,Y\sth_j]_t$ satisfies $[\Y\sth]_t \xrightarrow[]{d} 0_{d\times d}$ as $\sigma\to\infty$.
\item Contributions of $\mathcal{O}(b_\sigma^{-1})$ to the error $\boldsymbol{\omega}\sth - \boldsymbol{\omega}$ can be identified: there exists a continuous function $G_0:\Delta \to \mathbb{R}^d$ such that
\begin{equation*}
\lim_{\sigma\to\infty}\sup_{\x\in\Delta}\left|b_\sigma(\boldsymbol{\omega}\sth(\x) - \boldsymbol{\omega}(\x)) - G_0(\x)\right| = 0. 
\end{equation*}
\item The martingale central limit theorem applies to $b_\sigma \R\sth$, which for a continuous process is guaranteed by the condition that there exists a continuous $G:\Delta \to \mathbb{R}^{d\times d}$ such that for each $t>0$,
\begin{equation*}
b_\sigma^2[\Y\sth]_t - \int_0^t G(\Y\sth(s)) ds \xrightarrow[]{d} 0_{d\times d}, \qquad \sigma\to\infty.
\end{equation*}
\end{enumerate}
The above conditions are a simplification of those of \cite{kang2014}, which allows for the diffusion to be influenced by processes evolving on faster timescales. We gain further simplification from the fact that each process $\Y\sth$ shares the same state space $\Delta$, and that this space is compact.

We check each requirement in turn.
\begin{enumerate}[(i),noitemsep]
\item This follows from (a).
\item We can write
    \begin{align*}
    \R\sth(t) &=
    \Y\sth(t)-\Y\sth(0) - \int_0^t \boldsymbol{\omega}\sth(\Y\sth (s))ds 
    =
    \frac{1}{\sqrt{\sigma}} \int_0^t D^{1/2}(\Y\sth (s)) d\W(s).
    \end{align*}
    In other words, $\R\sth$ is a martingale by construction. A direct calculation shows that
\begin{equation}
\label{eq:covariations}
[\Y\sth]_t = \frac{1}{\sigma}\int_0^tD(\Y\sth(s))ds \xrightarrow[]{d} 0_{d\times d}, \qquad \sigma \to\infty,
\end{equation}
since the entries of $D$ are bounded on $\Delta$.
\item From the definitions of $\boldsymbol{\omega}\sth$ and $\boldsymbol{\omega}$ we find
\[
b_\sigma(\omega_i\sth(\x) - \omega_i(\x)) = \frac{b_\sigma}{\sigma}h_i(\x), \qquad i=1,\dots,d.
\]
Hence assuming $\frac{b_\sigma}{\sigma} \to 0$ as $\sigma\to\infty$ and recalling that $h_i$ is bounded on $\Delta$, this condition can be satisfied with the choice $G_0 = 0$.
\item From \eqref{eq:covariations} we get
\[
b_\sigma^2[\Y\sth]_t = \frac{b_\sigma^2}{\sigma}\int_0^t D(\Y\sth(s))ds,
\]
and so this condition is verified with choice $G=D$ provided we take $b_\sigma = \sqrt{\sigma}$.
\end{enumerate}

Under these conditions, Theorem 2.11 of \cite{kang2014} applies and in particular $\U\sth \xrightarrow[]{d} \U$, with the form of $\U$ identified by taking the limit of each of the three terms in \eqref{eq:martingale_approach}.

With the choice $b_\sigma = \sqrt{\sigma}$ we find that the first term of \eqref{eq:martingale_approach} satisfies $\U\sth(0)= b_\sigma[\Y^{(\sigma)}(0) - \xii(0) ]\xrightarrow[]{d} \U(0)$ as $\sigma\to \infty$, by our assumption. For the second term of \eqref{eq:martingale_approach}, first obtain
    % \begin{align*}
    % \omega_i\sth(\Y\sth (s))-\omega_i (\xii(s))
    % &=
    %\omega_i(\Y\sth (s))-\omega_i (\xii(s))  +\frac{1}{\sigma} h_i(\Y\sth (s))
    %\\ &=
    %-\frac{1}{2} \left\{  \left[Y_1\sth (s) Y_i\sth (s) - \xi_1(s)\xi_i(s)\right]
    %+ \delta_{1i} \left[Y_1\sth (s)  - \xi_1(s)\right] \right\}
    % + \frac{1}{\sigma} h_i(\Y\sth (s))
    %\\ &=
    %-\frac{1}{2} \left\{ 
    %\left[ Y_i\sth (s)- \delta_{1i} \right]\left[Y_1\sth (s)  - \xi_1(s)\right]
    %+
    %\xi_1(s) \left[ Y_i\sth (s) - \xi_i(s)\right]
    %\right\}
    % + \frac{1}{\sigma} h_i(\Y\sth (s)).
    %\end{align*}
    \begin{align*}
    & b_\sigma  \int_0^t \left[\omega_i\sth(\Y\sth (s))-\omega_i (\xii(s))\right]ds
    \\
    &  = b_\sigma  \int_0^t 
    \left[\omega_i(\Y\sth (s))-\omega_i (\xii(s))  +\frac{1}{\sigma} h_i(\Y\sth (s))\right]ds
    \\
    & =  -\frac{1}{2} 
    \int_0^t  \left\{
    \left[ Y_i\sth (s)- \delta_{1i} \right] b_\sigma \left[Y_1\sth (s)  - \xi_1(s)\right] 
    +
    \xi_1(s) b_\sigma \left[ Y_i\sth (s) - \xi_i(s)\right] \right\} ds
    \\
    &\quad +\frac{b_\sigma}{\sigma} \int_0^t h_i(\Y\sth (s))ds,
    \end{align*}  
%\com{(I could also use Taylor on the difference to obtain the gradient)}
which converges as $\sigma\to\infty$ to 
     \begin{align*}
    &  -\frac{1}{2}  
    \int_0^t \left\{
    \left[\xi_i (s)- \delta_{1i}\right] U_1(s)
    +  
    \xi_1(s) U_i(s) \right\} ds
    =
    \int_0^t 
    \langle \nabla \omega_i (\xi(s)), \U(s) \rangle ds
    =
    \int_0^t 
    [J\boldsymbol{\omega} (\xii(t))U(t)]_i  ds.
    \end{align*}
This limit corresponds to the drift of the SDE \eqref{eq:SDE_Gaussian}.
The third term of \eqref{eq:martingale_approach} can be rewritten as $b_\sigma \R\sth(t)$. Furthermore,
    \begin{align*}
    b_\sigma \R\sth(t)
    =
    \int_0^t D^{1/2}(\Y\sth (s)) d\W(s)
    \xrightarrow[]{ }  
    \int_0^t D^{1/2}(\xii (s)) d\W(s),
    \end{align*}
as $\sigma\to\infty$, which corresponds to the diffusion term of the SDE \eqref{eq:SDE_Gaussian}. \qed
		
\subsection{CBI fluctuations around the equilibrium point }

\label{sect:diffusion_CBI}
In this section we describe and derive the CBI limit (c) of Proposition \ref{thm:diffusion_limits} by analysing the convergence of infinitesimal generators. 
First, we point out that an analogous limit for WF models, albeit not phrased in terms of fluctuations around the equilibrium point, can be found in \cite{Ethier1988} and \cite{Nagylaki1990}, see also the Appendix for a sketch of the derivation.
This type of limit, unlike the Gaussian limit, seems to be less well known, 
%for example, it is not mentioned in \cite{Feder2014,lacerda2014}, and it is not common in the literature,
because it is a consequence of the equilibrium point being on the boundary.

Let us now  describe in detail the limiting CBI process $\Z$.
Each of the independent components $Z_2,\dots,Z_d$ can be seen as a Cox--Ingersoll--Ross process or as a special type of continuous-state branching process with immigration (CIR or CBI process, resp.). 
While CIR processes arise from mathematical finance \cite{CIR1985}, they correspond to a special type of CBI processes with immigration ($\alpha$-stable with $\alpha=2$), which instead arise from earlier studies that are more related to this work. 
In fact,  the study of continuous-state branching (CB) processes  was initiated in \cite{Feller1951} with the observation that the large-population limit of a Galton--Watson branching process is a diffusion, later named Feller's branching diffusion, a special case of CB process ($\alpha$-stable with $\alpha=2$).  
CBI processes, which generalise CB processes, were introduced in \cite{KawazuWatanabe1971} as approximations of branching processes with immigration, see e.g.\ \cite{li2020} for a review. 
It is not surprising that our asymptotic derivation for the WF  diffusion leads to this type of process, not only because the limit should naturally correspond to the limit of the corresponding  WF model in the analogous asymptotic regime,  
but also because intuitively under large selection, the frequencies of unfit alleles tend to behave (on their smaller scale) like CB processes with the addition of 
immigration which corresponds to mutation from the fit to the unfit allele.
If $P_{1i}=0$ for a certain $i\neq 1$, there is no immigration, i.e.\ no mutation, from the fit allele to the unfit allele $i$, and $Z_i$ is precisely a (subcritical) Feller branching diffusion.

The transition distributions of $Z_i$, $i=2,\dots,d$, are known explicitly, that is,
    \begin{align*}
    Z_i(t+s)\mid Z_i(s) 
    \sim \frac{1-e^{-\frac{1}{2}t}}{2}
    \mathcal{X}
    %\chi^2 \left( 2\theta P_{1i}, \frac{2e^{-\frac{1}{2}t}}{1-e^{-\frac{1}{2}t}} Z_i(s)\right),
    \quad t,s \geq 0, 
    \end{align*}
where 
%$\chi^2 \left( \nu, \lambda\right)$ 
$\mathcal{X}$ is distributed as a non-central chi-squared distribution with 
%$\nu$ 
$2\theta P_{1i}$
degrees of freedom and non-centrality parameter 
%$\lambda$
$\frac{2e^{-\frac{1}{2}t}}{1-e^{-\frac{1}{2}t}} Z_i(s)$. 
For the case $P_{1i}>0$, see e.g.\ \cite[eq.\ (17) and (18)]{CIR1985}; to compare, replace their $\kappa$ with $\frac{1}{2}$, $\theta$ with $\theta P_{1i}$ and $\sigma$ with 1. 
For the case $P_{1i}=0$, 
$\mathcal{X}$ can be considered as a non-central chi-squared random variable with zero degrees of freedom in the sense of \cite{siegel1979}. 
The distribution in this case was derived earlier \cite[Lemma 9]{FellerDiff1951} and corresponds to a Poisson-Gamma mixture with an atom at zero, as explained in \cite[eq.\ (7)]{burden2024}, to compare, simply replace $\alpha$ in \cite{burden2024} with $-\frac{1}{2}$. 
		
The genetic picture of the CBI process is as follows. Because selection is very strong, the population overwhelmingly comprises the fitter allele, and thus $\X$ will be close to $\e_1$. Occasionally, but only through mutation from the ubiquitous fit type to a rare unfit type, less fit alleles will be re-introduced to the population. As their frequencies are rare in a population almost wholly composed of allele 1, their frequencies do not `see' the upper limit of $1$ and instead (zoomed in on this spatial scale) evolve on $[0,\infty)$ according to the very well-known branching process approximation to a population model of fixed size. In the terminology of the CBI process, immigration here occurs at rate $\theta P_{1i}/2$ and the intrinsic growth parameter is $-1/2$. In this limit we do not see, for example, mutation between less fit types, or the relative selective strengths of the less fit types, play any role.
		
It is interesting to note that the boundary classification criteria (\cite{FellerDiff1951}; see also \cite[Example 8.2, p235--237]{IkedaWatanabe}) show that, if the upwards drift due to mutation from the fit type to the unfit type $i$ is not large enough, specifically $\theta P_{1i}<1$, then $Z_i$ can reach $0$ in finite time, which is thus an accessible endpoint. In particular, if $\theta P_{1i}=0$, then $Z_i$ is eventually absorbed at $0$. Otherwise, if $\theta P_{1i}\geq 1$, $Z_i$ always remains positive in finite time, i.e.\ $0$ is an inaccessible endpoint.  
Furthermore, exact path sampling of the processes $Z_i, i=2,\dots,d$, $P_{1i}>0$, can be done by means of the exact simulation algorithm constructed in \cite{makarov2010},
which is based on a scale and time transformation  reducing a CIR process to a squared Bessel process and on exact simulation of the latter.

\subsubsection*{Derivation of the CBI limit---proof of Proposition \ref{thm:diffusion_limits}(c) }

In order to study convergence of the sequence 
$ \Z\sth (t)= c_\sigma \left[ \X^{(\sigma)}( t/\sigma) - \e_1 \right] $, where $c_\sigma$ is a scaling factor yet to be determined, we derive its infinitesimal generator $\tilde{\mathcal{L}}\sth$ by a change of variables and a scaling of the generator $\mathcal{L}\sth$ of $\X\sth$. 
The next calculations also identify the appropriate scaling to be $c_\sigma=\sigma$, unlike the usual $\sigma^{1/2}$ one would expect from the central limit theorem. 
Divide $\mathcal{L}\sth$ by $\sigma$ to scale time, and 
let
$
\z= c_\sigma (\x-\e_{1}),
$   
to obtain
    \begin{align*}
    \tilde{\mathcal{L}}\sth=
    \frac{1}{\sigma}\left\{
    \frac{1}{2}
    \sum_{i,j=1}^d 
    d_{ij}\left(\e_1+c_\sigma^{-1}\z\right)     
    c_\sigma^2 \frac{\partial^2}{\partial z_i \partial z_j }
    + \sum_{i=1}^d \left[\mu_i\left(\e_1+c_\sigma^{-1}\z\right) + s_i\left(\e_1+c_\sigma^{-1}\z\right)  \right]
    c_\sigma\frac{\partial}{\partial z_i } 
    \right\}.
    \end{align*}
We now omit the first component, since $\sum_{i=1}^d x_i= 1$ and thus $\sum_{i=1}^d z_i= 0$, and calculate the diffusion terms
    \begin{align*}
    \frac{1}{2 \sigma}
    \sum_{i,j=2}^d 
    d_{ij}\left(\e_1+c_\sigma^{-1}\z\right)     
    c_\sigma^2 \frac{\partial^2}{\partial z_i \partial z_j }  
    %\comm{=
        %\frac{1}{2 \sigma}
        %\sum_{i,j=2}^d 
        %c_\sigma^{-1} z_i(\delta_{ij}-c_\sigma^{-1}z_j)   
        %c_\sigma^2 \frac{\partial^2}{\partial z_i \partial z_j } }
    =
    \frac{ c_\sigma }{\sigma} \frac{ 1}{2}
    \sum_{i=2}^d z_i  \frac{\partial^2}{\partial z_i^2 }
    -
    \frac{ 1}{\sigma} \frac{ 1}{2}
    \sum_{i,j=2}^d  z_i z_j \frac{\partial^2}{\partial z_i \partial z_j } ;
    \end{align*}
the mutation drift terms  
    \begin{align*}
    \frac{1}{\sigma}
    \sum_{i=2}^d 
    \mu_i\left(\e_1+c_\sigma^{-1}\z\right)
    c_\sigma\frac{\partial}{\partial z_i } 
    %&=
    %\frac{1}{\sigma}
    %\sum_{i=2}^d 
    %\frac{\theta}{2}\left(
    %P_{1i}+ c_\sigma^{-1} \sum_{j=1}^d P_{ji} z_j -  c_\sigma^{-1}z_i 
    %\right) 
    %c_\sigma\frac{\partial}{\partial z_i } 
    %\\ &
    =
    \frac{c_\sigma}{\sigma} 
    \sum_{i=2}^d 
    \frac{\theta}{2} P_{1i} \frac{\partial}{\partial z_i } 
    + 
    \frac{1}{\sigma} 
    \sum_{i=2}^d  \frac{\theta}{2} \left[ \sum_{j=2}^d (P_{ji} -P_{1i}) z_j - z_i \right] 
    \frac{\partial}{\partial z_i } ;
    \end{align*}
and the selection drift terms
    \begin{align*}
    \frac{1}{\sigma}  \sum_{i=2}^d  s_i\left(\e_1+c_\sigma^{-1}\z\right) 
    c_\sigma\frac{\partial}{\partial z_i } 
    =&
    - \sum_{i=2}^d \frac{1}{2}  z_i \frac{\partial}{\partial z_i }
    + 
    \frac{1}{\sigma} 
    \sum_{i=2}^d \frac{1}{2} \sigma_i  z_i \frac{\partial}{\partial z_i }
    + \frac{1}{c_\sigma}
    \sum_{i=2}^d \frac{1}{2} \left(\sum_{j=2}^d z_j \right)z_i \frac{\partial}{\partial z_i }
    \\ &- \frac{1}{\sigma c_\sigma}
    \sum_{i=2}^d \frac{1}{2} \left(\sum_{j=2}^d \sigma_j z_j \right)z_i \frac{\partial}{\partial z_i } .
    \end{align*}
It is now clear that for convergence to a non-trivial limit as $\sigma\to\infty$ we should choose $c_\sigma=\sigma$. 
%By an abuse of notation but without loss of information we will continue to denote by $\tilde{\mathcal{L}}\sth$ the generator acting on functions of $(Z_2(t),\dots,Z_d(t))$. 
Therefore, continuing to omit the first component  %with a slight abuse of notation but without loss of information, 
we find, 
%the generator of $ \sigma \left[ \X^{(\sigma)}( t/\sigma) - \e_1 \right] $ is 
    \begin{align*}
    \tilde{\mathcal{L}}\sth=
    \tilde{\mathcal{L}}_0 
    +\frac{1}{\sigma} \tilde{\mathcal{L}}_1 
    +\frac{1}{\sigma^2} \tilde{\mathcal{L}}_2,    
    \end{align*}
where
    \begin{align}
    \label{eq:generator_CBI}
    \nonumber
    &\tilde{\mathcal{L}}_0
    =
    \frac{ 1}{2}
    \sum_{i=2}^d z_i  \frac{\partial^2}{\partial z_i^2 }   
    +
    \sum_{i=2}^d 
    \frac{1}{2} (\theta P_{1i} -z_i) \frac{\partial}{\partial z_i };
    \\ \nonumber
    &\tilde{\mathcal{L}}_1 
    =
    -\frac{ 1}{2}
    \sum_{i,j=2}^d  z_i z_j \frac{\partial^2}{\partial z_i \partial z_j } 
    +     \sum_{i=2}^d  \left[ \frac{\theta}{2} \sum_{j=2}^d (P_{ji} -P_{1i}) z_j 
    +\frac{1}{2}z_i  \left(
    -\theta + \sigma_i +\sum_{j=2}^d z_j
    \right) 
    \right] 
    \frac{\partial}{\partial z_i };
    \\
    &\tilde{\mathcal{L}}_2
    =
    -      \sum_{i=2}^d \frac{1}{2} \left(\sum_{j=2}^d \sigma_j z_j \right)z_i \frac{\partial}{\partial z_i } .
    \end{align}   
The operator $\tilde{\mathcal{L}}_0$ is itself an infinitesimal generator and is precisely the generator of the limiting process $\Z$. 
Note that $\tilde{\mathcal{L}}_1 $ alone is not an infinitesimal generator because its second-order terms correspond to a negative semidefinite matrix.  

The calculations above prove uniform convergence of generators, i.e.\ 
for any twice continuously differentiable function $f$  with compact support
    \begin{align*}
    &\lim_{\sigma\to\infty} \sup_{z_2,\dots,z_d\in \mathbb{R}_{\geq 0} }
    \left| 
    \tilde{\mathcal{L}}\sth f(\z)
    -\tilde{\mathcal{L}}_0 f(\z)
    \right|
    =
    \lim_{\sigma\to\infty} \sup_{z_2,\dots,z_d\in \mathbb{R}_{\geq 0} }
    \left| 
    \frac{1}{\sigma} \tilde{\mathcal{L}}_1 f(\z)
    +\frac{1}{\sigma^2} \tilde{\mathcal{L}}_2 f(\z)
    \right|
    %\xrightarrow[\sigma\to\infty]{} 
    %\mathcal{L}_{\xi}=
    =0, 
    \end{align*}
which in turn implies weak convergence \cite[Prop. 3.3 and Thm 6.1, Ch 1, and Thm 2.5, Ch 4]{Ethier1986},  assuming the convergence of the initial distributions. \qed

Reading off from the form of $\tilde{\mathcal{L}}_0$ in the proof of Proposition \ref{thm:diffusion_limits}(c) we obtain the approximation
    \[
    \X\sth(t/\sigma)=\e_1 +\frac{1}{\sigma} \Z\sth (t/\sigma) \approx \e_1 +\frac{1}{\sigma} \Z(t),
    \]
where $\Z$ is the multi-type CBI process satisfying \eqref{eq:fluctuationsSDE}.

\section{Stationary distributions}
\label{sect:stationary}

The stationary distribution of the WF diffusion plays an important role in the expression of the sampling probability, which is discussed in Section \ref{sect:sampling_prob} and summarised by equation \eqref{eq:sampling_diffusion}. 

Under parent-independent mutation, i.e.\ $P_{ij}=Q_j,i,j=1,\dots,d,$ for $Q_j>0,j=1,\dots,d,$ it is well known  \cite{barbour2000,Wright1949} that the stationary density is the weighted Dirichlet density 
\begin{align}
	\label{eq:stationary_density}
	%\X^{(\sigma)}(\infty)
	%\sim 
	\pi(\x) &=
	\prod_{i=1}^d e^{\sigma_i x_i}  x_i^{\theta Q_i -1}
	\left(
	\int_\Delta
	\prod_{i=1}^d e^{\sigma_i x'_i}  {x'_i}^{\theta Q_i -1 }
	d \x' 
	\right)^{-1}.
\end{align} 
%\com{(Wakeley2008 eq (1) has $-\sigma$ instead, by setting $x_1=1-..$ and simplifying with the normalising constant)}
This is the  only case in which the stationary density is explicitly known. In general, the stationary distribution is unknown explicitly, but assuming that the mutation probability matrix $P$ is irreducible guarantees the existence of a unique stationary density with respect to Lebesgue measure \cite[Thm 3.1 and Thm 3.2]{shiga1981}. This is crucial in the remainder of this paper, thus $P$ is assumed to be irreducible from now on. Note that, in the PIM case, the irreducibility assumption corresponds to $Q_j>0,j=1,\dots,d$.

While in the PIM case the explicit density can be analysed asymptotically (see the end of this section), the idea for the general case is to use the limiting diffusions of the previous section to obtain an approximation of the stationary distribution. It turns out that only one of the two limiting diffusions is useful to this goal.

It is known for WF models that the Gaussian approximation can also be extended to stationarity in those regimes in which the equilibrium point of the deterministic trajectory is not on the boundary, see e.g.\ \cite{Nagylaki1986,Norman1972,Norman1974,Norman1975b}. 
However, in our study, the equilibrium point is on the boundary  and  the Gaussian approximation fails at stationarity.

\begin{remark}[degenerate Gaussian stationary distribution]
The Gaussian process $\U$, defined by \eqref{eq:SDE_Gaussian}, degenerates to the zero process as $t\to\infty$.
In fact,  
    $ [J\boldsymbol{\omega} (\xii(\infty))]_{ij}=
    \frac{1}{2}[\delta_{1i}\delta_{1j}-\delta_{1j}\xi_i(\infty)-\delta_{ij}\xi_1(\infty)]
    =
    - \frac{1}{2}  \delta_{ij}
    $, 
and  
    $ 
    d_{ij}(\xii(\infty))
    =\xi_i(\infty)(\delta_{ij}-\xi_j(\infty))
    =\delta_{1i}(\delta_{ij}-\delta_{1j})=0
    $,  
that is, 
    $$
    J\boldsymbol{\omega} (\xii(\infty))=- \frac{1}{2} I_d 
    \quad \text{ and } \quad  
    D(\xii(\infty))=0_{d\times d}. 
    $$
Therefore,   by \eqref{eq:mean_cov_PDE}, the mean and covariance functions of $\U(t)$, as $t\to\infty$, converge to 
    $$
    m(\infty)
    =0_d 
    \quad \text{ and } \quad
    C(\infty)= 0_{d\times d}.
    $$
	
\end{remark}

\noindent 
The degenerate Gaussian stationary distribution trivially tells us that $\X\sth(\infty) \approx \xii(\infty) +\frac{1}{\sqrt{\sigma}}\U(\infty)= \e_1$, which indicates that the $\sqrt{\sigma}-$scaling  is not appropriate to obtain a non-trivial distributional limit at stationarity.
On the contrary, the CBI process $\Z$ does not become degenerate as $t\to\infty$ and thus provides an alternative approximation to replace the classical Gaussian approximation.

\begin{remark}[Gamma stationary distribution]
If $P_{1i}>0, i=2,\dots,d,$ 
the stationary distribution of the CBI limit $\Z$, defined by \eqref{eq:SDE_CBI},  
has independent components, for $i=2,\dots, d$, with 
    $$
    Z_i(\infty)\sim \text{Gamma}(\theta P_{1i},1),
    $$
that is a Gamma distribution with shape parameter $\theta P_{1i}$ and rate parameter $1$, see e.g.\ \cite{CIR1985}.
 % i.e.\ the stationary density function is
%     $%\begin{equation*}
%     %\psi_i(z_i)=
%     \frac{1}{\Gamma(\theta P_{1i})} 
%     z_i ^{\theta P_{1i}-1}e^{-z_i},
%     $ %\end{equation*}
% see e.g.\ \cite{CIR1985}.
Furthermore, 
$$-Z_1(\infty)\sim \text{Gamma}(\theta(1-P_{11}),1).$$
The Gamma stationary distribution provides the non-trivial approximation  
$ \X^{(\sigma)}(\infty) \approx 
e_1 +\frac{1}{\sigma} \Z(\infty)$
which we  use to guide the scaling of the sampling probability in the next section. 
If instead, $P_{1i}=0$ for a certain $i\neq 1$, then $Z_i(\infty)=0$ (see e.g.\ \cite[Sect. 2.3]{burden2024}), as mentioned in section \ref{sect:diffusion_CBI}. 
\end{remark}

Finally, in the PIM case, 
knowing now the correct scaling, it is straightforward to check directly that the explicitly-known density function of 
$\sigma (\X^{(\sigma)}(\infty) -\e_1) $
is 
(from \eqref{eq:stationary_density}, changing variables $\z= \sigma (\x-\e_1 )$, with Jacobian determinant $\sigma^{-(d-1)}$, and recalling that $\sum_{i=1}^d z_i = 0$) ,
    \begin{align*}
    &\sigma^{-(d-1)} \pi(\e_1+ \sigma^{-1}\z)  \\
    &\quad = \frac{
        \sigma^{-(d-1)}
        e^{z_1 + \sigma}  ( \sigma^{-1} z_1 +1 )^{\theta Q_1 -1}
        \prod_{i=2}^d e^{\sigma_i \sigma^{-1} z_i}  ( \sigma^{-1} z_i)^{\theta Q_i -1}
    }{
        \int_\Delta
        \sigma^{-(d-1)}
        e^{z'_1 + \sigma}  ( \sigma^{-1} z'_1 +1 )^{\theta Q_1 -1}
        \prod_{i=2}^d e^{\sigma_i \sigma^{-1} z'_i}  ( \sigma^{-1} z'_i)^{\theta Q_i -1}
        d z'_2\dots dz'_d 
    }
    \\ 
    &\quad=
    \frac{
        ( \sigma^{-1} z_1 +1 )^{\theta Q_1 -1}
        \prod_{i=2}^d e^{\sigma_i \sigma^{-1} z_i}  \prod_{i=2}^d e^{-z_i }  z_i^{\theta Q_i -1}
    }{
        \int_\Delta
        ( \sigma^{-1} z'_1 +1 )^{\theta Q_1 -1}
        \prod_{i=2}^d e^{\sigma_i \sigma^{-1} z'_i} \prod_{i=2}^d e^{-z'_i }   (z_i')^{\theta Q_i -1}
        d z'_2\dots dz'_d  
    }\\
    &\quad \to \prod_{i=2}^d 
    \frac{1}{\Gamma(\theta Q_i)}
    e^{- z_i}  z_i^{\theta Q_i -1}, \quad  \sigma\to\infty,
    \end{align*}
the density of a collection of independent $\text{Gamma}(\theta Q_i,1)$ distributed random variables, $i=2,\dots, d$, as expected. 

\section{Sampling probabilities}
\label{sect:sampling_prob}

A sampling configuration is denoted by a vector $\n \in \mathbb{N}^d \setminus \{\boldsymbol{0}\}$, with $n_i$ being the number of individuals of type $i$ in the sample.  
The probability of obtaining the unordered configuration $\n$ is denoted by $p^{(\sigma)}(\n) $,
and the number of individuals in the configuration $\n$ is denoted by $\norm{\n}:=n_1+\dots+n_d$.  
Because of exchangeability, each of the  $\binom{\norm{\n}}{\n}$ ordered configurations that is consistent with a certain unordered configuration $\n$ is obtained with the same probability $q^{(\sigma)}(\n) $, and $p^{(\sigma)}(\n)= \binom{\norm{\n}}{\n} q^{(\sigma)}(\n) $. 
In a  population at stationarity, the probability $q^{(\sigma)}(\n) $, which we call  sampling probability, can be written as
\begin{align}
	\label{eq:sampling_diffusion}
	q^{(\sigma)}(\n)=\E{F(\X^{(\sigma)}(\infty),\n )},
	\quad \quad 
	\text{with }
	F(\x,\n) := \prod_{i=1}^d x_i^{n_i } ,
\end{align}
where  $\X^{(\sigma)}(\infty)$ is distributed according to the stationary distribution of the WF diffusion $\{ \X^{(\sigma)}(t)\}_{t\geq 0}$.
While the natural domain of $q\sth$ remains  $\mathbb{N}^d \setminus \{\boldsymbol{0}\}$, we adopt the standard convention of defining $q\sth(\boldsymbol{0})=1$ and  of defining probabilities of configurations with negative components as null, i.e.\ $q\sth(\n)=0$, for $\n\notin \mathbb{N}^d$.  These definitions allow for more compact expressions 
by implicitly restricting the recursion formulas below and in Section \ref{sect:sampling_prob_full_expansion}.

Let us assume for now $\sigma_2,\dots,\sigma_d=0$. This assumption is justified and discussed in details in Remark \ref{rmk:assumption_selection}. It is made here and then again in Sections \ref{sect:sampling_prob_full_expansion}, \ref{sect:genealogy}, and \ref{sect:duality}.
An equivalent way of describing $q^{(\sigma)}(\n)$, for $\n\in\mathbb{N}^d\setminus \{\boldsymbol{0}\} $, is through the recursion formula 
%{\tiny \comm{
		%   \begin{align*}
			%     & \frac{\norm{n}}{2} (\norm{n}-1+\theta +\sigma) q\sth(n)
			%     \\ & \quad = 
			%     \sum_{i=1}^d
			%    \frac{n_i(n_i-1)}{2}q\sth(n -e_i)
			%     +
			%    \sum_{i,j=1}^d
			%    n_i \frac{\theta P_{ji}}{2}q\sth(n -e_i+e_j)
			%    +
			%    \sum_{i,j=1}^d
			%    n_j \frac{\sigma}{2}q\sth(n  +e_i)
			%    \left[I\{j\in\mathcal{F}\}+I\{i\in\mathcal{U}\}\right]
			%    \\ & \quad = 
			%     \sum_{i=1}^d
			%    \frac{n_i(n_i-1)}{2}q\sth(n -e_i) 
			%    +
			%    \sum_{i,j=1}^d
			%    n_i \frac{\theta P_{ji}}{2}q\sth(n -e_i+e_j)
			%    +
			%    \sum_{i=1}^d
			%    n_1 \frac{\sigma}{2}q\sth(n  +e_i)
			%    +\sum_{i=2}^d \sum_{j=1}^d
			%    n_j \frac{\sigma}{2}q\sth(n  +e_i)
			%    \end{align*}
		%}}
    \begin{align}
    \nonumber
    & \{\norm{\n} (\norm{\n}-1+\theta ) 
    +
    (\norm{\n}-n_1 )\sigma \} q\sth(\n)
    \\  \label{eq:sampling_recursion}
    & \quad = 
    \sum_{i=1}^d
    n_i(n_i-1)q\sth(\n -\e_i)
    +
    \sum_{i,j=1}^d
    n_i \theta P_{ji}q\sth(\n -\e_i+\e_j)
    + \sum_{i=2}^d
    \norm{\n} \sigma q\sth(\n  +\e_i),
    \end{align}
which was originally derived for the two-allele case in  \cite{Krone1997} and for the multi-allele case in \cite{fearnhead2002,fearnhead2003}.
This recursion expresses the probability $q^{(\sigma)}(\n)$ by conditioning on the possible most recent event in the ancestral process that led to configuration $\n$: coalescences of type $i=1,\dots,d$ (represented by $\n-\e_i$), mutations from type $j=1,\dots,d$, to type $i=1,\dots,d$ (represented by $\n-\e_i+\e_j$), and selective events between an incoming lineage of type $i=2,\dots,d$ and a parental continuing lineage (represented by $\n+\e_i$). 
Selective events between a continuing lineage of type $i=1,\dots,d$ and a parental fit incoming lineage can be omitted because they do not lead to a change in configuration. Further details on this point, as well as on the underlying ancestral process, its rates, and the meaning of `incoming' and `continuing' branches, are postponed to Section \ref{sect:genealogy}. 
The coefficient on the left-hand side of \eqref{eq:sampling_recursion} is proportional to the total rate of events. In particular, 
the total coalescence rate is $\frac{1}{2}\norm{\n} (\norm{\n}-1) $, 
the total mutation rate is $\frac{1}{2}\theta \norm{\n} $, 
and the total selection rate is $\frac{1}{2}\sigma \norm{\n}$ to which the rate $\frac{1}{2}\sigma n_1$  of omitted selective events can be subtracted. 

A useful expression to be paired with \eqref{eq:sampling_recursion} is the consistency condition
\begin{align}
	\label{eq:to_reduce_recursion}
	\sum_{i=1}^d
	q\sth(\n +\e_i) = q\sth (\n),
\end{align}
for $\n\in \mathbb{N}^d$. This holds true because $q\sth(\n +\e_i)= \pi(i\mid \n)q\sth(\n)$, with $\pi(i\mid \n)$ being the probability of sampling type $i$ after having sampled the configuration $\n$. 
In the  PIM case, 
%when  $\norm{\n}>1$,  
expression \eqref{eq:to_reduce_recursion} allows us to simplify the recursion \eqref{eq:sampling_recursion}, for $\n\in \mathbb{N}^d\setminus\{\boldsymbol{0}\}$,  to 
% since $\sum_{j=1}^d q\sth(n -e_i+e_j)=q\sth(n - e_i)$,
    \begin{align*}
    \nonumber
    & \{\norm{\n} (\norm{\n}-1+\theta ) 
    +
    (\norm{\n}-n_1 )\sigma \} q\sth(\n)
    \\ & \quad =
    \sum_{i=1}^d
    n_i(n_i-1+\theta Q_i)q\sth(\n -\e_i) 
    + \norm{\n} \sigma \sum_{i=2}^d
    q\sth(\n  +\e_i).
    \end{align*}
In general, except for the PIM case, the sampling probabilities do not have an explicit expression. In fact,   \eqref{eq:sampling_diffusion} is implicit because of the   unknown stationary density and  the recursion \eqref{eq:sampling_recursion} is not usable in practice \cite{stephens2003}. 
The lack of an explicit expression for the sampling probability $q^{(\sigma)}(\n) $,
defined by \eqref{eq:sampling_diffusion} or equivalently \eqref{eq:sampling_recursion}, constitutes the main challenge for its  asymptotic analysis   which is the goal of this section. In particular, in Section \ref{sect:sampling_prob_full_expansion} we derive a full asymptotic expansion for $q^{(\sigma)}(\n) $ in inverse powers of $\sigma$, 
being guided for the choice of the correct scaling by the Gamma approximation presented in Section \ref{sect:sampling_prob_gamma}. First, in Section \ref{sect:sampling_prob_2alleles} we treat explicitly the special case of two alleles.

\subsection{Two alleles and parent independent mutation}
\label{sect:sampling_prob_2alleles}

In the PIM case, knowing that the stationary density  of the WF diffusion is the weighted Dirichlet density in \eqref{eq:stationary_density}, the sampling probability can be written more explicitly as
\begin{equation}
	\label{eq:PIM_sampling}
	q^{(\sigma)}(\n)
	= 
	\frac{
		\int_\Delta
		\prod_{i=1}^d e^{\sigma_i x_i}  x_i^{\theta Q_i -1 + n_i}
		d \x 
	}{
		\int_\Delta
		\prod_{i=1}^d e^{\sigma_i x_i}  x_i^{\theta Q_i -1 }
		d \x 
	} .
\end{equation}
In the two allele case, i.e.\ $d=2$, the mutation model can always be written in a way to satisfy the parent-independence assumption; see e.g.\ \cite{griffiths2008} for a detailed explanation. 
Thus, similarly to the calculations for ratios of analogous sampling probabilities in \cite{wakeley2008}, we obtain
\begin{align}
	q^{(\sigma)}(\n)
	& =
	\frac{
		\int_0^1
		e^{(\sigma_2-\sigma_1) x_2} (1-x_2)^{\theta Q_1-1+ n_1} x_2^{ \theta Q_2-1+ n_2} 
		d x_2 
	}{
		\int_0^1
		e^{(\sigma_2-\sigma_1) x_2} (1-x_2)^{\theta Q_1-1} x_2^{ \theta Q_2-1} 
		d x_2 
	}
	\nonumber\\& =
	\frac{\Gamma(n_1+ \theta Q_1)}{\Gamma( \theta Q_1)}
	\frac{\Gamma(n_2+ \theta Q_2)}{\Gamma( \theta Q_2)}
	\frac{\Gamma(\theta)}{\Gamma( \norm{\n}+\theta )}
	\frac{_1F_1[n_2+ \theta Q_2, \norm{\n}+\theta, \sigma_2-\sigma_1 ]}{_1 F_1[\theta Q_2, \theta, \sigma_2-\sigma_1 ]}, \label{eq:1F1solution}
\end{align}
where $_1F_1$ is Kummer's function, i.e.\ the confluent hypergeometric function with the integral representation
\begin{align*}
	_1F_1[a,b,z]&=
	\frac{\Gamma(b)}{\Gamma(a)\Gamma(b-a)}\int_0^1 e^{zy}y^{a-1}(1-y)^{b-a-1}dy,  
\end{align*}
for $a,b\geq 0$. It is known that, as  $z\to - \infty$, 
\begin{align*}
	_1F_1[a,b,z]=
	\frac{\Gamma(b)}{\Gamma(b-a)} 
	\frac{1}{(-z)^a}
	\sum_{k=0}^K \frac{1}{k!} (a)_k(a-b+1)_k 
	\frac{1}{(-z)^k}
	+ \mathcal{O}\left( \frac{1}{(-z)^{a+K+1}}\right) ,
\end{align*}
with $(a)_k=a(a+1)\cdots (a+k-1)=\frac{\Gamma(a+k)}{\Gamma(a)}$ being the rising factorial \cite[Ch. 6, Sect. 13]{bateman1953}.
%for properties of the function $ _1F_1$. 
Therefore,
\begin{align*}
	q^{(\sigma)}(\n)
	& =
	\frac{\Gamma(n_2+ \theta Q_2)}{\Gamma( \theta Q_2)}
	\frac{1}{\sigma_1^{n_2}}
	\frac{
		\sum_{k=0}^K \frac{1}{k!} (n_2+\theta Q_2)_k(1-\theta Q_1-n_1)_k 
		\frac{1}{\sigma_1^k}
		+ \mathcal{O}\left( \frac{1}{\sigma_1^{K+1}}\right) 
	}{
		\sum_{k=0}^K \frac{1}{k!} (\theta Q_2)_k(1-\theta Q_1)_k 
		\frac{1}{\sigma_1^k}
		+ \mathcal{O}\left( \frac{1}{\sigma_1^{K+1}}\right) 
	}
	\\& =
	\frac{1}{\sigma_1^{n_2}}
	\left\{\sum_{k=0}^K
	\tilde{q}_k(\n)  \frac{1}{\sigma_1^k}
	+ \mathcal{O}\left( \frac{1}{\sigma_1^{K+1}}\right)
	\right\} ,
	% \\& =
	%\frac{\Gamma(n_2+ \theta Q_2)}{\Gamma( \theta Q_2)}\frac{1}{(\sigma_1-\sigma_2)^{n_2}}
	%+\mathcal{O}\left(\frac{1}{(\sigma_1-\sigma_2)^{n_2+1}}\right)
\end{align*}
where the coefficients $\tilde{q}_k(\n)$ are given by the following. Letting
$ r_k(\n)=\frac{1}{k!} (n_2+\theta Q_2)_k(1-\theta Q_1-n_1)_k$, and 
$ s_k=\frac{1}{k!} (\theta Q_2)_k(1-\theta Q_1)_k$,
\begin{align}
	\tilde{q}_0(\n)
	&= \frac{\Gamma(n_2+ \theta Q_2)}{\Gamma( \theta Q_2)}
	,\notag\\
	\tilde{q}_1(\n)
	&= \tilde{q}_0(\n)
	[r_1(\n)-s_1],
	%\comm{\quad =
		%\tilde{q}_0(\n)\left[ (n_2+\theta Q_2)(1-n_1-\theta %Q_1) - \theta Q_2(1-\theta Q_1)\right]}
	\notag\\
	\tilde{q}_2(\n)
	&=\tilde{q}_0(\n)
	[r_2(\n)-s_2-s_1[r_1(\n)-s_1]]
	=\tilde{q}_0(\n)
	[r_2(\n)-s_2] -s_1\tilde{q}_1(\n)
	,\notag\\
	\vdots
	\notag\\
	\tilde{q}_k(\n)
	&=\tilde{q}_0(\n)
	\left[r_k(\n)-s_k\right] -
	\sum_{ m=1 }^{k-1} s_{k-m} \tilde{q}_{m}(\n).
	% \comm{\quad \quad  = \tilde{q}_0(\n) r_k(\n) -
		%\sum_{ m=0 }^{k-1} s_{k-m} \tilde{q}_m(\n) } 
	\label{eq:q-expansion-2allele}
\end{align}
We omit the straightforward but lengthy derivation of the coefficients above, and instead refer to Proposition \ref{thm:coefficient_PIM} in Section \ref{sect:sampling_prob_full_expansion} for a proof of a generalised formula based on a different approach. 
The expression for $ \tilde{q}_0(\n)$ in \eqref{eq:q-expansion-2allele} agrees with eq.\ (9b) of \cite{fan2024} who found the leading term ($k=0$).

In the multi-allele PIM case, an asymptotic analysis of the sampling probabilities based on the ratio of integrals appearing in \eqref{eq:PIM_sampling} seems to be possible, albeit  more cumbersome, whereas in the general case of possibly parent-dependent mutations, no explicit expression for the stationary density, and consequently for the sampling probabilities, is available. Therefore, below we take a different approach to the asymptotic analysis that does not rely on expression \eqref{eq:PIM_sampling} and thus works also for the general case of possibly parent-dependent mutations.

\subsection{Gamma approximation}
\label{sect:sampling_prob_gamma}

%In the general multi-allele case of parent dependent mutation,
%in order to derive an asymptotic approximation of  the sampling probabilities
%we need a different approach from the one of Subsection \ref{sect:sampling_prob_2alleles}, 
%because of the lack of explicit expressions in the general case. 

In this section, we present an approach for the asymptotic analysis of the sampling probability that does not rely on explicit expressions, which are not available in the general multi-allele case of parent-dependent mutation,
and is based on a \emph{Gamma approximation}, rather than the classical \emph{Gaussian approximation} whose stationary distribution is degenerate as shown in Section \ref{sect:stationary}. This approach describes the leading term (and in some cases, an additional non-null term) of the asymptotic expansion of the sampling probability and thus it indicates the correct scaling for the full asymptotic expansion, to be studied in the next section. 

We know from Section \ref{sect:stationary} that  $\X^{(\sigma)}(\infty)$ can be  approximated by $\e_1+\frac{1}{\sigma}\Z(\infty) $.  
Since we are interested in approximating the sampling probability $q^{(\sigma)}(\n)$ under 
$\X^{(\sigma)}(\infty)$,  we compute a different  sampling probability,  under $\e_1+\frac{1}{\sigma}\Z(\infty) $, as stated in the following. 

\begin{proposition}
	\label{thm:gamma_sampling}
	Let $q_\Gamma^{(\sigma)}(\n)$ be the sampling probability defined by  
	\begin{align*}
		q_\Gamma^{(\sigma)}(\n)
		&=
		\E{F\left(\e_1+\frac{1}{\sigma}\Z(\infty),\n \right)}.
	\end{align*}
	Then,
	\begin{align*}
		q_\Gamma^{(\sigma)}(\n)
		= 
		\frac{1}{\sigma^{\norm{\n}-n_1}}
		\sum_{k=0}^{n_1}
		\frac{1}{\sigma^{k}}
		\gamma_k(\n),
	\end{align*}
	where
	\begin{align*}
		\gamma_k(\n)=
		(-1)^{k} \sum_{k_2+\dots + k_d=k}
		\binom{n_1}{n_1-k \ k_2 \dots k_d}
		\prod_{i=2}^d \frac{\Gamma(\theta P_{1i}+n_i+k_i)}{\Gamma(\theta P_{1i})}.
	\end{align*}	
\end{proposition}
If $P_{1i}=0$, for some $i=2,\dots,d,$ then $q_\Gamma \sth (\n)=0$ for those $\n$ with $n_i>0$, since $Z_{i}(\infty)=0$. Note that, while this case is not of central interest, it is still covered by the formulas of the theorem above by simply interpreting a $\text{Gamma}(0,1) $ distribution as point mass at zero, and by adopting the natural convention that $\Gamma(z)/\Gamma(0)=0$ for $z>0$, and $\Gamma(0)/\Gamma(0)=1$.

\begin{proof}
By using the multinomial theorem, the independence of $Z_2,(\infty),\dots, Z_d(\infty)$, the known expression for moments of Gamma random variables, and by some rearranging, we obtain 
    \begin{align*}
    q_\Gamma^{(\sigma)}(\n)
    &
    = 
    \E{
        \left(1+\frac{1}{\sigma}Z_1(\infty)\right)^{n_1}
        \prod_{i=2}^d \left(\frac{1}{\sigma}Z_i(\infty)\right)^{n_i}
    }
    \\ & = 
    \E{
        \left(1-\frac{1}{\sigma}\sum_{i=2}^d Z_i(\infty)\right)^{n_1}
        \prod_{i=2}^d \left(\frac{1}{\sigma}Z_i(\infty)\right)^{n_i}
    }
    \\ & = 
    \E{
        \sum_{k_1+\dots +k_d=n_1} \binom{n_1}{k_1 \dots k_d}
        \prod_{i=2}^d \left(-\frac{1}{\sigma}Z_i(\infty)\right)^{k_i}
        \prod_{i=2}^d \left(\frac{1}{\sigma}Z_i(\infty)\right)^{n_i}
    }
    \\ & = 
    \sum_{k_1+\dots +k_d=n_1} 
    \frac{1}{\sigma^{\norm{\n}-k_1}}
    (-1)^{n_1-k_1} \binom{n_1}{k_1 \dots k_d}
    \E{
        \prod_{i=2}^d Z_i(\infty)^{n_i+k_i}
    }
    \\ & = 
    \sum_{k_1=0}^{n_1}
    \frac{1}{\sigma^{\norm{\n}-k_1}}
    (-1)^{n_1-k_1} \sum_{k_2+\dots +k_d=n_1-k_1}
    \binom{n_1}{k_1 \dots k_d}
    \prod_{i=2}^d \frac{\Gamma(\theta P_{1i}+n_i+k_i)}{\Gamma(\theta P_{1i})}
    \\ & = 
    \frac{1}{\sigma^{\norm{\n}-n_1}}
    \sum_{k=0}^{n_1}
    \frac{1}{\sigma^{k}}
    (-1)^{k} \sum_{k_2+\dots +k_d=k}
    \binom{n_1}{n_1-k \ k_2 \dots k_d}
    \prod_{i=2}^d \frac{\Gamma(\theta P_{1i}+n_i+k_i)}{\Gamma(\theta P_{1i})}.
    \end{align*}
\end{proof}

The Gamma approximation $q_\Gamma^{(\sigma)}(\n)$, being based on the two-term approximation $\e_1+\frac{1}{\sigma}\Z(\infty)$,  captures the first two terms of the asymptotic expansion of the sampling probability
$q\sth (\n)$, as confirmed by the results of the next section. However, %it is crucial to notice that 
the first of these terms is equal to zero when $\norm{\n}\neq n_1$, that is, when there is at least one copy of an unfit allele in the sample, since the term of lowest possible order $\mathcal{O}(1/\sigma^{0})$ in Proposition 4.1 vanishes. 
More precisely, the leading term of Proposition \ref{thm:gamma_sampling} is (for $k=n_1-k_1=k_2=\dots=k_d=0$) 
$$
\frac{1}{\sigma^{\norm{\n}-n_1}} \gamma_0(\n),
\quad \text{ with } \quad
\gamma_0(\n) =
\prod_{i=2}^d \frac{\Gamma(\theta P_{1i}+n_i)}{\Gamma(\theta P_{1i})}.
$$
This corresponds to expression \eqref{eq:zero_coefficient} derived in the next section. 
If $\norm{\n}\neq n_1$,  the leading term above
%, implicitly preceded by a null term, 
constitutes all of the `useful' information given by the Gamma approximation to the sampling probability;
that is, higher order terms beyond this leading order in the expansion for $q_\Gamma^{(\sigma)}(\n)$ do not agree with those of $q\sth (\n)$, as we show later.
On the other hand, if $\norm{\n} = n_1$, i.e.\ $\n=n_1\e_1$, then the Gamma approximation agrees with the expansion for $q^{(\sigma)}(\n)$ to two orders with non-zero coefficients---the leading term above, which is equal to $1$ in this case,  followed by  (for $k=1$ in Proposition \ref{thm:gamma_sampling})
$$
\frac{1}{\sigma^{\norm{\n}-n_1+1}} \gamma_1(n_1\e_1),
$$
with
\begin{align*}
	\gamma_1(\n) =  
	- \sum_{j=2}^d n_1 
	\frac{\Gamma(\theta P_{1j}+n_j+1)}{\Gamma(\theta P_{1j}+n_j)} 
	\prod_{i=2}^d \frac{\Gamma(\theta P_{1i}+n_i)}{\Gamma(\theta P_{1i})}
	=
	- \gamma_0(\n) n_1 [\theta (1-P_{11}) + \norm{\n}-n_1].
\end{align*}
As we show later, $\gamma_1(\n)$ agrees with the corresponding term in the asymptotic expansion of $q^{(\sigma)}(\n)$ when $\n=n_1\e_1$ (see expression \eqref{eq:one_coefficient_allfit} in the next section), though not in general.

To summarise, the leading term  $\gamma_0(\n)$ of the Gamma approximation always coincides with the leading term of the asymptotic expansion of the sampling probability $q\sth(\n)$, while the term  $\gamma_1(\n)$  coincides with the second term of the asymptotic expansion only when $\n=n_1\e_1$. 
This will become more clear in the next section, and it can also be directly verified from the explicit expressions in \eqref{eq:q-expansion-2allele}, for the case of two alleles and parent-independent mutation. 

\subsection{Full asymptotic expansion for general mutation models}
\label{sect:sampling_prob_full_expansion}

We denote by $\tilde{q}_k(\n)$, $k=0,1,2,\dots$, the coefficients of the following asymptotic expansion of the sampling probability 
\begin{align}
	\label{eq:asymptotic_expansion}
	q^{(\sigma)}(\n)
	=
	\frac{1}{\sigma^{\norm{\n}-n_1}}
	\left\{\sum_{k=0}^K
	\tilde{q}_k(\n)  \frac{1}{\sigma^k}
	+ \mathcal{O}\left( \frac{1}{\sigma^{K+1}}\right)
	\right\} 
\end{align}
for any $K\geq 0 $. This key equation is an asymptotic, rather than a convergent, series expression for $q\sth(\n)$; see Remark \ref{rmk:power-series}. Like $q\sth$ itself, the coefficients of \eqref{eq:asymptotic_expansion} are defined for convenience outside the natural domain $\mathbb{N}^d\setminus\{\boldsymbol{0}\}$. It follows from the definition of $q\sth$ that $\tilde{q}_k(\boldsymbol{0})=\ind{k=0}$, and, for $k=0,1,2,\dots, \n\notin\mathbb{N}^d,$ $\tilde{q}_k(\n)=0$.
%Note that, in Subsection \ref{sect:},
%we have already described all of the coefficients of the expansion above in the case of two alleles with parent independent mutation, 
%while, in Subsection \ref{sect:}, we have described the $0^{th}$ coefficient (and the $1^{st}$ when no unfit alleles are in the sample) in the general case.
In this section, we focus on characterising the coefficients of the expansion above by means of the recursion formula \eqref{eq:sampling_recursion} for a general mutation model. In Section \ref{sect:genealogy}, sampling probabilities are viewed in the framework of the ASG, and we use the expression above to analyse the rates of the ASG asymptotically.

\begin{lemma}
\label{thm:tricks_recursion}
Let $\tilde{q}_k(\n), k=0,1,2,\dots$, be defined by \eqref{eq:asymptotic_expansion}. Then, for $\n\in \mathbb{N}^d $, 
%    \begin{align*}
    %    \tilde{q}_0(\n+\e_1)
    %    =
    %    \tilde{q}_0(\n)
    %    \end{align*}
%and, for $k\geq 1$,
    \begin{align*}
    \tilde{q}_k(\n+\e_1)  
    +
    \ind{k\geq 1}
    \sum_{j=2}^d
    \tilde{q}_{k-1}(\n+\e_j) 
    =
    \tilde{q}_k(\n).
    \end{align*}
	
\begin{proof}
Inserting the asymptotic expansion \eqref{eq:asymptotic_expansion} in \eqref{eq:to_reduce_recursion} yields
    \begin{align*}
    & \frac{1}{\sigma^{\norm{\n}-n_1}} 
    \left\{\sum_{k=0}^K
    \tilde{q}_k(\n+\e_1)  \frac{1}{\sigma^k}
    + \mathcal{O}\left( \frac{1}{\sigma^{K+1}}\right)
    \right\} 
    \\
     &\quad +
    \sum_{j=2}^d
    \frac{1}{\sigma^{\norm{\n}-n_1+1}}
    \left\{\sum_{k=0}^K
    \tilde{q}_k(\n+\e_j)  \frac{1}{\sigma^k}
    + \mathcal{O}\left( \frac{1}{\sigma^{K+1}}\right)
    \right\}  
    \\ & \quad \quad = 
    \frac{1}{\sigma^{\norm{\n}-n_1}}
    \left\{\sum_{k=0}^K
    \tilde{q}_k(\n)  \frac{1}{\sigma^k}
    + \mathcal{O}\left( \frac{1}{\sigma^{K+1}}\right)
    \right\} .
    \end{align*}
Multiply both sides of the equation above by $\sigma^{\norm{\n}-n_1}$ and adjust the indices of the sums so that each sum contains $\frac{1}{\sigma^{k}}$, so that 
    %{\tiny\comm{
        %\begin{align*}
        %\sum_{k=0}^K
        % \tilde{q}_k(\n+\e_1)  \frac{1}{\sigma^k}
        % + \mathcal{O}\left( \frac{1}{\sigma^{K+1}}\right)
        %+
        % \sum_{j=2}^d
        %\sum_{k=1}^{K+1}
        % \tilde{q}_{k-1}(\n+\e_j)  \frac{1}{\sigma^k}
        % + \mathcal{O}\left( \frac{1}{\sigma^{K+2}}\right) 
        % = 
        %\sum_{k=0}^K
        % \tilde{q}_k(\n)  \frac{1}{\sigma^k}
        % + \mathcal{O}\left( \frac{1}{\sigma^{K+1}}\right)
        % .
        %\end{align*}
        %}}
    \begin{align*}
    \sum_{k=0}^K
    \frac{1}{\sigma^k}
    \tilde{q}_k(\n+\e_1) 
    +
    \sum_{k=1}^{K}
    \frac{1}{\sigma^k}
    \sum_{j=2}^d
    \tilde{q}_{k-1}(\n+\e_j) 
    =
    \sum_{k=0}^K
    \frac{1}{\sigma^k}
    \tilde{q}_k(\n) 
    + \mathcal{O}\left( \frac{1}{\sigma^{K+1}}\right).
    \end{align*}
The equality above must hold for the coefficients of each order, thus concluding the proof. Note that the argument is also valid for $\n=\boldsymbol{0}$ since equation \eqref{eq:to_reduce_recursion} holds in this case. 
\end{proof}
\end{lemma}  

\begin{remark}
\label{rmk:assumption_selection}
From the asymptotic analysis so far, it has emerged that none of the results depends on the selection parameters $\sigma_2,\dots,\sigma_d$ of the unfit alleles: these results include the diffusion limits of Section \ref{sect:diffusions} and consequently the Gamma approximation of Section \ref{sect:sampling_prob_gamma} which concern first-order asymptotics for general mutation mechanisms, and the full asymptotic expansion of Section \ref{sect:sampling_prob_2alleles} which includes also higher orders for the two-allele (PIM) case. 
This suggests that $\sigma_2,\dots,\sigma_d$ do not influence the upcoming asymptotic analysis, at least in the first order.
We thus assume from now on $\sigma_2,\dots,\sigma_d = 0$, which significantly facilitates the exposition, e.g.\ \cite{fearnhead2002} also assumes selection parameters of unfit alleles to be null.

Section \ref{sect:genealogy} and \ref{sect:duality} are concerned only with first-order results, thus we expect the same asymptotics to hold also for general $\sigma_2,\dots,\sigma_d$. 
A clue is provided by the first order in Theorem  \ref{thm:characterisation_asyptotic_coefficients} below, corresponding to the Gamma approximation of Section \ref{sect:sampling_prob_gamma} which \emph{is} valid for general $\sigma_2,\dots,\sigma_d$, confirming our expectation.
The question on whether higher orders in Theorem  \ref{thm:characterisation_asyptotic_coefficients} depend on  $\sigma_2,\dots,\sigma_d $  remains open in the general case (although we know they do not in the two-allele case).
\end{remark}

\begin{theorem}
	\label{thm:characterisation_asyptotic_coefficients}
	%The coefficient $\tilde{q}_0$ of the asymptotic expansion of $q\sth(n)$ solves 
	% \begin{align*}
		%     (\norm{n}-n_1 ) \tilde{q}_0(n) 
		%    =
		%     \sum_{i=2}^d
		%    n_i(n_i-1)
		%    \tilde{q}_0(n-e_i) +
		%    \sum_{i=2}^d
		%    n_i \theta P_{1i}
		%    \tilde{q}_0(n -e_i+e_1)  
		%    \end{align*}
	%or, equivalently,
	%    \begin{align*}
		%     (\norm{n}-n_1 ) \tilde{q}_0(n) 
		%     =
		%     \sum_{i=2}^d
		%    n_i(n_i-1+\theta P_{1i})
		%    \tilde{q}_0(n-e_i) 
		%    \end{align*}
	%The coefficient $\tilde{q}_1(n)$ of the asymptotic expansion of $q\sth(n)$ solves
	%   \begin{align*}
		%     (\norm{n}-n_1 ) \tilde{q}_1(n) 
		%      =&
		%    \sum_{i=2}^d
		%    n_i(n_i-1)
		%    \tilde{q}_1(n-e_i) 
		%    +
		%    \sum_{i=2}^d
		%    n_i \theta P_{1i}
		%    \tilde{q}_1(n -e_i+e_1) 
		%    \\  & + 
		%    [n_1 (n_1-1+\theta P_{11} )- \norm{n}  (\norm{n}-1+\theta ) ] \tilde{q}_{0}(n)
		%    \\  & 
		%    +
		%    \sum_{i,j=2}^d
		%    n_i \theta P_{ji}
		%    \tilde{q}_{0}(n-e_i+e_j) 
		%   +
		%    \norm{n}
		%    \sum_{i=2}^d
		%    \tilde{q}_{0}(n+e_i) 
		%    \end{align*}
	
	%For $k\geq 2$, the coefficient $\tilde{q}_k(n)$ of the asymptotic expansion of $q\sth(n)$ solves
	The coefficients $\tilde{q}_k(\n)$, $k=0,1,2,\dots$, of the asymptotic expansion  \eqref{eq:asymptotic_expansion} are characterised by the following set of boundary conditions and recursive equations.
	\\
	\begin{enumerate}[label=(\Roman*)]    
		\item (Boundary---fit allele) 
		\label{thm:boundary_fit}
		\begin{align*}
			\tilde{q}_k(\e_1) 
			&=
			\ind{k=0} - 
			\ind{k\geq 1}
			\sum_{i=2}^d
			\tilde{q}_{k-1}(\e_i).
		\end{align*}
		
		\item (Boundary---unfit alleles)
		\label{thm:boundary_unfit}
		For $ i=2,\dots,d$,
		\begin{align*}
			\tilde{q}_k(\e_i) 
			&=
			\ind{k=0} \theta P_{1i}  
			+ 
			\ind{k\geq 1} \bigg\{
			- \theta \sum_{j=2}^d 
			[\delta_{ij}+P_{1i}-P_{ji}]
			\tilde{q}_{k-1}(\e_j) 
			+ 
			\sum_{j=2}^d
			\tilde{q}_{k-1}(\e_i+\e_j) 
			\bigg\}.
		\end{align*} 
		
		\item (Recursion---fit allele only)
		\label{thm:recursion_fit}
		For $n_1>1$, 
		%%% ALTERNATIVE FORMULA
		%\begin{align*}
		% \tilde{q}_k(n_1\e_1) 
		%  = &
		% \tilde{q}_k((n_1-1)\e_1)  
		% \\  &
		% +
		% \ind{k\geq 1}   \bigg\{
		% (n_1-2) 
		%\tilde{q}_{k-1}((n_1-2)\e_1) 
		%+
		%[ \theta P_{11} -   (n_1-2+\theta ) ] \tilde{q}_{k-1}((n_1-1)\e_1)
		%\bigg\}
		%\\  &
		%+  \ind{k\geq 2}  
		%\sum_{j=2}^d
		%n_1 \theta P_{j1}
		%\tilde{q}_{k-2}((n_1-2)\e_1+\e_j) 
		%\end{align*}
		\begin{align*}
			\tilde{q}_k(n_1\e_1)
			= {}& 
			\tilde{q}_k((n_1-1)\e_1)
			- \ind{k\geq 1} 
			\theta (1-P_{11}) \tilde{q}_{k-1}((n_1-1)\e_1)
			\\  & {}+  \ind{k\geq 2} 
			\sum_{j=2}^d
			(n_1-2+\theta P_{j1})
			\tilde{q}_{k-2}((n_1-2)\e_1+\e_j).
		\end{align*}   
		\item (Recursion---general)
		\label{thm:recursion_general}
		For all $\n$ with $\norm{\n}>1$, 
		\begin{align*}
			(\norm{\n}-n_1 ) \tilde{q}_k(\n) 
			={}&
			\sum_{i=2}^d
			n_i(n_i-1 +\theta P_{1i})
			\tilde{q}_k(\n-\e_i)  
			\\  &  
			{}+ \ind{k\geq 1}  \bigg\{
			[n_1(n_1-1+ \theta P_{11}) - \norm{\n}  (\norm{\n}-1+\theta ) ] \tilde{q}_{k-1}(\n)
			\\  & {}+  
			\sum_{i,j=2}^d
			n_i \theta (P_{ji}- P_{1i})
			\tilde{q}_{k-1}(\n-\e_i+\e_j) 
			+ 
			\norm{\n}
			\sum_{i=2}^d
			\tilde{q}_{k-1}(\n+\e_i) 
			\bigg\}
			\\  & {}+  \ind{k\geq 2} 
			\sum_{j=2}^d
			n_1 (n_1-1+\theta P_{j1})
			\tilde{q}_{k-2}(\n-\e_1+\e_j).
		\end{align*}
	\end{enumerate}
\end{theorem}

\begin{proof}
	Insert the asymptotic expansion \eqref{eq:asymptotic_expansion} in the equation $\sum_{i=1}^d q\sth (\e_i)=1$ and rearrange to obtain
	\begin{align*}
		1= \sum_{i=1}^d q\sth (\e_i)
		&= 
		\sum_{k=0}^K
		\tilde{q}_k(\e_1)  \frac{1}{\sigma^k}
		+ \mathcal{O}\left( \frac{1}{\sigma^{K+1}}\right)
		+ \sum_{i=2}^d
		\frac{1}{\sigma}
		\left\{\sum_{k=0}^K
		\tilde{q}_k(\e_i)  \frac{1}{\sigma^k}
		+ \mathcal{O}\left( \frac{1}{\sigma^{K+1}}\right)
		\right\} 
		\\ &= 
		\sum_{k=0}^K
		\frac{1}{\sigma^k} \tilde{q}_k(\e_1) 
		+ 
		\sum_{k=1}^K \frac{1}{\sigma^k}
		\sum_{i=2}^d
		\tilde{q}_{k-1}(\e_i)  
		+ \mathcal{O}\left( \frac{1}{\sigma^{K+1}}\right).
	\end{align*}
	Matching the coefficients of each order yields \ref{thm:boundary_fit}.    
	
	Next, inserting the asymptotic expansion \eqref{eq:asymptotic_expansion} in \eqref{eq:sampling_recursion}  gives
	\begin{align*}
		& \{\norm{\n}  (\norm{\n}-1+\theta ) 
		+
		(\norm{\n}-n_1 )\sigma \} 
		\frac{1}{\sigma^{\norm{\n}-n_1}}
		\left\{\sum_{k=0}^K
		\tilde{q}_k(\n)  \frac{1}{\sigma^k}
		+ \mathcal{O}\left( \frac{1}{\sigma^{K+1}}\right)
		\right\} 
		\\  & \quad = 
		n_1(n_1-1)
		\frac{1}{\sigma^{\norm{\n}-n_1}}
		\left\{\sum_{k=0}^K
		\tilde{q}_k(\n -\e_1)  \frac{1}{\sigma^k}
		+ \mathcal{O}\left( \frac{1}{\sigma^{K+1}}\right)
		\right\} 
		\\  & \quad\quad +
		\sum_{i=2}^d
		n_i(n_i-1)
		\frac{1}{\sigma^{\norm{\n}-n_1-1}}
		\left\{\sum_{k=0}^K
		\tilde{q}_k(\n-\e_i)  \frac{1}{\sigma^k}
		+ 
		\mathcal{O}\left( \frac{1}{\sigma^{K+1}}\right)
		\right\} 
		\\  & \quad\quad +
		n_1 \theta P_{11}
		\frac{1}{\sigma^{\norm{\n}-n_1}}
		\left\{\sum_{k=0}^K
		\tilde{q}_k(\n)  \frac{1}{\sigma^k}
		+ \mathcal{O}\left( \frac{1}{\sigma^{K+1}}\right)
		\right\} 
		\\  & \quad\quad +
		\sum_{i=2}^d
		n_i \theta P_{1i}
		\frac{1}{\sigma^{\norm{\n}-n_1-1}}
		\left\{\sum_{k=0}^K
		\tilde{q}_k(\n -\e_i+\e_1)  \frac{1}{\sigma^k}
		+ \mathcal{O}\left( \frac{1}{\sigma^{K+1}}\right)
		\right\} 
		\\  & \quad\quad +
		\sum_{j=2}^d
		n_1 \theta P_{j1}
		\frac{1}{\sigma^{\norm{\n}-n_1+1}}
		\left\{\sum_{k=0}^K
		\tilde{q}_k(\n-\e_1+\e_j)  \frac{1}{\sigma^k}
		+ \mathcal{O}\left( \frac{1}{\sigma^{K+1}}\right)
		\right\} 
		\\  & \quad\quad +
		\sum_{i,j=2}^d
		n_i \theta P_{ji}
		\frac{1}{\sigma^{\norm{\n}-n_1}}
		\left\{\sum_{k=0}^K
		\tilde{q}_k(\n-\e_i+\e_j)  \frac{1}{\sigma^k}
		+ \mathcal{O}\left( \frac{1}{\sigma^{K+1}}\right)
		\right\} 
		\\  & \quad\quad +
		\sum_{i=2}^d
		\norm{\n} \sigma 
		\frac{1}{\sigma^{\norm{\n}-n_1+1}}
		\left\{\sum_{k=0}^K
		\tilde{q}_k(\n+\e_i)  \frac{1}{\sigma^k}
		+ \mathcal{O}\left( \frac{1}{\sigma^{K+1}}\right)
		\right\} .
	\end{align*}
	Multiply both sides of the equation above for $\sigma^{\norm{\n}-n_1-1}$ and adjust the indices of the sums so that each sum contains $\frac{1}{\sigma^{k}}$, so that 
	%{\tiny\comm{
			%    \begin{align*}
				%     \{\norm{n} & (\norm{n}-1+\theta ) -n_1 \theta P_{11} \}
				%     \sum_{k=0}^K
				%     \tilde{q}_k(n)  \frac{1}{\sigma^{k+1}}
				%     + \mathcal{O}\left( \frac{1}{\sigma^{K+2}}\right)
				%     +
				%   (\norm{n}-n_1 )
				%    \sum_{k=0}^K
				%     \tilde{q}_k(n)  \frac{1}{\sigma^{k}}
				%     + \mathcal{O}\left( \frac{1}{\sigma^{K+1}}\right)
				%    \\ = &
				%     n_1(n_1-1)
				%     \sum_{k=0}^K
				%     \tilde{q}_k(n -e_1)  \frac{1}{\sigma^{k+1}}
				%     + \mathcal{O}\left( \frac{1}{\sigma^{K+2}}\right)
				%\\ &
				%    +
				%     \sum_{i=2}^d
				%    n_i(n_i-1)
				%    \sum_{k=0}^K
				%     \tilde{q}_k(n-e_i)  \frac{1}{\sigma^{k}}
				%     + 
				%     \mathcal{O}\left( \frac{1}{\sigma^{K+1}}\right)
				%    \\ & +
				%    \sum_{i=2}^d
				%    n_i \theta P_{1i}
				%   \sum_{k=0}^K
				%     \tilde{q}_k(n -e_i+e_1)  \frac{1}{\sigma^{k}}
				%     + \mathcal{O}\left( \frac{1}{\sigma^{K+1}}\right)
				%    %\\ & 
				%    +
				%    \sum_{j=2}^d
				%    n_1 \theta P_{j1}
				%    \sum_{k=0}^K
				%     \tilde{q}_k(n-e_1+e_j)  \frac{1}{\sigma^{k+2}}
				%     + \mathcal{O}\left( \frac{1}{\sigma^{K+3}}\right)
				%    \\ & +
				%    \sum_{i,j=2}^d
				%    n_i \theta P_{ji}
				%   \sum_{k=0}^K
				%     \tilde{q}_k(n-e_i+e_j)  \frac{1}{\sigma^{k+1}}
				%     + \mathcal{O}\left( \frac{1}{\sigma^{K+2}}\right)
				%\\  & 
				%   + 
				%    \sum_{i=2}^d
				%    \norm{n} 
				%    \sum_{k=0}^K
				%     \tilde{q}_k(n+e_i)  \frac{1}{\sigma^{k+1}}
				%     + \mathcal{O}\left( \frac{1}{\sigma^{K+2}}\right)
				%    \end{align*}
			%    }}
	\begin{align*}
		& \sum_{k=0}^K  \frac{1}{\sigma^{k}}
		(\norm{\n}-n_1 ) \tilde{q}_k(\n) 
		+
		\sum_{k=1}^{K} \frac{1}{\sigma^{k}}
		[\norm{\n}  (\norm{\n}-1+\theta ) -n_1 \theta P_{11}] \tilde{q}_{k-1}(\n)  
		\\ & \quad = 
		\sum_{k=0}^K  \frac{1}{\sigma^{k}}
		\sum_{i=2}^d
		n_i(n_i-1)
		\tilde{q}_k(\n-\e_i) 
		%\\ & \quad \quad 
		+
		\sum_{k=0}^K \frac{1}{\sigma^{k}}
		\sum_{i=2}^d
		n_i \theta P_{1i}
		\tilde{q}_k(\n -\e_i+\e_1)  
		\\ & \quad \quad + 
		\sum_{k=1}^K\frac{1}{\sigma^{k}}
		n_1(n_1-1)
		\tilde{q}_{k-1}(\n -\e_1)  
		%\\ & \quad \quad 
		+
		\sum_{k=1}^K  \frac{1}{\sigma^{k}}
		\sum_{i,j=2}^d
		n_i \theta P_{ji}
		\tilde{q}_{k-1}(\n-\e_i+\e_j) 
		\\ & \quad \quad +
		\sum_{k=1}^K  \frac{1}{\sigma^{k}}
		\sum_{i=2}^d
		\norm{\n} 
		\tilde{q}_{k-1}(\n+\e_i) 
		%\\ & \quad \quad 
		+
		\sum_{k=2}^K  \frac{1}{\sigma^{k}}
		\sum_{j=2}^d
		n_1 \theta P_{j1}
		\tilde{q}_{k-2}(\n-\e_1+\e_j) 
		+ \mathcal{O}\left( \frac{1}{\sigma^{K+1}}\right) .
	\end{align*}
	The equality above must hold for the coefficients of each order. That is, for each $k=0,1,2,\dots$,  
	\begin{align}
		\label{eq:recursion_basis}
		\nonumber
		(\norm{\n}-n_1 ) \tilde{q}_k(\n) 
		=&
		\sum_{i=2}^d
		n_i(n_i-1)
		\tilde{q}_k(\n-\e_i) 
		+
		\sum_{i=2}^d
		n_i \theta P_{1i}
		\tilde{q}_k(\n -\e_i+\e_1)  
		\\ \nonumber & 
		+ 
		\ind{k\geq 1} \bigg\{
		n_1(n_1-1)
		\tilde{q}_{k-1}(\n -\e_1) 
		+ 
		[n_1 \theta P_{11} - \norm{\n}  (\norm{\n}-1+\theta ) ] \tilde{q}_{k-1}(\n) 
		\\ \nonumber & +  
		\sum_{i,j=2}^d
		n_i \theta P_{ji}
		\tilde{q}_{k-1}(\n-\e_i+\e_j) 
		+ 
		\norm{\n}
		\sum_{i=2}^d
		\tilde{q}_{k-1}(\n+\e_i) 
		\bigg\}
		\\  & +  \ind{k\geq 2}  
		\sum_{j=2}^d
		n_1 \theta P_{j1}
		\tilde{q}_{k-2}(\n-\e_1+\e_j) .
	\end{align}
	Equation \eqref{eq:recursion_basis} constitutes the core of the theorem and it is used to prove \ref{thm:boundary_unfit}, \ref{thm:recursion_fit}, and \ref{thm:recursion_general}, as follows. 
	
	For $i=2,\dots,d$, apply \eqref{eq:recursion_basis} to  $\n=\e_i$ and appeal to \ref{thm:boundary_fit} to obtain 
	\begin{align*}
		\tilde{q}_k(\e_i) 
		&=
		\theta P_{1i} 
		\tilde{q}_k(\e_1)  
		+ 
		\ind{k\geq 1} \bigg\{
		- \theta  \tilde{q}_{k-1}(\e_i) 
		+  
		\sum_{j=2}^d
		\theta P_{ji}
		\tilde{q}_{k-1}(\e_j) 
		+ 
		\sum_{j=2}^d
		\tilde{q}_{k-1}(\e_i+\e_j) 
		\bigg\}
		\\  &=
		\theta P_{1i}  \ind{k=0}
		\\ &\quad + 
		\ind{k\geq 1} \bigg\{
		-\theta P_{1i}
		\sum_{j=2}^d
		\tilde{q}_{k-1}(\e_j) 
		- \theta  \tilde{q}_{k-1}(\e_i) 
		+  
		\sum_{j=2}^d
		\theta P_{ji}
		\tilde{q}_{k-1}(\e_j) 
		+ 
		\sum_{j=2}^d
		\tilde{q}_{k-1}(\e_i+\e_j) 
		\bigg\},
	\end{align*}   
	which proves \ref{thm:boundary_unfit}. 
	
	For $\n$ with $\norm{\n}>1$, apply Lemma \ref{thm:tricks_recursion} to  obtain
	\begin{align*}
		\tilde{q}_k(\n-\e_i+\e_1)  
		&=
		\tilde{q}_k(\n-\e_i)  
		-\ind{k\geq 1}
		\sum_{j=2}^d
		\tilde{q}_{k-1}(\n-\e_i+\e_j), \\
		%    \end{align*}
	%and   
	%    \begin{align*}
		\tilde{q}_k(\n-\e_1)  
		&=
		\tilde{q}_k(\n)  
		+
		\ind{k\geq 1}
		\sum_{j=2}^d
		\tilde{q}_{k-1}(\n-\e_1+\e_j), 
	\end{align*}
	which applied to \eqref{eq:recursion_basis} yields \ref{thm:recursion_general}.
	
	Finally, for $\norm{\n}=n_1>1$, i.e.\ $\n=n_1\e_1$, \ref{thm:recursion_general} is still valid but does not give an explicit formula. 
	Instead, apply \ref{thm:recursion_general} to $(n_1-1)\e_1$ to obtain
	\begin{align}
		0
		={}& 
		\ind{k\geq 1} 
		[(n_1-1)(n_1-2+ \theta P_{11}) -(n_1-1)  (n_1-2+\theta ) ] \tilde{q}_{k-1}((n_1-1)\e_1)
		\label{eq:partIII-recursion}\\  &  
		{}+ 
		(n_1-1)
		\ind{k\geq 1} 
		\sum_{i=2}^d
		\tilde{q}_{k-1}((n_1-1)\e_1+\e_i) 
		\\ & + 
		  \ind{k\geq 2} 
		\sum_{j=2}^d
		(n_1-1) (n_1-2+\theta P_{j1})
		\tilde{q}_{k-2}((n_1-2)\e_1+\e_j).\notag
	\end{align}   
	Noting that, by Lemma \ref{thm:tricks_recursion},
	\begin{align*}
		\ind{k\geq 1} 
		\sum_{i=2}^d
		\tilde{q}_{k-1}((n_1-1)\e_1+\e_i) 
		=
		\tilde{q}_{k}((n_1-1)\e_1)- \tilde{q}_{k}(n_1\e_1),
	\end{align*}
	and substituting this identity into the first summation on the right-hand side of \eqref{eq:partIII-recursion} proves \ref{thm:recursion_fit}. 
\end{proof}

The results of Theorem \ref{thm:characterisation_asyptotic_coefficients} under the special case of parent-independent mutation are stated in the following. 

\begin{corollary}[PIM]
	\label{thm:characterisation_asyptotic_coefficients_PIM}
	Assume parent-independent mutation, i.e.\ $ P_{ij}=Q_j$, $i,j=1,\dots,d$. Then, for $k=0,1,2,\dots,$
	\begin{enumerate}[label=(\Roman*)]
		
		\item (Boundary---fit allele) 
		\label{thm:boundary_fit_PIM}
		\begin{align*}
			\tilde{q}_k(\e_1) 
			=
			\ind{k=0} - 
			\ind{k\geq 1}
			\sum_{i=2}^d
			\tilde{q}_{k-1}(\e_i).
		\end{align*}
		\item (Boundary---unfit alleles)
		\label{thm:boundary_unfit_PIM}
		For $i=2,\dots,d$,
		\begin{align*}
			\tilde{q}_k(\e_i) 
			&=
			\ind{k=0} \theta Q_i  
			+ 
			\ind{k\geq 1} \bigg\{
			- \theta 
			\tilde{q}_{k-1}(\e_i) 
			+ 
			\sum_{j=2}^d
			\tilde{q}_{k-1}(\e_i+\e_j) 
			\bigg\}.
		\end{align*}    
		\item (Recursion---fit allele only)
		\label{thm:recursion_fit_PIM}
		For $n_1>1$, 
		%%% ALTERNATIVE FORMULA
		%\begin{align*}
		% \tilde{q}_k(n_1\e_1) 
		%  = &
		% \tilde{q}_k((n_1-1)\e_1)  
		% \\  &
		% +
		% \ind{k\geq 1}   \bigg\{
		% (n_1-2) 
		%\tilde{q}_{k-1}((n_1-2)\e_1) 
		%+
		%[ \theta P_{11} -   (n_1-2+\theta ) ] \tilde{q}_{k-1}((n_1-1)\e_1)
		%\bigg\}
		%\\  &
		%+  \ind{k\geq 2}  
		%\sum_{j=2}^d
		%n_1 \theta P_{j1}
		%\tilde{q}_{k-2}((n_1-2)\e_1+\e_j) 
		%\end{align*}
		\begin{align*}
			\tilde{q}_k(n_1\e_1)
			={}& 
			\tilde{q}_k((n_1-1)\e_1)
			- \ind{k\geq 1} 
			 (n_1 -2+\theta) \tilde{q}_{k-1}((n_1-1)\e_1)
			\\  & {}+  \ind{k\geq 1} (n_1-2+\theta Q_{1}) 
			\tilde{q}_{k-1}((n_1-2)\e_1) .
		\end{align*}   
		\item (Recursion---general)
		\label{thm:recursion_general_PIM}
		For all $\n$ with $\norm{\n}>1$, 
		\begin{align*}
			(\norm{\n}-n_1 ) \tilde{q}_k(\n) 
			={}&
			\sum_{i=2}^d
			n_i(n_i-1 +\theta Q_{i})
			\tilde{q}_k(\n-\e_i)  
			\\ & + \ind{k\geq 1}  \bigg\{
			- \norm{\n}  (\norm{\n}-1+\theta )  \tilde{q}_{k-1}(\n)
			\\  &  
			{}+n_1(n_1-1+\theta Q_1)\tilde{q}_{k-1}(\n-\e_1)
			+ 
			\norm{\n}
			\sum_{i=2}^d
			\tilde{q}_{k-1}(\n+\e_i) 
			\bigg\}.
		\end{align*}
	\end{enumerate}
\end{corollary} 

\begin{proof}
\ref{thm:boundary_fit_PIM} and \ref{thm:boundary_unfit_PIM} are trivial. For \ref{thm:recursion_fit_PIM} and \ref{thm:recursion_general_PIM}, apply Theorem \ref{thm:characterisation_asyptotic_coefficients}, insert the expression from Lemma \ref{thm:tricks_recursion}, and rearrange. 
\end{proof}

Theorem \ref{thm:characterisation_asyptotic_coefficients} shows that the coefficients 
$\tilde{q}_k(\n)$, $k=0,1,2,\dots$, of the asymptotic expansion \eqref{eq:asymptotic_expansion} are characterised by a double recursion; that is, a main recursion for $k=0,1,2,\dots$ and, for each $k$, a secondary recursion in $\norm{\n}$. More precisely, fix $k$ and suppose $\tilde{q}_{k-2}(\n)$ and $\tilde{q}_{k-1}(\n)$ are known for all $\n$ (if $k=1$, it is enough to know   $\tilde{q}_{k-1}(\n)$, and if $k=0$, no prior knowledge is needed).
Then,  steps \ref{thm:boundary_fit}--\ref{thm:recursion_general} of Theorem 
\ref{thm:characterisation_asyptotic_coefficients} provide a recursion in $\norm{\n}$ to express $\tilde{q}_k(\n)$ in terms of 
$\tilde{q}_{k-1}$,  $\tilde{q}_{k-2}$ (when $k\geq1$ and $k\geq 2$ resp.), and
$\tilde{q}_k(\m)$ with  $\m$ such that  $\norm{\m}< \norm{\n}$. 

Another way of interpreting Theorem \ref{thm:characterisation_asyptotic_coefficients}  is as a recursion in $k+\norm{\n}$. That is, once  $\tilde{q}_k(\n)$ is computed for all $k$ and $\n$ such that $k+\norm{\n}<h$, 
%with $h\in \mathbb{N}\setminus \{0\}$
Theorem \ref{thm:characterisation_asyptotic_coefficients} provides a formula to compute $\tilde{q}_k(\n)$ for all $k$ and $\n$ such that $k+\norm{\n}=h$.
The initial step of this recursion, i.e.\ $h=\norm{\n}=1$,  is given by  
%Theorem \ref{thm:characterisation_asyptotic_coefficients}
\ref{thm:boundary_fit} and \ref{thm:boundary_unfit} for $k=0$. Thus, all values of $\tilde{q}_k(\n)$ can be computed efficiently by dynamic programming. 

The double recursion shows that the coefficients are well characterised in theory, and is useful for the proof of Proposition \ref{thm:coefficient_PIM}, whereas the recursion in $k+\norm{\n}$ provides a general way of numerically computing the coefficients in practice. 
Furthermore, more explicit expressions are available in some cases, to which the rest of this section is dedicated.

\subsubsection*{The $0^{th}$ coefficient}

Let us make explicit the boundary conditions and the formulas for $\tilde{q}_0(\n)$.  
For the boundary conditions, from Theorem \ref{thm:characterisation_asyptotic_coefficients}[\ref{thm:boundary_fit},\ref{thm:boundary_unfit}]
we have 
\begin{equation}
	\label{eq:q_0_boundary}
	\tilde{q}_0(\e_1)=1, \quad \text{and} \quad   
	\tilde{q}_0(\e_i)= \theta P_{1i}, 
	\quad \text{for} \quad i=2,\dots,d.
\end{equation} 
Furthermore, Theorem \ref{thm:characterisation_asyptotic_coefficients}[\ref{thm:recursion_fit}] yields
\begin{equation}
	\label{eq:q_0_recursion_fit_only}
	\tilde{q}_0(n_1\e_1)= \tilde{q}_0((n_1-1)\e_1)= \dots = \tilde{q}_0(\e_1)=1,
\end{equation}
and Theorem \ref{thm:characterisation_asyptotic_coefficients}[\ref{thm:recursion_general}] gives the recursion 
\begin{align}
	\label{eq:q_0_recursion}
	(\norm{\n}-n_1 ) \tilde{q}_0(\n) 
	=&
	%\sum_{i=2}^d
	% n_i(n_i-1)
	%\tilde{q}_0(\n-\e_i) 
	%+
	%\sum_{i=2}^d
	%n_i \theta P_{1i}
	%\tilde{q}_0(\n -\e_i+\e_1) 
	%\\ =&
	\sum_{i=2}^d
	n_i(n_i-1+\theta P_{1i})
	\tilde{q}_0(\n-\e_i) .
\end{align}
%Note also that  $\tilde{q}_0(\n)$ does not depend on the number $n_1$ of fit alleles, as anticipated by Lemma \ref{thm:tricks_recursion}. 
It is interesting to point out here that formulas 
\eqref{eq:q_0_recursion_fit_only} and \eqref{eq:q_0_recursion} 
are closely related to %(i.e.\ contain the rates of)
the limiting ancestral processes of a sample containing fit alleles only and containing some unfit alleles, respectively. In particular one can read off from \eqref{eq:q_0_recursion} the relative rates of coalescence and mutation. This is clarified and  discussed in the next section. 

The recursions \eqref{eq:q_0_recursion_fit_only} and  \eqref{eq:q_0_recursion} with the boundary conditions \eqref{eq:q_0_boundary} give a complete description of $\tilde{q}_0(\n)$, and thus it is now straightforward to verify that the Gamma approximation of the previous section gives the correct formula for $\tilde{q}_0(\n)$, as expected. That is,  
\begin{align}
	\label{eq:zero_coefficient}
	\tilde{q}_0(\n)=  
	\gamma_0(\n) =
	\prod_{i=2}^d \frac{\Gamma(\theta P_{1i}+n_i)}{\Gamma(\theta P_{1i})} .
\end{align}
In fact,  $\gamma_0(\e_1)=1$, and, for $i=2,\dots,d$, $\gamma_0(\e_i)=\frac{\Gamma(\theta P_{1i}+1)}{\Gamma(\theta P_{1i})}= \theta P_{1i}$. Furthermore, $\gamma_0(n_1\e_1)=1$, and $\gamma_0(\n)$ solves \eqref{eq:q_0_recursion}, since 
$\gamma_0(\n-\e_i)=
%\gamma_0(\n-\e_i+\e_1)= 
\frac{1}{n_i-1+\theta P_{1i}} \gamma_0(\n)$.

\subsubsection*{The $1^{st}$ coefficient}

Given that we know $\tilde{q}_0(\n)$ explicitly, we now show how to derive $\tilde{q}_1(\n)$. Let us start with boundary conditions. 
By Theorem \ref{thm:characterisation_asyptotic_coefficients}[\ref{thm:boundary_fit}],
\begin{align*}
	\tilde{q}_1(\e_1) 
	= - 
	\sum_{i=2}^d
	\tilde{q}_{0}(\e_i)
	=
	-  \sum_{i=2}^d \theta P_{1i}
	= 
	- \theta (1-P_{11}) .
\end{align*}
Note that $\tilde{q}_1(\e_1) = \gamma_1(\e_1)$ as expected. 
By Theorem  \ref{thm:characterisation_asyptotic_coefficients}[\ref{thm:boundary_unfit}], for $i=2,\dots,d,$
\begin{align*}
	\tilde{q}_1(\e_i) 
	&=
	- \theta \sum_{j=2}^d 
	[\delta_{ij}+P_{1i}-P_{ji}]
	\tilde{q}_{0}(\e_j) 
	+ 
	\sum_{j=2}^d
	\tilde{q}_{0}(\e_i+\e_j)
	\\&= 
	- \theta \sum_{j=2}^d 
	[\delta_{ij}+P_{1i}-P_{ji}]
	\theta P_{1j}
	+ 
	\sum_{j=2}^d
	(\theta P_{1i}+\delta_{ij}) \theta P_{1j}
	\\&= 
	\theta P_{1i} (1-\theta) + \theta^2 \sum_{j=2}^d  P_{1j}P_{ji}.
	%\comm{=
		%\theta P_{1i}
		%\bigg\{ 1-\theta P_{11} -\theta \bigg[ 1-P_{11} 
		%-\frac{1}{\theta P_{1i}} \sum_{j=2}^d  P_{1j}P_{ji}\bigg] \bigg\}
		%}
\end{align*} 
By Theorem \ref{thm:characterisation_asyptotic_coefficients}[\ref{thm:recursion_fit}], for $n_1>1$, 
    \begin{align}
     \nonumber
    \tilde{q}_1(n_1\e_1) 
    &=
    %\tilde{q}_1((n_1-1)\e_1)  
    %+      (n_1-2) 
    %\tilde{q}_{0}((n_1-2)\e_1) 
    %+
    %[ \theta P_{11} -   (n_1-2+\theta ) ] \tilde{q}_{0}((n_1-1)\e_1)
    %\\  = &
    \tilde{q}_1((n_1-1)\e_1)  
    - \theta(1- P_{11} ) \tilde{q}_{0}((n_1-1)\e_1)
    =  \dots 
    \\ \nonumber &=
    \tilde{q}_1(\e_1) - (n_1-1)\theta(1- P_{11} )
    \\ 
    \label{eq:one_coefficient_allfit}
    &= 
    - n_1\theta(1- P_{11} ).
    \end{align}
Again $\tilde{q}_1(n_1\e_1) = \gamma_1(n_1\e_1)$ as expected. Finally, for  $\n$ with $\norm{\n}>1$, use Theorem \ref{thm:characterisation_asyptotic_coefficients}[\ref{thm:recursion_general}], and rearrange to obtain  
\begin{align*}
	(\norm{\n}-n_1 ) \tilde{q}_1(\n) 
	={} &
	\sum_{i=2}^d
	n_i(n_i-1 +\theta P_{1i})
	\tilde{q}_1(\n-\e_i)  
	\\  &  
	{}+ 
	[n_1(n_1-1+ \theta P_{11}) - \norm{\n}  (\norm{\n}-1+\theta ) ] \tilde{q}_{0}(\n)
	\\  & {}+  
	\sum_{i,j=2}^d
	n_i \theta (P_{ji}- P_{1i})
	\tilde{q}_{0}(\n-\e_i+\e_j) 
	+ 
	\norm{\n}
	\sum_{i=2}^d
	\tilde{q}_{0}(\n+\e_i) 
	\\
	={} &
	\sum_{i=2}^d
	n_i(n_i-1 +\theta P_{1i})
	\tilde{q}_1(\n-\e_i) 
	\\  &  
	{}- (\norm{\n}-n_1 ) (n_1-1+ \theta P_{11}) \tilde{q}_0(\n)
	\\  & +  
	\sum_{i,j=2}^d
	n_i \theta (P_{ji}- P_{1i})
	\frac{n_j+\theta P_{1j}-\delta_{ij}}{n_i+\theta P_{1i}-1}
	\tilde{q}_0(\n).
\end{align*}
All the formulas needed to describe $\tilde{q}_1$ explicitly are derived above. 
They can be used to compute $\tilde{q}_1$ by dynamic programming in the general case of parent-dependent mutation, and, in the PIM case, to verify through straightforward calculations that the explicit formula for $\tilde{q}_1$ is simply the generalisation of the two-allele formula derived in Section \ref{sect:sampling_prob_2alleles}, that is 
\begin{equation*}
	\tilde{q}_1(\n)
	= \tilde{q}_0(\n)
	[r_1(\n)-s_1],
\end{equation*}
with 
$r_1(\n)= [\norm{\n}-n_1+ \theta(1-Q_1)] (1-\theta Q_1 -n_1)$ and $s_1= \theta (1-Q_1)(1-\theta Q_1)$. In fact this generalisation holds at all orders, as we discover in the next result.

%\comm{
	%    \begin{align*}
		%    c_1(\n)&= \frac{\tilde{q}_1(\n)}{\tilde{q}_0(\n)}
		%     \\
		%     (\norm{n}-n_1 ) c_1(\n) 
		%      &=
		%    \sum_{i=2}^d
		%    n_i
		%    c_1(\n-\e_i)   
		%    - (\norm{\n}-n_1 ) (n_1-1+ \theta P_{11})
		%    +  
		%    \sum_{i,j=2}^d
		%    n_i \theta (P_{ji}- P_{1i})
		%    \frac{n_j+\theta P_{1j}-\delta_{ij}}{n_i+\theta P_{1i}-1}
		%    \end{align*}
	%}

\subsubsection*{The $k^{th}$ coefficient}

While Theorem \ref{thm:characterisation_asyptotic_coefficients} can be used in general to compute the coefficients of each order numerically, in the PIM case, Corollary \ref{thm:characterisation_asyptotic_coefficients_PIM} can be used to prove a more explicit expression for the coefficients, which is stated in the following and generalises the two-allele expression derived in Section \ref{sect:sampling_prob_2alleles}. 

\begin{proposition}
	\label{thm:coefficient_PIM}
	Assume parent-independent mutation, i.e.\ $ P_{ij}=Q_j$, $i,j=1,\dots,d$. Then    
	the coefficients $\tilde{q}_k(\n)$, $k=0,1,2,\dots$, of the asymptotic expansion  \eqref{eq:asymptotic_expansion} can be written as
	\begin{align}
		\label{eq:coefficient_PIM}
		\tilde{q}_k(\n)
		= \tilde{q}_0(\n) r_k(\n) -
		\sum_{ m=0 }^{k-1} s_{k-m} \tilde{q}_m(\n),
	\end{align}
	where $\tilde{q}_0(\n)$ is defined by \eqref{eq:zero_coefficient}; and 
	\begin{align*}
		r_k(\n)&= \frac{1}{k!}(\norm{\n}-n_1+\theta (1-Q_1))_k(1-\theta Q_1-n_1)_k  , 
		\\
		s_k&= \frac{1}{k!}(\theta (1-Q_1))_k(1-\theta Q_1)_k.
	\end{align*}
\end{proposition}

\begin{proof}
	The proof is by induction. For $k=0$, \eqref{eq:coefficient_PIM} trivially corresponds to  $\tilde{q}_0(\n)$ defined by \eqref{eq:zero_coefficient}, 
	which satisfies Corollary \ref{thm:characterisation_asyptotic_coefficients_PIM}[\ref{thm:boundary_fit_PIM}--\ref{thm:recursion_general_PIM}], as shown in the previous sections. 
	%For $k=1$, we already know from the previous subsection that it is straightforward to show that \eqref{eq:coefficient_PIM} satisfies Corollary \ref{thm:characterisation_asyptotic_coefficients_PIM}[\ref{thm:boundary_fit_PIM}-\ref{thm:recursion_general_PIM}]. 
	Fix $k\geq 1$, and assume that \eqref{eq:coefficient_PIM} satisfies Corollary \ref{thm:characterisation_asyptotic_coefficients_PIM}[\ref{thm:boundary_fit_PIM}--\ref{thm:recursion_general_PIM}] for $1,\dots, k-1$.
	We have to show that \eqref{eq:coefficient_PIM} then satisfies \ref{thm:boundary_fit_PIM}--\ref{thm:recursion_general_PIM} for this fixed $k$. Note that for $k=1$, we are not making an inductive assumption, but the calculations below show directly that \eqref{eq:coefficient_PIM} satisfies Corollary \ref{thm:characterisation_asyptotic_coefficients_PIM}[\ref{thm:boundary_fit_PIM}--\ref{thm:recursion_general_PIM}] for $k=1$. 
	
	To show \ref{thm:boundary_fit_PIM} and \ref{thm:boundary_unfit_PIM}, first, for $i,j=2,\dots,d,$ note that
	\begin{align*}
		\tilde{q}_{k-1}(\e_i)&=
		\tilde{q}_{0}(\e_i)r_{k-1}(\e_i)
		- \sum_{ m=1 }^{k-1} s_{k-m} \tilde{q}_{m-1}(\e_i),
		\\
		\tilde{q}_{k-1}(\e_i+\e_j)&=
		(\theta Q_j+\delta_{ij}) \tilde{q}_{0}(\e_i)
		r_{k-1}(\e_i+\e_j)
		- \sum_{ m=1 }^{k-1} s_{k-m} \tilde{q}_{m-1}(\e_i+\e_j),
		\\
		r_{k-1}(\e_i)&=
		\frac{k}{-\theta Q_1 \theta (1-Q_1)}
		r_{k}(\e_1)
		=
		\frac{k}{[\theta (1-Q_1)+k][-\theta Q_1+k] }
		r_{k}(\e_i),
		\\
		r_{k-1}(\e_i+\e_j)&=
		\frac{k}{[\theta (1-Q_1)+1][-\theta Q_1+k] }
		r_{k}(\e_i),
		\\ 
		s_k &=  \frac{-\theta Q_1+k}{-\theta Q_1} r_{k}(\e_1)
		=
		\frac{\theta (1-Q_1)}{\theta (1-Q_1)+k } r_{k}(\e_i) .
	\end{align*}
	Therefore,
	%{\tiny\comm{
			%    \begin{align*}
				%    -\theta (1-Q_1)   r_{k-1}(\e_i)  +s_k 
				%    = 
				%    r_{k}(\e_1) \left\{
				%    \frac{k}{\theta Q_1 } -\frac{-\theta Q_1+k}{\theta Q_1}
				%    \right\} 
				%    = r_{k}(\e_1)
				%    \end{align*}
			%}}
	\begin{align*}
		- 
		\sum_{i=2}^d
		\tilde{q}_{k-1}(\e_i) 
		& =
		- \sum_{i=2}^d \tilde{q}_{0}(\e_i) r_{k-1}(\e_i)
		+
		\sum_{ m=1 }^{k-1} s_{k-m} \sum_{i=2}^d\tilde{q}_{m-1}(\e_i)
		\\ & =
		-\theta (1-Q_1)   r_{k-1}(\e_2)  +s_k -s_k 
		-\sum_{ m=1 }^{k-1} s_{k-m} \tilde{q}_{m}(\e_i)
		\\ & =  r_{k}(\e_1) -\sum_{ m=0 }^{k-1} s_{k-m} \tilde{q}_{m}(\e_i)
		\\ & = \tilde{q}_k(\e_1),
	\end{align*}
	where we have used the inductive assumption that \eqref{eq:coefficient_PIM} satisfies Corollary \ref{thm:characterisation_asyptotic_coefficients_PIM}[\ref{thm:boundary_fit_PIM}] for $m=1,\dots,k-1$.
	Furthermore,
	%{\tiny
		%\comm{
			%    \begin{align*}
				%    & -\theta  r_{k-1}(\e_i)  + [\theta (1-Q_1)+1]r_{k-1}(\e_i+\e_j) +s_k
				%      = 
				%    r_{k}(\e_i) \left\{
				%    \frac{- \theta k}{[\theta (1-Q_1)+k][-\theta Q_1+k] }
				%    + \frac{k}{-\theta Q_1+k }
				%    +\frac{\theta (1-Q_1)}{\theta (1-Q_1)+k }
				%    \right\} 
				%    \\ & \quad  =
				%    \frac{r_{k}(\e_i)}{[\theta (1-Q_1)+k][-\theta Q_1+k] } \left\{
				%   - \theta k
				%    + k[\theta (1-Q_1)+k]
				%    +\theta (1-Q_1)[-\theta Q_1+k]
				%    \right\} 
				%     = 
				%   \frac{r_{k}(\e_i)\{[\theta (1-Q_1)+k][-\theta Q_1+k]\}}{[\theta (1-Q_1)+k][-\theta Q_1+k] } 
				%    = r_{k}(\e_i)
				%    \end{align*}
			%}}
	\begin{align*}
		&- \theta 
		\tilde{q}_{k-1}(\e_i) 
		+ 
		\sum_{j=2}^d
		\tilde{q}_{k-1}(\e_i+\e_j) 
		\\ & \quad =
		\tilde{q}_{0}(\e_i)\left\{ 
		- \theta r_{k-1}(\e_i) + 
		\sum_{j=2}^d (\theta Q_j+\delta_{ij}) r_{k-1}(\e_i+\e_j)
		\right\}
		\\ & \quad\quad
		- \sum_{ m=1 }^{k-1} s_{k-m}
		\left\{ 
		- \theta 
		\tilde{q}_{m-1}(\e_i) 
		+ 
		\sum_{j=2}^d
		\tilde{q}_{m-1}(\e_i+\e_j) 
		\right\} 
		\\ & \quad =
		\tilde{q}_{0}(\e_i)\left\{ 
		- \theta r_{k-1}(\e_i) + 
		[\theta (1-Q_1)+1] r_{k-1}(\e_i+\e_2)
		+s_k
		\right\}
		- s_k \tilde{q}_{0}(\e_i) -  \sum_{ m=1 }^{k-1} s_{k-m}
		\tilde{q}_{m}(\e_i) 
		\\ & \quad =
		\tilde{q}_{0}(\e_i) r_{k}(\e_i) 
		- \sum_{ m=0 }^{k-1} s_{k-m}  \tilde{q}_{m}(\e_i)
		\\ &\quad =
		\tilde{q}_k(\e_i) ,
	\end{align*} 
	where we have used the inductive assumption that \eqref{eq:coefficient_PIM} satisfies Corollary \ref{thm:characterisation_asyptotic_coefficients_PIM}[\ref{thm:boundary_unfit_PIM}] for $m=1,\dots,k-1$.
	
	To show \ref{thm:recursion_fit_PIM}, first, for $n_1>1,$  note that 
	\begin{align*}
    \allowdisplaybreaks
		\tilde{q}_{k}((n_1-1)\e_1)&=  
		r_{k}((n_1-1)\e_1) 
		-\sum_{ m=0 }^{k-1} s_{k-m} \tilde{q}_{m}((n_1-1)\e_1) ,
		\\
		\tilde{q}_{k-1}((n_1-1)\e_1)&=  
		r_{k-1}((n_1-1)\e_1)
		-\sum_{ m=1 }^{k-1} s_{k-m} \tilde{q}_{m-1}((n_1-1)\e_1)   ,
		\\
		\tilde{q}_{k-1}((n_1-2)\e_1)&=  
		r_{k-1}((n_1-2)\e_1)
		-\sum_{ m=1 }^{k-1} s_{k-m} \tilde{q}_{m-1}((n_1-2)\e_1)  ,
		\\
		r_{k}((n_1-1)\e_1)&=  
		\frac{1-\theta Q_1-n_1+k}{1-\theta Q_1-n_1} 
		r_k(n_1\e_1) ,
		\\
		r_{k-1}((n_1-1)\e_1)&=  
		\frac{k}{[\theta(1-Q_1)+k-1][1-\theta Q_1-n_1]} 
		r_k(n_1\e_1)  ,
		\\
		r_{k-1}((n_1-2)\e_1)&=  
		\frac{k[1-\theta Q_1-n_1+k]}{[\theta(1-Q_1)+k-1][1-\theta Q_1-n_1][2-\theta Q_1-n_1]} 
		r_k(n_1\e_1)  .
	\end{align*}  
	Therefore,
	%{\tiny
		%\comm{
			%    \begin{align*}
				%    & r_k((n_1-1)\e_1)
				%     -  (n_1-2+\theta ) r_{k-1}((n_1-1)\e_1)
				%     +  (n_1-2+\theta Q_{1}) r_{k-1}((n_1-2)\e_1)
				%     \\ &\quad =
				%     \frac{r_k(n_1\e_1)}{1-\theta Q_1-n_1}
				%     \bigg\{
				%     1-\theta Q_1-n_1+k
				%     -k 
				%     \frac{n_1-2+\theta}{\theta(1-Q_1)+k-1} 
				%     -k
				%     \frac{1-\theta Q_1-n_1+k}{\theta(1-Q_1)+k-1}
				%     \bigg\}
				%     %\\ &\quad 
				%     =
				%     \frac{r_k(n_1\e_1)}{1-\theta Q_1-n_1}
				%     \bigg\{
				%     1-\theta Q_1-n_1+k
				%     -k 
				%     \frac{\theta(1-Q_1)+k-1}{\theta(1-Q_1)+k-1} 
				%     \bigg\}
				%      =
				%     r_k(n_1\e_1)
				%    \end{align*}
			%}}
	\begin{align*}
		& 
		\tilde{q}_k((n_1-1)\e_1)
		-  (n_1-2+\theta ) \tilde{q}_{k-1}((n_1-1)\e_1)
		+  (n_1-2+\theta Q_{1}) \tilde{q}_{k-1}((n_1-2)\e_1)  
		\\  & \quad  =
		r_k((n_1-1)\e_1)
		-  (n_1-2+\theta ) r_{k-1}((n_1-1)\e_1)
		+  (n_1-2+\theta Q_{1}) r_{k-1}((n_1-2)\e_1) 
		\\
		& \quad \quad 
		-\sum_{ m=0 }^{k-1} s_{k-m} \bigg\{
		\tilde{q}_m((n_1-1)\e_1)
		- \ind{m\geq 1}  (n_1-2+\theta ) \tilde{q}_{m-1}((n_1-1)\e_1)
		\\ & \quad \quad  
		+ \ind{m\geq 1} (n_1-2+\theta Q_{1}) \tilde{q}_{m-1}((n_1-2)\e_1)
		\bigg\}
		\\  & \quad  =
		r_k(n_1\e_1)- \sum_{ m=0 }^{k-1} s_{k-m} \tilde{q}_{m}(n_1\e_1)
		\\
		&\quad 
		=\tilde{q}_k(n_1\e_1),
	\end{align*} 
	where we have used the inductive assumption that \eqref{eq:coefficient_PIM} satisfies Corollary \ref{thm:characterisation_asyptotic_coefficients_PIM}[\ref{thm:recursion_fit_PIM}] for $m=1,\dots,k-1$, and for $m=0$ by construction. 
	
	Finally, to show \ref{thm:recursion_general_PIM}, 
	first, for $i=2,\dots,d,$ note that, 
	\begin{align*}
		\tilde{q}_{k}(\n-\e_i)
		&=
		%\tilde{q}_0(\n-\e_i) r_k(\n-\e_i) -
		%\sum_{ m=0 }^{k-2} s_{k-1-m} \tilde{q}_m(\n-\e_i)
		%\\    &= 
		\frac{1}{n_i-1+\theta Q_i} \tilde{q}_0(\n) r_k(\n-\e_i)
		-\sum_{ m=0 }^{k-1} s_{k-m} \tilde{q}_{m}(\n-\e_i),
		\\
		\tilde{q}_{k-1}(\n)
		&=
		%\tilde{q}_0(\n) r_{k-1}(\n) -
		%\sum_{ m=0 }^{k-2} s_{k-1-m} \tilde{q}_m(\n)
		%\\     &= 
		\tilde{q}_0(\n) r_{k-1}(\n)
		-\sum_{ m=1 }^{k-1} s_{k-m} \tilde{q}_{m-1}(\n),
		\\
		\tilde{q}_{k-1}(\n-\e_1)
		%&=
		%\tilde{q}_0(\n-\e_1) r_{k-1}(\n-\e_1) -
		%\sum_{ m=0 }^{k-2} s_{k-1-m} \tilde{q}_m(\n-\e_1)
		%\\
		&= 
		\tilde{q}_0(\n) r_{k-1}(\n-\e_1)
		-\sum_{ m=1 }^{k-1} s_{k-m} \tilde{q}_{m-1}(\n-\e_1),
		\\
		\tilde{q}_{k-1}(\n+\e_i)
		&=
		%\tilde{q}_0(\n+\e_i) r_{k-1}(\n+\e_i) -
		%\sum_{ m=0 }^{k-2} s_{k-1-m} \tilde{q}_m(\n+\e_i)
		%\\
		%&= 
		(n_i+\theta Q_i)\tilde{q}_0(\n) r_{k-1}(\n+\e_i)
		-\sum_{ m=1 }^{k-1} s_{k-m} \tilde{q}_{m-1}(\n+\e_i),
	\end{align*}	
        \begin{align*}
		r_{k}(\n-\e_i)&=  
		\frac{\norm{\n}-n_1-1+\theta(1-Q_1)}{\norm{\n}-n_1+\theta(1-Q_1)+k-1} 
		r_k(\n) ,
		\\
		r_{k-1}(\n)&=  
		\frac{k}{[\norm{\n}-n_1+\theta(1-Q_1)+k-1][-n_1-\theta Q_1 +k]} 
		r_k(\n)  ,
		\\
		r_{k-1}(\n-\e_1)&=  
		\frac{k}{[\norm{\n}-n_1+\theta(1-Q_1)+k-1][-n_1-\theta Q_1 +1]} r_k(\n)  ,
		\\
		r_{k-1}(\n+\e_i)&=  
		\frac{k}{[\norm{\n}-n_1+\theta(1-Q_1)][-n_1-\theta Q_1 +k]} r_k(\n) . 
	\end{align*}  
	Therefore,
	%{\tiny\comm{
			%    \begin{align*}
				%    &(\norm{\n}-n_1) r_k(\n-\e_i) 
				%     - \norm{\n}  (\norm{\n}-1+\theta )  r_{k-1}(\n)
				%     \\  & 
				%    + n_1(n_1-1+\theta Q_1)r_{k-1}(\n-\e_1)
				%    + 
				%    \norm{\n} (\norm{\n}-n_1 +\theta (1-Q_1))
				%    r_{k-1}(\n+\e_i) 
				%    \\ &\quad =
				%    \frac{r_k(\n)}{\norm{\n}-n_1+\theta(1-Q_1)+k-1}
				%    \bigg\{
				%    (\norm{\n}-n_1) [\norm{\n}-n_1-1+\theta(1-Q_1)]
				%    %\\ &\quad \quad 
				%    - k \norm{\n}    \frac{\norm{\n}-1+\theta}{-n_1-\theta Q_1 +k}
				%    -k n_1
				%    +  k \norm{\n}
				%    \frac{\norm{\n}-n_1+ \theta(1-Q_1)+k-1}{-n_1-\theta Q_1 +k}
				%    \bigg\}
				%    \\ &\quad =
				%    \frac{r_k(\n)}{\norm{\n}-n_1+\theta(1-Q_1)+k-1}
				%    \bigg\{
				%    (\norm{\n}-n_1) [\norm{\n}-n_1-1+\theta(1-Q_1)] %  
				%    %\\ &\quad \quad 
				%    -k n_1
				%    +  k \norm{\n} \frac{-n_1-\theta Q_1 +k}{-n_1-\theta Q_1 +k}
				%    \bigg\}
				%    \\ &\quad =
				%    \frac{r_k(\n)}{\norm{\n}-n_1+\theta(1-Q_1)+k-1}
				%    \bigg\{
				%    (\norm{\n}-n_1) [\norm{\n}-n_1-1+\theta(1-Q_1)+k]   
				%    \bigg\}
				%    %\\ &\quad 
				%    =
				%    (\norm{\n}-n_1)r_k(\n)
				%    \end{align*}
			%}}
	\begin{align*}
		\allowdisplaybreaks
		&\sum_{i=2}^d
		n_i(n_i-1 +\theta Q_{i})
		\tilde{q}_k(\n-\e_i)  
		- \norm{\n}  (\norm{\n}-1+\theta )  \tilde{q}_{k-1}(\n)
		\\  &  
		+n_1(n_1-1+\theta Q_1)\tilde{q}_{k-1}(\n-\e_1)
		+ 
		\norm{\n}     \sum_{i=2}^d     \tilde{q}_{k-1}(\n+\e_i) 
		\\  & \quad  =
		\tilde{q}_{0}(\n) 
		\bigg\{
		(\norm{\n}-n_1) r_k(\n-\e_2) 
		- \norm{\n}  (\norm{\n}-1+\theta )  r_{k-1}(\n)
		\\  &  \quad  \quad 
		+ n_1(n_1-1+\theta Q_1)r_{k-1}(\n-\e_1)
		+ 
		\norm{\n} (\norm{\n}-n_1 +\theta (1-Q_1))
		r_{k-1}(\n+\e_2) 
		\bigg\}
		\\& \quad \quad  - 
		\sum_{ m=0 }^{k-1} s_{k-m}
		\bigg\{
		\sum_{i=2}^d
		n_i(n_i-1 +\theta Q_{i})
		\tilde{q}_m(\n-\e_i)  
		-\ind{m\geq 1} \norm{\n}  (\norm{\n}-1+\theta )  \tilde{q}_{m-1}(\n)
		\\  &  
		\quad \quad 
		+\ind{m\geq 1} n_1(n_1-1+\theta Q_1)\tilde{q}_{m-1}(\n-\e_1)
		+ \ind{m\geq 1} \norm{\n}     \sum_{i=2}^d     \tilde{q}_{m-1}(\n+\e_i) 
		\bigg\}
		\\  & \quad  = \tilde{q}_0(\n) (\norm{\n}-n_1) r_k(\n) -
		\sum_{ m=0 }^{k-1} s_{k-m} (\norm{\n}-n_1) \tilde{q}_m(\n)
		\\  & \quad =
		(\norm{\n}-n_1) \tilde{q}_k(\n),
	\end{align*}
	where we have used the inductive assumption that \eqref{eq:coefficient_PIM} satisfies Corollary \ref{thm:characterisation_asyptotic_coefficients_PIM}[\ref{thm:recursion_general_PIM}] for $m=1,\dots,k-1$, and for $m=0$ by construction.
\end{proof}

\begin{remark}\label{rmk:power-series}
    One might wonder whether it is possible to let $K\to\infty$ in \eqref{eq:asymptotic_expansion} to obtain a power series converging to $q^{(\sigma)}(\n)$. In previous work \cite{jenkins2012} the analogous question is addressed for an expansion of the two-locus sampling distribution in inverse powers of the recombination parameter $\rho$:
\begin{equation*}
%\label{eq:rho}
q^{(\rho)}(\n) = \sum_{k=0}^K \frac{q_k(\n)}{\rho^k} + \mathcal{O}(\rho^{-{K+1}}), \qquad \rho \to \infty.
\end{equation*}
In that case one can say quite a lot: $q^{(\rho)}(\n)$ can be recovered from the solution to a finite, linear system of equations with coefficients linear in $\rho$, which implies that $q^{(\rho)}(\n)$ is a \emph{rational} function of $\rho$ with a series expansion in powers of $\rho^{-1}$ whose radius of convergence can in principle be recovered by finding the poles in $\mathbb{C}$ of this rational function. Unfortunately, this argument all breaks down for selection, since the corresponding system of equations for selection, \eqref{eq:sampling_recursion}, is no longer a finite system. One can however gain some insight into the convergence properties of \eqref{eq:asymptotic_expansion} as $K\to\infty$ owing to the particular solution \eqref{eq:1F1solution} available for the special case of a diallelic, parent-independent mutation model. That solution is given in terms of the confluent hypergeometric function $_1F_1$, which can exhibit an essential singularity at $\infty$; it is known \cite[p2]{Slater} that there is no simple series solution of the form
\[
_1F_1[a,b,z] = \sum_{k=0}^\infty \frac{u_k}{z^k}.
\]
%Using standard identities for hypergeometric functions we found simple examples of configurations $\n$ for which the series expansion of \eqref{eq:1F1solution} in powers of $\sigma^{-1}$ fails to converge to $q^{\sigma}(\n)$ for all $\sigma < \infty$.
We emphasise however that divergence of \eqref{eq:asymptotic_expansion} as $K\to\infty$ should not be viewed as problematic: the philosophy of asymptotic methods is that the first few terms of a divergent series can offer superior accuracy to a convergent one \cite[p21]{Hinch}.
\end{remark}

\section{Ancestral processes}
\label{sect:genealogy}

The ancestral selection graph was introduced in \cite{Krone1997} and \cite{Neuhauser1997}
and models the ancestral history of a sample by means of branching-coalescing virtual and real lineages which evolve backwards in time (see also \cite{dk1999} for a very general construction). The genealogy of the sample is embedded in the graph and can be recovered once the ultimate ancestor is found, i.e.\ only one lineage in the graph is left.
We consider a \emph{conditional} ASG, i.e.\ with typed lineages evolving backwards rather than having types which are superimposed on lineages afterwards. The former appears in  
\cite{baake2008,Etheridge2009,Favero2021,Favero2022,fearnhead2002,slade2000simulation,slade2000MRCA,stephens2003}, among others. 
In the conditional ASG, the embedded genealogy, which corresponds to the real lineages, is constructed while the graph evolves, not afterwards as for the unconditional ASG. The virtual lineages are still needed in order to account for the effects of selection in the rates. 

We describe in detail below the block-counting process of this conditional ASG, at stationarity, which counts the number of real and virtual lineages of each type. 
As in the previous Section \ref{sect:sampling_prob_full_expansion}, we make the assumption $\sigma_2,\dots,\sigma_d=0$, discussed in Remark \ref{rmk:assumption_selection}.  %and $\sigma_1=\sigma$. 
Let $\n\in\mathbb{N}^d$  and $\nuu\in\mathbb{N}^d$ be the present numbers of real and virtual lineages, respectively, and refer to a lineage of type $i$ as an \emph{$i$-lineage}. From configuration $(\n,\nuu)$, the block-counting process of the conditional ASG jumps to the following possible configurations: 
\begin{itemize}
	
	\item[$\circ$]
	$(\n,\nuu) \mapsto(\n-\e_i,\nuu)$ 
	at rate 
	$\frac{n_i(n_i-1)}{2}\frac{q\sth(\n+\nuu -\e_i)}{q\sth(\n+\nuu)}$, for $i=1,\dots,d$,
	
	(coalescence of two real $i$-lineages);
	
	\item[$\circ$]
	$(\n,\nuu) \mapsto(\n,\nuu-\e_i)$ 
	at rate 
	$\frac{\nu_i(\nu_i-1+2n_i)}{2}\frac{q\sth(\n+\nuu -\e_i)}{q\sth(\n+\nuu)}$, 
	for $i=1,\dots,d$,
	
	(coalescence of  two virtual $i$-lineages, or of a real and a virtual $i$-lineage);
	%\item[$\circ$]
	%$(n,\nu-e_i)$ 
	%at rate $\frac{2 n_i\nu_i}{2}\frac{q\sth(n+\nu -e_i)}{q\sth(n+\nu)}$
	%[coalescence of a real and a virtual $i$-lineage]
	
	\item[$\circ$]
	$(\n,\nuu) \mapsto(\n-\e_i+\e_j,\nuu)$
	at rate   
	$n_i \frac{\theta P_{ji}}{2}\frac{q\sth(\n+\nuu -\e_i+\e_j)}{q\sth(\n+\nuu)}$,
	%$j$ to $i$ forwards
	for $i,j=1,\dots,d$,
	
	(mutation of a real lineage from type $i$ to type $j$);
	
	\item[$*$]
	$(\n,\nuu) \mapsto(\n,\nuu-\e_i+\e_j)$
	at rate   
	$\nu_i \frac{\theta P_{ji}}{2}\frac{q\sth(\n+\nuu -\e_i+\e_j)}{q\sth(\n+\nuu)}$,
	%$j$ to $i$ forwards
	for $i,j=1,\dots,d$, 
	
	(mutation of a virtual lineage from type $i$ to type $j$);
	
	\item[$\circ$]
	$(\n,\nuu) \mapsto(\n,\nuu+\e_i)$ 
	at rate
	$\norm{\n+\nuu} \frac{\sigma}{2}\frac{q\sth(\n+\nuu  +\e_i)}{q\sth(\n+\nuu)}$,
	for $i=2,\dots,d$,
	
	(selective event between an unfit incoming $i$-lineage and a parental continuing lineage);
	
	\item[$*$]
	$(\n,\nuu) \mapsto(\n,\nuu+\e_i)$ 
	at rate
	$(n_1+\nu_1) \frac{\sigma}{2}\frac{q\sth(\n+\nuu +\e_i)}{q\sth(\n+\nuu)}$,
	for $i=1,\dots,d$,
	
	(selective event between a continuing $i$-lineage and a parental fit incoming lineage);
\end{itemize}
where $q\sth$ is the sampling probability of Section \ref{sect:sampling_prob}. 
Note that the total rate,  the sum of all of the above rates, is equal to 
$\frac{1}{2}\norm{\n+\nuu}(\norm{\n+\nuu}-1+\theta +\sigma)$, by using formulas \eqref{eq:sampling_recursion} and \eqref{eq:to_reduce_recursion}, and thus explodes, as $\sigma\to\infty$. 
%\comm{\footnotesize{(recall: incoming branch is parental if and only if it is of the fit type)
		%\begin{itemize}
		%    \item[$\circ$]
		%    continuing branch of type $j$ (parental, doesn't matter if fit/unfit), incoming branch of type $i$ (must be unfit) 
		%    \\ Rate: $(n_j+\nu_j) \frac{\sigma}{2}\frac{q\sth(n+\nu  +e_i)}{q\sth(n+\nu)}
		%        I\{i\in\mathcal{U}\}$, sum over $j$ now
		%    \item[$*$]
		%     incoming branch of type $j$ (parental, must be fit), continuing branch of type $i$ (no matter if fit/unfit). 
		%    \\
		%        Rate: $(n_j+\nu_j) \frac{\sigma}{2}\frac{q\sth(n+\nu +e_i)}{q\sth(n+\nu)}
		%        I\{j\in\mathcal{F}\}
		%        $  ,  $j$ must be $1$   \com{(this event can be omitted in the reduced ASG )}
		%\end{itemize}
		%}}

Let us explain the selective events above. A selection event in the ASG is associated with two lineages:
an incoming lineage (the lineage that transmits its type to the descendant only if fit) and a continuing lineage (the lineage that transmits its type to the descendant unless the incoming lineage is fit). The lineage that transmits its type to the descendant is called parental. 
Now, pick a lineage uniformly at random from the current configuration $(\n,\nuu)$, say of type $j$ (with probability proportional to $n_j+\nu_j$).  
Say that this is the type of the descendant in a selection event.
Two possible selection events may have led to a descendant of type $j$. 
The first involves a parental (fit or unfit) continuing $j$-lineage and an incoming unfit $i$-lineage; this happens at rate $(n_j+\nu_j) \frac{\sigma}{2}\frac{q\sth(n+\nu  +e_i)}{q\sth(n+\nu)}\ind{i\neq 1}$, 
which, summing over $j$ corresponds to the first of the selective jumps in the list above. 
The second possible event, which can only happen if $j=1$, involves a parental incoming $j$-lineage and a (fit or unfit) continuing $i$-lineage; this happens at rate 
$(n_1+\nu_1)\frac{\sigma}{2}\frac{q\sth(n+\nu +e_i)}{q\sth(n+\nu)}$, which appears as a second kind of selective jump in the list above. 
In both cases a new virtual $i$-lineage appears. 
See \cite{fearnhead2002} for more details. 

We can further `reduce' the conditional ASG, following \cite{baake2008,fearnhead2002,slade2000MRCA}, by modifying the jumps marked by $*$ above.  
More precisely, we discard any virtual lineage that undergoes a mutation event, so that jumps to $(\n,\nuu-\e_i+\e_j)$ become jumps to $(\n,\nuu-\e_i)$, and we eliminate selective events where the parental lineage is a fit incoming branch (the second kind of selective events described above). 
This means that virtual lineages of the fit type can be entirely omitted. If the initial configuration does not contain virtual lineages, i.e.\ the initial configuration is of the form $(\n,\boldsymbol{0})$, because of the reduction no virtual lineages of the fit type are created while the process evolves and we have $\nu_1 = 0$ at all times.

To summarise, the block-counting process of the reduced conditional ASG is the continuous-time jump process on $\mathbb{N}^d \times \mathbb{N}^d$
%on $\mathbb{N}^d\times \mathbb{N}^{d-1}$ 
with the following jump rates, conditioned on the current configuration $(\n,\nuu)$:
	\begin{align}
	\nonumber
	&\rho\sth_{\n,\nuu}(\n-\e_i,\nuu)
	=\frac{n_i(n_i-1)}{2}\frac{q\sth(\n+\nuu -\e_i)}{q\sth(\n+\nuu)}, \quad i=1,\dots,d ; 
	\\ \nonumber
	& \rho\sth_{\n,\nuu}(\n,\nuu-\e_i)
	= \frac{\nu_i(\nu_i-1+2n_i)}{2}\frac{q\sth(\n+\nuu -\e_i)}{q\sth(\n+\nuu)}
	\\ \nonumber & \quad\quad\quad\quad\quad\quad\quad\quad
	+
	\sum_{j=1}^d\nu_i \frac{\theta P_{ji}}{2}\frac{q\sth(\n+\nuu -\e_i+\e_j)}{q\sth(\n+\nuu)},
	  i=2,\dots,d;
	\\ \nonumber
	&\rho\sth_{\n,\nuu}(\n-\e_i+\e_j,\nuu)
	=n_i \frac{\theta P_{ji}}{2}\frac{q\sth(\n+\nuu -\e_i+\e_j)}{q\sth(\n+\nuu)},
	\quad i,j=1,\dots,d ;
	\\ \label{eq:rates_reduced_conditional_ASG}
	&\rho\sth_{\n,\nuu}(\n,\nuu+\e_i)
	= \norm{\n+\nuu} \frac{\sigma}{2}\frac{q\sth(\n+\nuu  +\e_i)}{q\sth(\n+\nuu)},
	\quad i=2,\dots,d.
	\end{align}
Under the PIM assumption, it is proven \cite{slade2000MRCA,fearnhead2002} that the reduction does not affect the distribution of the genealogy of a sample, and thus the reduced graph encodes all the relevant information. 
For general mutation models this is not true. Nevertheless, the process with jump rates given by \eqref{eq:rates_reduced_conditional_ASG} still defines a valid Markov jump process whose properties we can analyse. We should not expect this process to necessarily embed the true genealogy (for non-PIM models). However, we now proceed to an asymptotic analysis without assuming PIM because this leads to limiting processes that have interesting properties---more precisely, as $\sigma\to\infty$ we find that the process described by \eqref{eq:rates_reduced_conditional_ASG} converges to a limiting process which is dual to the limiting diffusion of Section \ref{sect:diffusions}, as we show below.

Note that the rate of the eliminated selective events is $\sum_{i=1}^d (n_1+\nu_1) \frac{\sigma}{2}\frac{q\sth(\n+\nuu +\e_i)}{q\sth(\n+\nuu)} =  (n_1+\nu_1) \frac{\sigma}{2} $, by  formula \eqref{eq:to_reduce_recursion}. Therefore, the  total rate of the reduced process is equal to 
$$\frac{1}{2}\norm{\n+\nuu}(\norm{\n+\nuu}-1+\theta )+\frac{1}{2} \sigma(\norm{\n+\nuu} - n_1 -\nu_1).$$
The reduced total rate above can still explode as $\sigma\to\infty$, namely for those configurations with $\norm{\n+\nuu} - n_1 -\nu_1>0 $. 
However, for configurations without unfit lineages, i.e.\ $\norm{\n+\nuu}= n_1+\nu_1$, the rate above does not depend on sigma and thus it does not explode. 
This is what allows the following asymptotic analysis, which would not be possible for the unreduced process. This is further discussed in Section \ref{sect:genealogy_interpretation}.

\subsection{Asymptotic analysis of the rates}
The analysis of this section focuses on the asymptotic behaviour of the jump rates of the reduced, conditional ASG, as in \cite{wakeley2008}, rather than on a rigorous proof of convergence of the processes. 
A rigorous proof seems to be possible by using the separation of timescales arguments of \cite{mohle}. Here we simply generalise the argument of \cite{wakeley2008} to the multi-allele case in order to obtain a process that is needed for the duality relationship of the next section.

From the asymptotic expansion of the sampling probability of Section \ref{sect:sampling_prob}, we have
\begin{align*}
	&q\sth(\n+\nuu) = \frac{1}{\sigma ^{\norm{\n+\nuu}-(n_1+\nu_1)}} \left[ \tilde{q}_0(\n+\nuu)+\mathcal{O}\left(\frac{1}{\sigma} \right) \right],
	\\
	& q\sth(\n+\nuu-\e_i) = \frac{1}{\sigma ^{\norm{\n+\nuu}-(n_1+\nu_1)- 1+\delta_{1i}}} 
	\left[ \tilde{q}_0(\n+\nuu-\e_i)+\mathcal{O}\left(\frac{1}{\sigma} \right) \right],
	\\
	&q\sth(\n+\nuu-\e_i+\e_j) 
	=\frac{1}{\sigma ^{\norm{\n+\nuu}-(n_1+\nu_1)+\delta_{1i}-\delta_{1j}}} 
	\left[ \tilde{q}_0(\n+\nuu-\e_i+\e_j)+\mathcal{O}\left(\frac{1}{\sigma} \right) \right],
	\\
	& q\sth(\n+\nuu+\e_i) = \frac{1}{\sigma ^{\norm{\n+\nuu}-(n_1+\nu_1)+ 1-\delta_{1i}}} 
	\left[ \tilde{q}_0(\n+\nuu+\e_i)+\mathcal{O}\left(\frac{1}{\sigma} \right) \right].
\end{align*}
Therefore, recalling formula \eqref{eq:zero_coefficient} for $\tilde{q}_0$,
	\begin{align*}
	&\frac{q\sth(\n+\nuu -\e_i)}{q\sth(\n+\nuu)}
	=
	\begin{cases}
	1 + \mathcal{O}\left(\frac{1}{\sigma} \right),
	&  i= 1 ;\\ 
	\sigma \frac{1}{n_i +\nu_i -1 +\theta P_{1i}}  + \mathcal{O}\left(1 \right),
	&  i =2,\dots, d ;
	\end{cases}
	\\
	&\frac{q\sth(\n+\nuu -\e_i+\e_j)}{q\sth(\n+\nuu)}
	=
	\begin{cases}
	1,  %+ \mathcal{O}\left(\frac{1}{\sigma} \right) 
	& i=j= 1;\\
	\frac{1}{\sigma} (n_j +\nu_j+ \theta P_{1j} )  
	+ \mathcal{O}\left(\frac{1}{\sigma^2} \right),
	& i=1,\, j =2,\dots,d ;
	\\
	\sigma \frac{1}{n_i +\nu_i -1 +\theta P_{1i}} 
	+ \mathcal{O}\left(1\right),
	& i=2,\dots,d,\, j =1 ;
	\\
	\frac{n_j +\nu_j-\delta_{ij}+ \theta P_{1j}}{n_i +\nu_i -1 +\theta P_{1i}} 
	+ \mathcal{O}\left(\frac{1}{\sigma} \right),
	& i,j=2,\dots,d ;
	\end{cases}
	\\
	&\frac{q\sth(\n+\nuu +\e_i)}{q\sth(\n+\nuu)}
	= 
	\begin{cases}
	1  + \mathcal{O}\left(\frac{1}{\sigma} \right), 
	&  i= 1 ;\\ 
	\frac{1}{\sigma} (n_i +\nu_i+ \theta P_{1i} ) 
	+ \mathcal{O}\left(\frac{1}{\sigma^2} \right),
	&  i =2,\dots, d . 
	\end{cases}
	\end{align*}
It is now straightforward to write asymptotic expressions for the rates in \eqref{eq:rates_reduced_conditional_ASG} as follows
\begin{align*}
	&\rho\sth_{\n,\nuu}(\n-\e_i,\nuu)
	=
	\begin{cases}
		\frac{n_1(n_1-1)}{2} + \mathcal{O}\left(\frac{1}{\sigma} \right),
		& i=1; \\
		\sigma\frac{n_i(n_i-1)}{2} \frac{1}{n_i +\nu_i -1 +\theta P_{1i}}  + \mathcal{O}\left(1 \right),
		& i=2,\dots,d ;
	\end{cases} 
	\\
	& \rho\sth_{\n,\nuu}(\n,\nuu-\e_i)
	=\sigma \frac{\nu_i(\nu_i-1+2n_i+\theta P_{1i})}{2}\frac{1}{n_i +\nu_i -1 +\theta P_{1i}} 
	+ \mathcal{O}\left(1 \right),
	\quad i=2,\dots,d;
	\\
	&\rho\sth_{\n,\nuu}(\n-\e_i+\e_j,\nuu)
	=
	\begin{cases}
		n_1 \frac{\theta P_{11}}{2}  + \mathcal{O}\left(\frac{1}{\sigma} \right) ,
		& i=j= 1;\\
		\mathcal{O}\left(\frac{1}{\sigma} \right),
		& i=1,\, j =2,\dots,d ;
		\\
		\sigma n_i \frac{\theta P_{1i}}{2} \frac{1}{n_i +\nu_i -1 +\theta P_{1i}} 
		+ \mathcal{O}\left(1\right),
		& i=2,\dots,d,\, j =1 ;
		\\
		n_i \frac{\theta P_{ji}}{2} \frac{n_j +\nu_j-\delta_{ij}+ \theta P_{1j}}{n_i +\nu_i -1 +\theta P_{1i}} 
		+ \mathcal{O}\left(\frac{1}{\sigma} \right),
		& i,j=2,\dots,d;
	\end{cases}
	\\
	&\rho\sth_{\n,\nuu}(\n,\nuu+\e_i)
	= 
	\frac{ \norm{\n+\nuu}}{2}  (n_i +\nu_i+ \theta P_{1i} ) 
	+ \mathcal{O}\left(\frac{1}{\sigma} \right),
	\quad  i =2,\dots,d. 
\end{align*}

If there are no unfit lineages in a configuration $(\n,\nuu)$, i.e.\ $\norm{\n+\nuu}= n_1+\nu_1$, we have previously noticed that the total jump rate does not explode. In this case, the only possible jumps in the strong selection limit are given by the rates 
\begin{align}
	\nonumber
	&\rho\sth_{\n,\nuu}(\n-\e_1,\nuu)
	\xrightarrow[\sigma\to\infty]{}
	\frac{n_1(n_1-1)}{2},
	%\\
	%&\rho\sth_{\n,\nuu}(\n-\e_i+\e_j,\nuu)
	%\xrightarrow[\sigma\to\infty]{}
	%n_1 \frac{\theta P_{11}}{2}   ,
	\\  \label{eq:rates_limits}
	&\rho\sth_{\n,\nuu}(\n,\nuu+\e_i)
	\xrightarrow[\sigma\to\infty]{}
	\frac{ \norm{\n+\nuu}}{2}  \theta P_{1i} ,
	\quad  i =2,\dots, d . 
\end{align}
Instead, if  $\norm{\n+\nuu}> n_1+\nu_1$, we divide the rates $\rho\sth_{\n,\nuu}$ by $\sigma$, which corresponds to rescaling time by a factor $\sigma$; then we obtain that the only possible jumps in the limit are given by the rates
\begin{align}
	\nonumber
	&\frac{1}{\sigma} \rho\sth_{\n,\nuu}(\n-\e_i,\nuu)
	\xrightarrow[\sigma\to\infty]{}
	\frac{1}{2} \frac{n_i(n_i-1)}{n_i +\nu_i -1 +\theta P_{1i}} ,
	\quad i=2,\dots,d ;
	\\ \nonumber
	& \frac{1}{\sigma}  \rho\sth_{\n,\nuu}(\n,\nuu-\e_i)
	\xrightarrow[\sigma\to\infty]{}
	\frac{1}{2}\frac{\nu_i(\nu_i-1+2n_i+\theta P_{1i})}{n_i +\nu_i -1 +\theta P_{1i}} ,
	\quad i=2,\dots,d;
	\\ \label{eq:rates_limits_scaled}
	&\frac{1}{\sigma} \rho\sth_{\n,\nuu}(\n-\e_i+\e_1,\nuu)
	\xrightarrow[\sigma\to\infty]{}
	\frac{1}{2} \frac{n_i\theta P_{1i}}{n_i +\nu_i -1 +\theta P_{1i}} ,
	\quad i=2,\dots,d  .
\end{align}
%\comm{\footnotesize Double checked: Total rate above = $\frac{1}{2}(\norm{\n+\nuu}-n_1-\nu_1)$= limit of reduced total rate divided by $\sigma$}

\subsection{Interpretation of the ancestral process  under strong selection}
\label{sect:genealogy_interpretation}

\cite{wakeley2008} performs a similar asymptotic analysis in the case of two alleles, calling the two processes in \eqref{eq:rates_limits} and \eqref{eq:rates_limits_scaled} the `slow' and `fast' processes respectively. We now follow and generalise \cite{wakeley2008} to provide an interpretation of the asymptotic analysis of the rates in the multi-allele case; this is not a trivial extension because two-allele models can always be parametrised to obey the PIM assumption. 
Recall that under PIM the reduced graph is equivalent to the original graph, but not in general. Anyway, we keep the analysis general and study the limiting ancestral process also when the PIM assumption is not met, as the processes we obtain have an interesting property (duality) which we discuss in the next section. 

To construct the ancestral history of a sample in the limit under strong selection, we naturally start with no virtual lineages. 
If there are unfit alleles in the sample, 
considering \eqref{eq:rates_limits_scaled}, 
proceed by coalescing two unfit $i$-lineages, $i=2,\dots,d$, at rate  
\begin{align*}
	\frac{1}{2} \frac{n_i(n_i-1)}{n_i  -1 +\theta P_{1i}} ,
\end{align*}
and by mutating an unfit $i$-lineage, $i=2,\dots,d$, into a fit lineage at rate
\begin{align*}
	\frac{1}{2} \frac{n_i\theta P_{1i}}{n_i  -1 +\theta P_{1i}}  ,
\end{align*}
until no unfit lineage is left. 
This process acts on a faster timescale, as its rates are obtained by rescaling time. %, and in fact its two-alleles version is called fast process by \cite{wakeley2008}. 
The process corresponds to a Kingman coalescent on the unfit lineages. Because these unfit lineages reside within a population that is overwhelmingly comprised of the fit type, the only mutation events of relevance on this timescale are (backwards in time) mutations back to the fit type. %being equivalent to mutation to a new type in the infinite-allele mutation models. 

Once there are no unfit lineages left, and still no virtual lineages, considering that \eqref{eq:rates_limits} applies, now proceed on the original timescale by coalescing two fit lineages at rate 
$
\frac{n_1(n_1-1)}{2},
$
and by introducing a virtual unfit $i$-lineage $i=2,\dots,d$, at rate 
$
\frac{n_1 \theta P_{1i}}{2} , 
$
until a virtual lineage is created. 
%This process acts on the original timescale, which is why its two-allele version is called the slow process by \cite{wakeley2008}.
The newly created virtual lineage of type $i$ is immediately removed: reverting to the fast timescale,  the removal happens at rate $\frac{1}{2}$, obtained from \eqref{eq:rates_limits_scaled} with $\n=n_1\e_1$ and $\nuu=\e_i$.
Therefore, the addition of virtual lineages, which would be immediately be removed, can be entirely omitted on the slow timescale and the slow process simply corresponds to the block-counting process of a Kingman coalescent on the fit lineages, which is simply a death process with death rate  $
\frac{n_1(n_1-1)}{2}.
$
%given by the first entry in \eqref{eq:rates_limits}.

This interpretation of the ancestral process in the large-selection limit tells us what are the unlikely events that lead to the presence of any unfit alleles in the sample.
For example, with probability of order $\frac{1}{\sigma}$ we observe one unfit allele in the sample which must have come through a recent (most recent step back in the genealogical process) mutation from the fit type. With probability of order $\frac{1}{\sigma^2}$ we observe a sample containing two copies of unfit alleles. If the two unfit copies are of the same type,  
they must have come from a recent coalescence (one step back) and a mutation (two steps back) from the fit type; if they are of different types, they must have come from two recent (last two steps back) independent mutations from the fit type. The  probabilities are approximated by  the leading term $\tilde{q}_0(\n)\frac{1}{\sigma^{\norm{\n}-n_1}}$  of the asymptotic expansion \eqref{eq:asymptotic_expansion} of the sampling probability described in Section  \ref{sect:sampling_prob}.

\section{Duality}
\label{sect:duality}

It is well known \cite{barbour2000,Etheridge2009,Favero2021} that there is a normalised moment  duality relationship between the WF diffusion  $\X\sth$ of Section \ref{sect:diffusions} and the block-counting process of the (unreduced) conditional ASG of Section \ref{sect:genealogy}, which we denote here by   $\N\sth=\{\N\sth(t)\}_{t \geq 0 }$. 
This means that, 
for all $\x\in\Delta, \n,\nuu \in \mathbb{N}^d$,
\begin{align*}
	\E{H\sth(\X\sth(t),\n+\nuu)\mid \X\sth(0)=\x}
	= 
	\E{H\sth(\x,\N\sth(t))\mid \N\sth(0)=(\n,\nuu)},
\end{align*}
where $H\sth$ is the normalised moment duality function 
$$
H\sth(\x,\n+\nuu)
=\frac{F(\x,\n+\nuu)}{\E{F(\X^{(\sigma)}(\infty),\n+\nuu )}}
=  \frac{1}{q\sth(\n+\nuu)} \prod_{i=1}^d x_i^{n_i+\nu_i }  ,
$$
with $F$ being the moment duality function defined in \eqref{eq:sampling_diffusion} and $q\sth$ being the sampling probability of Section \ref{sect:sampling_prob}. 
In other words, one can transform the expectation with respect to $\X\sth$ on the left-hand side, which evolves forwards in time from $\x$, into an expectation with respect to $\N\sth$ on the right-hand side, which evolves backwards in time from $(\n,\nuu)$, by means of the duality function $H\sth$.

Note that, in the duality above,  the duality function groups together real and virtual lineages to match the ancestral process to the diffusion which does not differentiate between the notions of real and virtual.
However, when the conditional ASG is reduced,  as described in the previous section, 
real and virtual lineages are treated differently and
the distinction between real and virtual lineages becomes crucial, thus the duality above cannot be carried over to the reduced process. 
Nevertheless, we show in this section that the `lost' duality can be recovered in the strong selection limit.

\subsection{Asymptotic duality}

Denote by $\M=\{\M(t)\}_{t\geq 0}$ the fast process on $\n$ taking values in $\mathbb{N}^d$, when $\nuu = \boldsymbol{0}$, as defined in \eqref{eq:rates_limits_scaled}, or equivalently defined by its generator 
\begin{align}
	\label{eq:generator_fast_ancestral_limit}
	\mathcal{G} h(\n)=
	\sum_{i=2}^d
	\frac{1}{2}\frac{n_i(n_i-1)}{n_i -1 +\theta P_{1i}}
	\left[h(\n-\e_i)-h(\n)\right]
	+
	\sum_{i=2}^d
	\frac{1}{2}\frac{n_i \theta P_{1i}}{n_i  -1 +\theta P_{1i}}
	\left[h(\n-\e_i+\e_1)-h(\n)\right].
\end{align}
This process acts when there are some unfit lineages and no virtual lineages. Because of the lack of virtual lineages, which are the obstacle to duality, it seems it might be possible to find a process that is dual to $\M$. Uncovering it is the goal of this section. 
In fact, we prove below that the dual process is precisely the limiting CBI diffusion of Section \ref{sect:diffusions}. 

First, we derive our candidate duality function.  
Because the original pre-limiting duality function $H\sth$ contains $\X\sth(\infty)$, we can  guess from Section \ref{sect:stationary} that the useful scaling is the one that led to the CBI diffusion, which does not become degenerate at stationarity (unlike the Gaussian diffusion), i.e.\ $\x=\sigma(\z-\e_1)$. Then
\begin{align*}
	F(\x,\n)
	=
	\left(1+\frac{1}{\sigma}z_1\right)^{n_1}
	\prod_{i=2}^d \left(\frac{1}{\sigma}z_i\right)^{n_i}
	= 
	\frac{1}{\sigma^{\norm{\n}-n_1}}
	\left\{\prod_{i=2}^d z_i^{n_i}
	+\mathcal{O}\left(\frac{1}{\sigma}\right)\right\},
\end{align*}
as in the proof of Proposition \ref{thm:gamma_sampling}, 
and, by the asymptotic expansion of Section \ref{sect:sampling_prob} for 
$
q\sth(\n),
%\E{F(\X^{(\sigma)}(\infty),\n )}
%=
%\frac{1}{\sigma^{\norm{\n}-n_1}}
% \left\{ \tilde{q}_0(\n)
%\prod_{i=2}^d \frac{\Gamma(\theta P_{1i}+n_i)}{\Gamma(\theta P_{1i})}
% +\mathcal{O}\left(\frac{1}{\sigma}\right)\right\}
$
we obtain our candidate duality function as a limit of the duality function $H\sth$; that is,
\begin{align}
	\label{eq:duality_fct_limit}
	H(\z,\n)=  
	\frac{1}{ \tilde{q}_0(\n)}
	%\prod_{i=2}^d \frac{\Gamma(\theta P_{1i})}{\Gamma(\theta P_{1i}+n_i)} 
	\prod_{i=2}^d z_i^{n_i}
\end{align}
where $\tilde{q}_0(\n)$,  defined in \eqref{eq:zero_coefficient}, is the leading order non-zero coefficient of the sampling probability which we have already seen to be consistent with the leading order of its Gamma approximation under the CBI diffusion. 
%At this point, the natural candidate for the dual process is the CBI limit of Section \ref{sect:diffusions}, which proves to be correct as confirmed by the following. 
Now it is possible to prove the following. 

\begin{theorem}
	\label{thm:duality}
	The ancestral process $\M$ with infinitesimal generator $\mathcal{G}$, defined in \eqref{eq:generator_fast_ancestral_limit},
	and   the diffusion $\Z$ with infinitesimal  generator  $\tilde{\mathcal{L}}_0$, defined in \eqref{eq:generator_CBI},
	are dual with respect to the duality function $H$, defined in \eqref{eq:duality_fct_limit}. That is, 
	for all $\z\in \mathbb{R}_{\leq 0 } \times \mathbb{R}^{d-1}_{\geq 0 }$  with $ \sum_{i=1}^dz_i=0 $, and for all $\n\in\mathbb{N}^d$,
	\begin{align*}
		\E{H(\Z(t),\n)\mid \Z(0)=\z}
		= 
		\E{H(\z,\M(t))\mid \M(0)=\n}.
	\end{align*}
	
\end{theorem}

\begin{proof}
The proof is straightforward since we have already derived the correct candidate duality function. 
By e.g.\ \cite[Prop. 1.2]{Jensen2014} or \cite[p.115]{liggett2010}, it is enough  to verify the duality relationship on the corresponding generators, i.e.\  
    \begin{align}
    \mathcal{G} H(\z, \cdot) (\n)
    &=
    \tilde{\mathcal{L}}_0 H(\cdot, \n) (\z), \label{eq:duality}
    \end{align}
where $H(\z, \cdot)$  and   $H(\cdot, \n) $ are the projections of $H$ on  $\mathbb{N}^d$  and $\mathbb{R}_{\leq 0 } \times \mathbb{R}^{d-1}_{\geq 0 }$,  which belong to the domains of the generators $\mathcal{G}$ and $\tilde{\mathcal{L}}_0$,  respectively.
Using that
    \begin{align*}
    H(\z, \n-\e_i)= H(\z,\n-\e_i+\e_1)= \frac{n_i -1 +\theta P_{1i}}{z_i} H(\z,\n), 
    \quad i= 2,\dots, d ,
    \end{align*}
we obtain 
    \begin{align*}
    \mathcal{G} H(\z, \cdot) (\n)
    &=
    \sum_{i=2}^d
    \left[ 
    \frac{1}{2}\frac{n_i(n_i-1)}{n_i -1 +\theta P_{1i}} 
    +\frac{1}{2} \frac{n_i \theta P_{1i}}{n_i  -1 +\theta P_{1i}}
    \right]
    \left[ \frac{n_i -1 +\theta P_{1i}}{z_i} -1 \right] H(\z,\n)
    \\
    &=
    \frac{1}{2}
    \sum_{i=2}^d
    \frac{n_i}{z_i}  (n_i  -1 +\theta P_{1i} -z_i )
    H(\z,\n). 
    \\
    \end{align*}
Furthermore, noting that
    \begin{align*}
    \frac{\partial}{\partial z_i } H(\z,\n)= 
    \frac{n_i}{z_i} H(\z,\n), \quad 
    \frac{\partial^2}{\partial z_i^2 } H(\z,\n)=
    \frac{n_i(n_i-1)}{z_i^2} H(\z,\n),
    \quad i= 2,\dots, d,
    \end{align*}
we obtain
    \begin{align*}
    \tilde{\mathcal{L}}_0H(\cdot, \n) (\z)
    &=
    \frac{1}{2}
    \sum_{i=2}^d z_i  \frac{\partial^2}{\partial z_i^2 } H(\z,\n)
    +  \sum_{i=2}^d  \frac{1}{2}[\theta P_{1i} -z_i ] \frac{\partial}{\partial z_i } H(\z,\n)
    \\
    &=
    \frac{1}{2}
    \sum_{i=2}^d \frac{n_i(n_i-1)}{z_i} H(\z,\n)
    + \frac{1}{2} \sum_{i=2}^d \frac{n_i}{z_i} [\theta P_{1i} -z_i ]  H(\z,\n)
    \\
    &=
    \frac{1}{2}
    \sum_{i=2}^d
    \frac{n_i}{z_i}  (n_i  -1 +\theta P_{1i} -z_i )
    H(\z,\n),
    \end{align*}
thus verifying \eqref{eq:duality}.
\end{proof}

\begin{remark}
	Since the components, except for the first one, of the dual processes are independent, i.e.\ the CBI components $Z_2,\dots,Z_d$ are independent of each other and the components $M_2,\dots,M_d$ of the ancestral process are independent of each other, the duality holds component-wise.
	More precisely, for $i=2,\dots,d$, $Z_i$ and $M_i$ are dual with respect to the duality function 
	\begin{align*}
		H_i(z_i,n_i)=
		\frac{\Gamma(\theta P_{1i})}{\Gamma(\theta P_{1i}+n_i)} z_i^{n_i}.
	\end{align*}
	That is,  for all $z_i\in \mathbb{R}_{\geq 0 }$, and for all  $n_i\in\mathbb{N}$, $i=2,\dots, d$,
	\begin{align*}
		\E{H_i(Z_i(t),n_i)\mid Z_i(0)=z_i}
		= 
		\E{H(z_i,M_i(t))\mid M_i(0)=n_i}.
	\end{align*}
	Note that $M_i$ is a pure death process and its death rate, from \eqref{eq:generator_fast_ancestral_limit}, is linear. A different dual for $Z_i$ is developed in \cite{Ruggiero2014} in which one component of the dual process is a birth-death process also with linear death rate.
\end{remark}
To complete the asymptotic duality framework, now consider the slow process on $\n$ taking values in $\mathbb{N}^d$, as defined in \eqref{eq:rates_limits}, which corresponds to a death process with rates $\frac{n_1(n_1-1)}{2}$ and acts when there are no virtual nor unfit lineages. In this case, the duality function $H$ becomes constant, equal to $1$.
With respect to this constant duality function, trivially, the slow process is dual to the process which is fixed at the equilibrium point $\e_1$ and which has infinitesimal generator equal to $0$. This is a very special case of the well-known duality between Kingman's coalescent and the WF diffusion, in which the population is composed of only one type.

Finally, it is interesting to notice that the asymptotic dualities of this section hold in general, without need for the PIM assumption. In the pre-limit both the duality and (in the non-PIM case) equivalence with the true genealogical distribution are lost with the reduction of the graph. We showed that in the strong selection limit this duality is restored, whereas it is not straightforward to see whether the equivalence with the true genealogical distribution is also restored. This question would require further, non-trivial, analysis and could benefit from the results of this and the previous section, which motivates  why we considered the reduced graph for general mutation models without assuming PIM.

%%%%%%%%%%%%%%%%%%%%%%%%%%%%%%%%%%%%%%%%%%%%%%%%%%%%%%%%%%%%%%%%%%%
%%                                                               %%
%% Use the two commands below for producing your bibliography    %%
%% with bibtex, then comment again the commands and include the  %%
%% content of the .bbl file in this file below the commands.     %%
%%                                                               %%
%%%%%%%%%%%%%%%%%%%%%%%%%%%%%%%%%%%%%%%%%%%%%%%%%%%%%%%%%%%%%%%%%%%
%also uncomment the usepackage natbib in the beginning
%change \citep,\citet to \cite

%\bibliographystyle{amsplain}
%\bibliography{mybib.bib}

% add below the content of your .bbl file produced by bibtex.

%\providecommand{\bysame}{\leavevmode\hbox to3em{\hrulefill}\thinspace}
%\providecommand{\MR}{\relax\ifhmode\unskip\space\fi MR }
%% \MRhref is called by the amsart/book/proc definition of \MR.
%\providecommand{\MRhref}[2]{%
%	\href{http://www.ams.org/mathscinet-getitem?mr=#1}{#2}
%}
%\providecommand{\href}[2]{#2}

%%%%%%%%%%%%%%%%%%%%%%%%%%%%%%%%%%%%%%%%%%%%%%%%%%%%%%%%%%%%%%%%%%%
%%                                                               %%
%% You may add acknowledgments (optional).                       %%
%%                                                               %%
%%%%%%%%%%%%%%%%%%%%%%%%%%%%%%%%%%%%%%%%%%%%%%%%%%%%%%%%%%%%%%%%%%%
\begin{acks}
We would like to thank Jere Koskela for valuable discussions in the preliminary phase of this work. 
We are grateful to the two anonymous reviewers whose thorough comments improved the quality of this paper. 
MF acknowledges the support of the Knut and Alice Wallenberg Foundation (Program for Mathematics, grant 2020.072). 
\end{acks}

\appendix
\section{Appendix $\cdot$ Limits of Wright--Fisher models}

It is well known that the WF diffusion arises as the limit of a sequence of WF models as the population size grows to infinity, under some assumptions guaranteeing that the strength of the genetic drift is comparable to, or stronger than, the strength of mutation and selection (if the genetic drift is weaker, a deterministic limit is obtained). Therefore, studying the asymptotic behaviour of the WF diffusion can also be achieved by taking a step back and studying the asymptotic behaviour of the corresponding pre-limiting WF model. While in the rest of the paper we work with the WF diffusion, here we focus on WF models and their convergence in order to facilitate the comparison with classical results in the literature, e.g.\ \cite{Ethier1988,Feller1951,Nagylaki1986,Nagylaki1990,Norman1972,Norman1974,Norman1975,Norman1975b}. Our aim is to bring together a few disparate convergence results into a single, self-contained reference.

\subsection{The Wright--Fisher model with selection and mutation}
%\label{sect:WF_model}
Consider a homogeneous, panmictic population of a fixed size $N$ evolving in discrete generations.  
Each individual carries one of a finite number $d$ of possible genetic types, or alleles. Reproduction takes place as follows: 
Each individual in a generation chooses an individual from the previous generation as their parent. The chosen parent is of type $i$ with a probability that is proportional to the frequency of individuals of type $i$ in the parental generation, multiplied by $1+s_i\Nth$. The parameter $s_i\Nth$ represents the selective advantage, if positive, or disadvantage, if negative, of an individual of type $i$. The offspring inherits type $i$ from the parent with probability $ u_{ii}\Nth $, or mutates to type $j$ with probability $ u_{ij}\Nth $. More concisely, letting  $\X\Nth (k)$ be the vector of frequencies in generation $k$,
\begin{align*}
	& N \X\Nth (k+1) \mid N \X\Nth (k)= N \x \sim \text{Multinomial}( N, q\Nth(\x) ),
	\\
	& \text{with} \quad q_i\Nth(\x)= \sum_{j=1}^d\tilde{q}_j\Nth(\x) u_{ji}\Nth,
	\quad \text{and} \quad
	\tilde{q}_i\Nth(\x)= \frac{(1+s_i\Nth)x_i}{\sum_{j=1}^d(1+s_j\Nth)x_j} .
\end{align*}
The WF model consists of the sequence of frequencies evolving forwards in time, $\{\X\Nth(k)\}_{k\in \mathbb{N}}$. See e.g.\ \cite[Ch.10]{Ethier1986}.

\subsection{Tools for convergence}
To study the convergence of $\{\X\Nth(k)\}_{k\in\mathbb{N}}$, as $N\to\infty$,
or of a general  (discrete-time and) time-homogeneous Markov chain,
the central objects of interest are the means and the covariances of its infinitesimal increments, as sketched in the following. 
Consider the semigroup $T\Nth$ of $\X\Nth$, i.e.\ 
$T\Nth f(\x) = \E{f(\X\Nth(k+1))\mid \X\Nth(k)=\x }$, and its discrete-time generator %since an infinitesimal generator is not available because time is discrete, consider instead the operator
$A\Nth= \alpha_N (T\Nth - I )$, where $\alpha_N$ is the time scaling (to be chosen appropriately and satisfying $\alpha_N\to\infty$ as $N\to\infty$) and $I$ is the identity operator. 
By classical results \cite[Thm 6.5, Ch 1 and Thm 2.6, Ch 4]{Ethier1986}, 
proving that $A\Nth$ converges to  $A$ as $N\to\infty$
(in the sense of \cite[Thm 6.5, Ch 1]{Ethier1986})
implies  
that $(T\Nth)^{\lfloor \alpha_N t \rfloor}\to T(t)$ 
and thus 
%(by \cite[Thm 2.12, Ch 4]{Ethier1986}) 
that $\{\X\Nth(\lfloor \alpha_N t \rfloor)\}_{t \geq 0} \xrightarrow[]{d} \{\X( t )\}_{t \geq 0}$, where 
$\X$ is a continuous-time Markov process with infinitesimal generator $A$ and semigroup $T$.
By informally using Taylor's approximation,
\begin{align*}
	A\Nth f(\x) ={} & \alpha_N \E{f(\X\Nth(k+1))- f(\x) \mid \X\Nth(k)=\x }
	\\
	\approx {}&
	\alpha_N  \left\{ 
	\sum_{i=1}^d \E{X_i\Nth(k+1)- x_i \mid \X\Nth(k)=\x }
	\frac{\partial}{\partial x_i} f(\x) \right.
	\\
	& {}+ \frac{1}{2} \left.
	\sum_{i,j=1}^d 
	\E{\left( X_i\Nth(k+1)- x_i \right) \left( X_j\Nth(k+1)- x_j \right)  \mid \X\Nth(k)=\x }
	\frac{\partial^2}{\partial x_i \partial x_j} f(\x)
	\right\}
	\\
	& \xrightarrow[N\to\infty]{}  A f(\x) =
	\sum_{i=1}^d  b_{i}(\x)
	\frac{\partial}{\partial x_i} f(\x) 
	+ \frac{1}{2} 
	\sum_{i,j=1}^d a_{ij}(\x)
	\frac{\partial^2}{\partial x_i \partial x_j} f(\x),
\end{align*}
where 
\begin{align*}
	b_{i}(\x)  & = 
	\lim_{N\to \infty} \alpha_N \E{X_i\Nth(k+1)- x_i \mid \X\Nth(k)=\x }
	=
	\lim_{N\to \infty} \alpha_N  \Ex{ \Delta X_i\Nth } 
	,
	\\
	a_{ij}(\x) &=
	\lim_{N\to \infty} \alpha_N
	\E{\left( X_i\Nth(k+1)- x_i \right) \left( X_j\Nth(k+1)- x_j \right)  \mid \X\Nth(k)=\x }
	\\ &=  \lim_{N\to \infty} \alpha_N
	\Ex{ \Delta X_i\Nth \Delta X_j\Nth }
	,
\end{align*}
with $\Ex{ \Delta X_i\Nth }$ and $\Ex{ \Delta X_i\Nth \Delta X_j\Nth } $ denoting the respective conditional mean and covariance of the infinitesimal increments $\Delta X_i \Nth := X_i\Nth (k+1) - X_i\Nth(k)$, $i=1,\dots,d,$ conditioned on the current value $\X\Nth(k)=\x$. 
While the reasoning above only provides a sketch of a convergence proof, it shows  that these means and covariances are the  quantities of interest, 
% \begin{align*}
	% &  \Ex{ \Delta X_i\Nth } 
	% =\E{X_i\Nth(k+1)- x_i \mid \X\Nth(k)=\x }
	% \\
	% & \Ex{ \Delta X_i\Nth \Delta X_j\Nth } 
	% = \E{\left( X_i\Nth(k+1)- x_i \right) \left( X_j\Nth(k+1)- x_j \right)  \mid \X\Nth(k)=\x }
	% \end{align*}
which we use in the rest of this section to recap several (well- or lesser-) known results concerning convergence of WF models.

\subsection{Convergence: deterministic limit vs Wright--Fisher diffusion }   
%\label{sect:appendix_deterministic_limit}
Consider a sequence $\{\X\Nth(k)\}_{k\in \N}$ of WF models, 
%as defined in Section \ref{sect:WF_model}, 
each with finite population size $N$, and mutation and selection parameters given for $i,j=1,\dots,d$ by
\begin{align*}
	u_{ij}\Nth &= \epsilon\Nth \frac{\theta}{2} P_{ij}, \quad i\neq j; & u_{ii}\Nth &=1- \epsilon\Nth \frac{\theta}{2} (1-P_{ii});
	\\
	s_i\Nth &= \epsilon_i\Nth \frac{\sigma_i}{2};
\end{align*}
where the  scaling sequences $\epsilon\Nth, \epsilon\Nth_i$ are nonnegative and  decreasing to zero but yet to be defined. 
Depending on the scaling chosen, two kinds of limits for the sequence $\{\X\Nth(\lfloor \alpha_N t \rfloor)\}_{t \geq 0} $ can be obtained: (a) deterministic limits, including an analogy of Proposition \ref{thm:diffusion_limits}(a), or (d) WF diffusion limits, as  illustrated in the following. 

To ease the reading, we use the classical big $\mathcal{O}$ notation, and,  
for sequences $b_N,b'_N\to 0$, we write $b_N \prec b'_N $ to indicate that $b_N=  o(b'_N)$, i.e.\ $\lim_{N\to\infty}\frac{b_N}{b'_N}=0$.
Let us now compute the means and the covariances of the infinitesimal increments of this sequence of models.

\subsubsection*{Means of infinitesimal increments}
The means of the infinitesimal increments of the  $\alpha_N$-timescaled WF model can be expressed as  
\begin{align*}
	\alpha_N \Ex{ \Delta X_i\Nth } 
	&= 
	\alpha_N \left\{ \E{X_i\Nth(k+1)\mid \X\Nth(k)=\x }
	-x_i \right\}
	\\ &=
	\alpha_N \left\{q_i\Nth(\x) -x_i\right\}
	\\
	&= \alpha_N 
	\frac{\sum_{j=1}^d (1+s\Nth_j)x_ju_{ji}\Nth - 
		x_i\sum_{j=1}^d (1+s\Nth_j)x_j}
	{\sum_{j=1}^d (1+s\Nth_j)x_j}
	\\
	%& \approx \alpha_N  \left\{
	% \sum_{j=1}^d (1+s\Nth_j)x_ju_{ji}\Nth - 
	%x_i\sum_{j=1}^d (1+s\Nth_j)x_j \right\}
	%\\
	%& =
	%\epsilon\Nth  
	%\frac{\theta}{2}  \left( \sum_{j=1}^d P_{ji}x_j - x_i   \right)
	%+
	%\frac{x_i}{2}  \left(\epsilon\Nth_i \sigma_i - \sum_{j=1}^d \epsilon\Nth_j %\sigma_j x_j \right)
	%+  \epsilon\Nth \frac{\theta}{2}\sum_{j=1}^d  (P_{ji}-%\delta_{ij})\epsilon_j\Nth \frac{\sigma_j}{2}x_j
	%\\
	%& \approx 
	%\epsilon\Nth  
	%\frac{\theta}{2}  \left( \sum_{j=1}^d P_{ji}x_j - x_i   \right)
	%+
	%\frac{x_i}{2}  \left(\epsilon\Nth_i \sigma_i - \sum_{j=1}^d \epsilon\Nth_j \sigma_j x_j \right)
	%+  \sum_{j=1}^d \mathcal{O}(\epsilon\Nth %\epsilon_j\Nth)
	%\\
	& = \alpha_N \epsilon\Nth  
	\frac{\theta}{2}  \left( \sum_{j=1}^d P_{ji}x_j - x_i   \right)
	+
	\frac{x_i}{2}  \left(\alpha_N \epsilon\Nth_i \sigma_i - \sum_{j=1}^d \alpha_N \epsilon\Nth_j \sigma_j x_j \right)
	\\ & \quad {}+  \sum_{j=1}^d \mathcal{O}(\alpha_N \epsilon\Nth \epsilon_j\Nth). 
\end{align*}
Then,
\begin{align}
	\label{inf_drift}
	\lim_{N\to \infty}\alpha_N \Ex{ \Delta X_i\Nth } 
	&=\begin{cases}
		\text{(i) } & \mu_i(\x)+s_i(\x)= 
		\frac{\theta}{2}  \left( \sum_{j=1}^d P_{ji}x_j - x_i   \right)
		+
		\frac{x_i}{2}  \left(\sigma_i - \sum_{j=1}^d   \sigma_j x_j \right)
		\\
		& \text{if } \alpha_N^{-1}= 
		\epsilon\Nth = \epsilon_i\Nth,\,  i=1,\dots,d;
		\\
		\text{(ii) } & \sigma_1 \frac{x_1}{2} (\delta_{1 i } -  x_i )\\
		& \text{if } \alpha_N^{-1}= \epsilon_1\Nth \succ
		\epsilon\Nth, \epsilon_i\Nth,\,  i=2,\dots,d.
	\end{cases}     
\end{align}
Other finite limits for the drift above can be obtained as long as the strongest parameters behave like $\alpha_N^{-1}$, while the others are  smaller (the forces corresponding to the smaller parameters vanish in the limit). If all parameters are smaller than $\alpha_N^{-1}$ then the limit above is zero, whereas if some parameters are bigger than $\alpha_N^{-1}$ then the limit is infinite. From now on we assume the limit above is finite.
\subsubsection*{Covariances of infinitesimal increments}
The covariances of the infinitesimal increments of the  $\alpha_N$-timescaled WF model can be expressed as 
    \begin{align*}
    & \alpha_N \Ex{ \Delta X_i\Nth \Delta X_j\Nth } 
%    \\ & \quad = 
%    \alpha_N \bigg\{ 
%    \cov\left(X_i\Nth(k+1)X_j\Nth(k+1)\mid \X\Nth(k)=\x \right)
%    + x_ix_j +q_i\Nth(\x) q_j\Nth(\x) 
%    \\ & \quad  \quad - x_i \E{X_j\Nth(k+1)\mid \X\Nth(k)=\x }
%    - x_j \E{X_i\Nth(k+1)\mid \X\Nth(k)=\x } \bigg\}
\\ & \quad = 
   \alpha_N 
\cov\left(X_i\Nth(k+1)X_j\Nth(k+1)\mid \X\Nth(k)=\x \right)
+ \alpha_N (q_i\Nth(\x) -x_i)(q_j\Nth(\x) -x_j) 
%\\ & \quad  \quad 
%+(q_i\Nth(\x) - x_i) \E{X_j\Nth(k+1)-q_i\Nth(\x)\mid \X\Nth(k)=\x }
%\\ & \quad  \quad 
%+(q_j\Nth(\x) - x_j ) \E{X_i\Nth(k+1)-q_j\Nth(\x)\mid \X\Nth(k)=\x } \bigg\}
\\ & \quad =
\frac{\alpha_N}{N}q_i\Nth(\x) 
(\delta_{ij} -q_j\Nth(\x))
+\alpha_N (q_i\Nth(\x) -x_i)(q_j\Nth(\x) -x_j)
    \\ &  \quad=
    \frac{\alpha_N}{N} x_i
    (\delta_{ij} -x_j) 
    + \mathcal{O}\left( \frac{\alpha_N}{N}\epsilon\Nth\right)
    + \sum_{k=1}^d\mathcal{O}\left(\frac{\alpha_N}{N}\epsilon_k\Nth\right)
    + \alpha_N
    \Ex{ \Delta X_i\Nth } \Ex{ \Delta X_j\Nth }
    \\ & \quad  = 
    \frac{\alpha_N}{N} x_i
    (\delta_{ij} -x_j) 
    + \mathcal{O}\left( \frac{\alpha_N}{N}\epsilon\Nth\right)
    + \sum_{k=1}^d\mathcal{O}\left(\frac{\alpha_N}{N}\epsilon_k\Nth\right)
    +  \mathcal{O}(\epsilon\Nth )+
    \sum_{k=1}^d \mathcal{O}( \epsilon_k\Nth),
    \end{align*}
where in the last step $\alpha_N \Ex{ \Delta X_i\Nth } \Ex{ \Delta X_j\Nth }$ converges to zero because we are assuming that $\alpha_N \Ex{ \Delta X_i\Nth }$ has a finite limit. Thus,
\begin{align}
	\label{inf_covariance}
	\lim_{N\to \infty}\alpha_N \Ex{ \Delta X_i\Nth\Delta X_j\Nth } 
	&=\begin{cases}
		\text{(a) } & 0 \quad \text{(deterministic limit)}\\
		& \text{if } \alpha_N^{-1}\succ \frac{1}{N}; \\
		\text{(d) } & x_i(\delta_{ij} -x_j) \quad \text{(WF diffusion)}\\
		& \text{if }  \alpha_N^{-1}=\frac{1}{N};\\
		& \infty \quad \text{(explosion)}\\
		& \text{if } \alpha_N^{-1} \prec \frac{1}{N}.
	\end{cases} 
\end{align}
From now on we assume the limit above is finite. 
In order to have finite limits in both \eqref{inf_drift} and  \eqref{inf_covariance},
the timescale $\alpha_N$ must be chosen to be growing inversely to the decay of the largest parameter, and the latter decay should be equal to, or larger than, $N^{-1}$.
The case when it is equal, $\alpha_N = N^{-1}$, and the WF diffusion limit is obtained, can be referred to as weak mutation or selection, since these latter processes operate on a timescale no faster than genetic drift. On the other hand, the case when it is larger, $\alpha_N^{-1}\succ \frac{1}{N}$, and a deterministic limit is obtained, can be referred to as strong mutation or selection. The strong selection regime we consider in this paper might seem different because we start with a WF diffusion which is obtained as a limit of WF models with weak mutation and selection, but it is in fact analogous because we then let the first selection parameter $\sigma_1$ of the diffusion go to infinity.

We now consider each valid combination of limiting cases for the infinitesimal mean and infinitesimal covariance.

\subsubsection*{Case (a, i) $\left[ \frac{1}{N} \prec \alpha_N^{-1}= 
	\epsilon\Nth = \epsilon_i\Nth, \; i=1,\dots,d \right]$}

All mutation and selection parameters have the same strength. Mutation and selection are stronger than genetic drift, so a deterministic limit is obtained; 
that is, 
$\{\X\Nth(\lfloor \alpha_N t \rfloor)\}_{t \geq 0} \to \{\boldsymbol{\chi}( t )\}_{t \geq 0}$ as $N\to\infty$, where $\boldsymbol{\chi}$ is the solution to the ODE
\begin{align*}
	\frac{d}{dt} \chi_i (t)= \mu_i(\boldsymbol{\chi}(t))+s_i(\boldsymbol{\chi}(t)) =
	\frac{\theta}{2}  \left( \sum_{j=1}^d P_{ji}\chi_j(t) - \chi_i (t)  \right)
	+
	\frac{\chi_i(t)}{2}  \left(\sigma_i - \sum_{j=1}^d   \sigma_j \chi_j (t)\right).
\end{align*}
Assuming neutrality and PIM, %In the neutral case, if mutations are parent independent, i.e.\ $P_{ij}=Q_j,  i,j=1,\dots,d$, 
the above ODE becomes
\begin{align*}
	\frac{d}{dt} \chi_i (t)= 
	\frac{\theta}{2}  \left( Q_i - \chi_i (t)  \right), \quad i=1,\dots,d,
\end{align*}
with explicit solution
\begin{align*}
	\chi_i(t)= Q_i + (\chi_i(0)-Q_i)e^{-\frac{\theta}{2}t}, \quad i=1,\dots,d,
\end{align*}
and $\chi_i(\infty)=\lim_{t\to\infty}\chi_i(t)=Q_i$.

\subsubsection*{Case  (a, ii)
	$\left[\frac{1}{N} \prec \alpha_N^{-1}= \epsilon_1\Nth \succ
	\epsilon\Nth, \epsilon_i\Nth, \; i=2,\dots,d, \; \sigma_1=1 \right] $}

The selection parameter of the first allele is stronger than the selection parameters of the other alleles and of all the mutation parameters. Selection on the first allele is stronger than genetic drift, so a deterministic limit is obtained; that is, 
$\{\X\Nth(\lfloor \alpha_N t \rfloor)\}_{t \geq 0} \to \{\xii( t )\}_{t \geq 0}$ as $N\to\infty$, where the deterministic $\xi$ is the solution to the ODE \eqref{eq:ODE_deterministic_trajectory}, with equilibrium point $\xii(\infty)=\e_1$.
%\begin{align*}
%\frac{d}{dt} \xi_i (t)= 
%\frac{\xi_1(t)}{2} (\delta_{1 i } -  \xi_i(t) );
%\quad\quad 
% \xi_i (t)= 
% \begin{cases}
	%  \frac{\xi_1(0)}{ \xi_1(0)+ e^{-t}(1-\xi_1(0))}, &\quad\text{for } i=1; 
	% \\
	%  \frac{\xi_i(0)}{ e^{t}\xi_i(0)+ 1-\xi_i(0)}, &\quad\text{for } i\neq 1;
	% \end{cases}
% \to e_1 , \text{as  } t\to\infty
%\end{align*}
Obviously, this case corresponds to setting $\sigma_i=0$ for $i=2,\dots d$, in  Case (a, i). 
Furthermore, this deterministic limit can also be obtained by first obtaining a WF diffusion, as in Case (d, i) below, and subsequently letting $\sigma_1\to\infty$ (while rescaling time by $\sigma_1$), precisely as in Proposition \ref{thm:diffusion_limits}(a).  

\subsubsection*{Case (d, i) $ \left[\frac{1}{N} = \alpha_N^{-1}= 
	\epsilon\Nth = \epsilon_i\Nth, \;  i=1,\dots,d  \right] $}

All mutation and selection parameters have the same strength.
Mutation and selection are comparable to genetic drift, so the WF diffusion of Section \ref{sect:diffusions} is obtained; that is, 
$\{\X\Nth(\lfloor \alpha_N t \rfloor)\}_{t \geq 0} \xrightarrow[]{d} \{\X( t )\}_{t \geq 0}$ as $N\to\infty$ where $\X$ is the diffusion with drift terms $\mu_i(\x) + s_i (\x) $ and the typical WF diffusion coefficients $x_i (\delta_{ij} -x_j)$, 
as defined in \eqref{generatorWFdiff}.

\subsubsection*{Case  (d, ii) $ \left[ \frac{1}{N} = \alpha_N^{-1}= \epsilon_1\Nth \succ  \epsilon\Nth, \epsilon_i\Nth , i= 2,\dots,d; \right] $}
The selection parameter of the first allele is stronger than the selection parameters of the other alleles and of all the mutation parameters. Selection on the first allele is comparable to genetic drift, so a WF diffusion is obtained which  
%Only the selection parameter of the first allele appears in the drift coefficients, $ \frac{x_1}{2} (\delta_{1 i } -  x_i )$,  and the diffusion coefficients are the typical $x_i (\delta_{ij} -x_j)$. 
corresponds to setting $\sigma_i=0$ for $i=2,\dots,d$, and $\theta=0$ in the limiting WF diffusion of case (d, i).

\subsection{Fluctuations around the deterministic limit}

When a deterministic limit is obtained for the sequence of WF models, as in cases (a, i) and (a, ii), it is possible to derive asymptotic results for fluctuations around the deterministic limit or its discretised version. 
During the 70--80s, Norman and Nagylaki showed that these fluctuations can be approximated by Gaussian diffusions, e.g.\  \cite{Norman1972,Norman1975} and \cite{Nagylaki1986}. 
Their results cover our cases (a, i) and (a, ii), but involve lengthy calculations, complicated conditions to ensure convergence and not very explicit expressions for the limiting quantities. See also Section \ref{sect:diffusions} for further comment. We refrain from repeating their calculations here and refer instead to Proposition \ref{thm:diffusion_limits}(b) which contains an analogous result for the case of interest (a, ii), starting from the WF diffusion rather than the corresponding WF model; this avoids conflating the asymptotics of the parameters of interest with the large population size limit of the discrete model and is based on an easier-to-read martingale approach. Our approach is in the spirit of a sequence of papers starting with \cite{kurtz1971} who obtained scaling-limit fluctuation results taking as the starting point a sequence of finite-population models \emph{already in continuous time}; see also \cite[Ch.\ 11]{Ethier1986} and \cite{kang2014} for more recent treatments.

\subsection{Fluctuations around the equilibrium point}
%\label{sect:appendix_fluct_equi}
Considering the equilibrium point at stationarity of the deterministic limit, assuming it exists, it is also possible to study the asymptotic behaviour of fluctuations around this equilibrium point, see e.g.\ \cite[Sect.\ 9]{Feller1951}. 
In several cases, the Gaussian approximations of the previous section extend to the stationary regime, see e.g.\ \cite{Norman1972,Norman1974,Norman1975b} and \cite{Nagylaki1986}, provided that certain conditions concerning the equilibrium point are satisfied.
However, the extension fails when the equilibrium point is on the boundary, as in our case of interest (a, ii), which requires a different treatment and a different scaling leading to a \emph{non-Gaussian} approximation. \cite{Ethier1988} and \cite{Nagylaki1990} obtain this limit though without phrasing it in terms of fluctuations. 
In the following, we illustrate how this  alternative spatial scaling leads to the non-Gaussian limiting diffusion owing to  the equilibrium point being on the boundary. This convergence of WF models is analogous to the convergence of WF diffusions in Proposition \ref{thm:diffusion_limits}(c).

\subsubsection*{Case (a, i) with PIM under neutrality} 
%\comm{\footnotesize Recap: $ \frac{1}{N} \prec \alpha_N^{-1}= 
	%        \epsilon\Nth $ , 
	%\quad $\X\Nth(\lfloor \alpha_N t \rfloor) \xrightarrow[N\to\infty]{} \chi( t )= Q + (\chi(0)-Q)e^{-\frac{\theta}{2}t} \xrightarrow[t\to\infty]{}Q $
	%}  \\
The neutral two-allele setting is studied in \cite[Sect.\ 9]{Feller1951}. Let us translate it into our notation and generalise it to the multi-allele, PIM setting. 
In order to study convergence of the sequence
$\Z\Nth(\lfloor \alpha_N t \rfloor)=a_N(\X\Nth(\lfloor \alpha_N t \rfloor)-\Q)$,
where $\Q$ is the equilibrium point and $a_N$ is a space scaling to be defined, 
we calculate the means of its infinitesimal increments 
\begin{align*}
	\alpha_N \mathbb{E}_{\z}\left[
	\Delta Z_i\Nth
	\right]  
	&= 
	\alpha_N a_N 
	\mathbb{E}_{a_N^{-1}\z+\Q}
	\left[  \Delta X_i\Nth
	\right]
	\approx 
	\alpha_N a_N\left\{  
	- a_N ^{-1} \epsilon\Nth  
	\frac{\theta}{2} z_i
	\right\}
	=
	- \frac{\theta}{2} z_i;
\end{align*} 
and the covariances
\begin{align*}
	%\label{inf_var_equilibrium_fluctuations_feller}
	\nonumber
	\alpha_N \mathbb{E}_{\z}\left[
	\Delta Z_i\Nth \Delta Z_j\Nth
	\right]  
	&= 
	\alpha_N a_N^2 
	\mathbb{E}_{a_N^{-1}\z+\Q}
	\left[  \Delta X_i\Nth \Delta X_j\Nth
	\right]
	\\ &=
	\frac{\alpha_N a_N^2 }{N}
	q_i\Nth(a_N^{-1}\z+\Q) 
	(\delta_{ij} -q_j\Nth(a_N^{-1}\z+\Q))
	\\ \nonumber
	&\approx
	\frac{\alpha_N a_N^2 }{N}
	Q_i (\delta_{ij}- Q_j)
	= 
	\frac{ a_N^2 }{\epsilon\Nth N}
	Q_i (\delta_{ij}- Q_j) .
\end{align*}
Therefore, by choosing the space scaling $a_N= \sqrt{ \epsilon\Nth N}$ and assuming $\Q$ has all its entries positive, 
the sequence
$\Z\Nth$
%(\lfloor \alpha_N t \rfloor)=a_N(\X\Nth(\lfloor \alpha_N t \rfloor)-\Q)$
converges to an Ornstein--Uhlenbeck process $\{\Z(t)\}_{t\geq 0}$ that solves the SDE
\begin{align*}
	d \Z(t)= -  \frac{\theta}{2} \Z(t)dt
	+ \Sigma^{1/2} d \W(t) ,
\end{align*}
where $\Sigma$ is a constant matrix with components $\Sigma_{ij}= Q_i(\delta_{ij}-Q_j)\in (0,1)$.
This is an example of the fact that, when the equilibrium point is in the interior of the domain, the scaling limit of the fluctuations around the equilibrium point is a Gaussian process. 

\subsubsection*{Case (a, i) with general mutation and selection} 
As long as the equilibrium point $\boldsymbol{\chi}(\infty)$ exists and is not on the boundary, (i.e.\ none of its components is zero),
from the previous calculations it follows that choosing again $a_N= \sqrt{ \epsilon\Nth N}$ yields an analogous result, which leads to a constant diffusion matrix $\Sigma_{ij}= \chi_i(\infty)(\delta_{ij}-\chi_j(\infty)) $. The drift term changes accordingly and the limit is still a Gaussian process.
%Instead,  if the equilibrium point is on the boundary, a different choice for the scaling is required as illustrated  in the following. 

\subsubsection*{Case (a, ii)}
%\comm{\footnotesize
	%Recap : 
	%$ \frac{1}{N} \prec \alpha_N^{-1}= %\epsilon_1\Nth \succ
	%        \epsilon\Nth, \epsilon_i\Nth \forall i\neq 1; \sigma_1=1$\\
	% $\{\X\Nth(\lfloor \alpha_N t \rfloor)\}_{t \geq 0} \to \{\xi( t )\}_{t \geq 0}$, where the deterministic $\xi$ is the solution to the ODE \eqref{eq:ODE_deterministic_trajectory}. $\xi(\infty)=\e_1$
	% \\
	%The timescale $\alpha_N$ is the inverse of the largest (slowest) parameter decay $\epsilon_1\Nth$, while the space scale $a_N$ is to be determined.
	%}\\
In this case, the equilibrium point is $\e_1$ which crucially is on the boundary. 
To study convergence of the sequence $\Z\Nth(\lfloor \alpha_N t \rfloor)=a_N(\X\Nth(\lfloor \alpha_N t \rfloor)-\e_1)$, with a space scaling $a_N$ to be defined, 
% $\sum_{i=1}^d Z\Nth_i(k)=0, \forall k, \forall N$.
we calculate the means and covariances of its infinitesimal increments. 
We omit the first component $Z_1$, without loss of information, as the components sum to zero. 
%So we can omit the first component of the sequence in the following, in order to have simpler calculations. 
For $i=2,\dots,d$,
\begin{align*}
	\alpha_N \mathbb{E}_{\z}\left[
	\Delta Z_i\Nth
	\right]  
	&= 
	\alpha_N \E{a_N X_i\Nth(k+1) -z_i\mid \X\Nth(k)=a_N^{-1}\z+e_1 }
	\\ &=
	\alpha_N a_N 
	\mathbb{E}_{a_N^{-1}\z+e_1}
	\left[  \Delta X_i\Nth
	\right]
	%\\ &=
	%\alpha_N a_N \left\{
	%q_i\Nth(a_N^{-1}\z+e_1) - a_N^{-1}z_i
	%\right\}
	\\
	& =
	\alpha_N \left\{     a_N \epsilon\Nth  
	\frac{\theta}{2} P_{1i}
	-  \epsilon_1\Nth 
	\frac{z_i}{2} 
	+
	\mathcal{O}(\epsilon\Nth)
	+ \mathcal{O}(\epsilon_i\Nth)
	+ \mathcal{O}(a_N^{-1}\epsilon_1\Nth)
	\right\}
	\\ & =
	-     \frac{z_i}{2} 
	+
	a_N \frac{\epsilon\Nth }{\epsilon\Nth_1} 
	\frac{\theta}{2} P_{1i}
	+ \mathcal{O}\left(\frac{\epsilon\Nth }{\epsilon\Nth_1}\right)
	+ \mathcal{O}\left(\frac{\epsilon_i\Nth }{\epsilon\Nth_1}\right)
	+ \mathcal{O}(a_N^{-1}),
\end{align*} 
%and thus, 
%    \begin{align*}
	%    \lim_{N\to \infty}
	%    \alpha_N \mathbb{E}_{\z} 
	%    \left[
	%      \Delta Z_i\Nth  \right]  
	%    &=\begin{cases}
		%    \text{ if}  
		%    & a_N\prec \frac{\epsilon_1\Nth}{\epsilon\Nth}, \quad -\frac{z_i}{2};
		%    \\
		%    \text{ if } 
		%    & a_N = \frac{\epsilon_1\Nth}{\epsilon\Nth}, \quad 
		%    -\frac{z_i}{2} + \frac{\theta}{2} P_{1i} ;
		%    \\
		%    \text{ if }  
		%    & a_N\succ \frac{\epsilon_1\Nth}{\epsilon\Nth}, \quad \infty .
		%    \end{cases}     
	%    \end{align*}
and, for $i,j=2,\dots,d $,
\begin{align*}
	%\label{inf_var_equilibrium_fluctuations}
	\nonumber
	\alpha_N \mathbb{E}_{\z}\left[
	\Delta Z_i\Nth \Delta Z_j\Nth
	\right]  
	&= 
	\alpha_N a_N^2 
	\mathbb{E}_{a_N^{-1}\z+e_1}
	\left[  \Delta X_i\Nth \Delta X_j\Nth
	\right]
	\\ &=
	\frac{\alpha_N a_N^2 }{N}
	q_i\Nth(a_N^{-1}\z+e_1) 
	(\delta_{ij} -q_j\Nth(a_N^{-1}\z+e_1))
	\\ \nonumber
	& \quad {}+
	\alpha_N a_N^2 
	\mathbb{E}_{a_N^{-1}\z+e_1}
	\left[  \Delta X_i\Nth  \right]
	\mathbb{E}_{a_N^{-1}\z+e_1}
	\left[ \Delta X_j\Nth \right]
	\\ \nonumber
	&=
	\frac{ a_N }{\epsilon_1\Nth N} z_i \delta_{ij}
	+
	\mathcal{O} \left(\frac{1 }{\epsilon_1\Nth N}\right)
	+
	\mathcal{O} \left(\frac{a_N }{N}\right)
	\\ \nonumber
	& \quad 
	{}+ \mathcal{O}(\epsilon_1\Nth )
	+ \mathcal{O}(a_N\epsilon\Nth )
	+ \mathcal{O}(a_N^{-1}).
\end{align*}
%\begin{align*}
%        \lim_{N\to \infty}
%       \alpha_N \mathbb{E}_{\z} 
%       \left[\Delta Z_i\Nth  \right]  
%       \mathbb{E}_{\z} 
%        \left[\Delta Z_j\Nth  \right]
%        &=\begin{cases}
	%       \text{ if }  
	%       & a_N\prec N\epsilon_1\Nth, \quad 
	%      0 ;
	%%     \\
	%        \text{ if } 
	%        & a_N = N\epsilon_1\Nth, \quad 
	%        z_i \delta_{ij} ;
	%        \\
	%        \text{ if }  
	%        & a_N\succ N\epsilon_1\Nth, \quad \infty 
	%        \end{cases}     
%        \end{align*}
By choosing the space scaling $a_N= N\epsilon_1\Nth$, 
\begin{align*}
\lim_{N\to \infty}
\alpha_N \mathbb{E}_{\z} 
\left[\Delta Z_i\Nth  \Delta Z_j\Nth \right]
&= 
z_i \delta_{ij} ;     
%\quad\quad \comm{\footnotesize or \quad 0 \; \text{ if} \; a_N\prec N\epsilon_1\Nth; \infty \;  \text{  if} \;  a_N\succ N\epsilon_1\Nth}
\end{align*}
and with this spatial scale, the drift depends on the strength of the mutation parameters, that is, 
\begin{align*}
\lim_{N\to \infty}
\alpha_N \mathbb{E}_{\z} 
\left[
\Delta Z_i\Nth  \right]  
&=\begin{cases}
	-\frac{z_i}{2} & \text{if }  \epsilon_N \prec \frac{1}{N};
	\\
	-\frac{z_i}{2} + \frac{\theta}{2} P_{1i} & \text{if } \epsilon_N= \frac{1}{N}. 
	% \\
	%\text{ if }  
	% & \epsilon_N \succ \frac{1}{N}, \quad \infty 
\end{cases}     
\end{align*}
%\com{[Answer to Paul's most-asked question]} 
Note that the space scale $a_N= N\epsilon_1\Nth$ here is larger than the one in  case (a, i), which is instead $a_N= \sqrt{N\epsilon_1\Nth}$. This fundamental difference is why we can no longer appeal to a functional central limit theorem. The reason for the difference is that the term
\[
q_i\Nth(a_N^{-1}\z+e_1)(\delta_{ij} -q_j\Nth(a_N^{-1}\z+e_1))
\]
in the covariance calculations above 
%\eqref{inf_var_equilibrium_fluctuations}
decays like $a_N^{-1}$ due to the equilibrium point being on the boundary, whereas 
\[
q_i\Nth(a_N^{-1}\z+\Q)(\delta_{ij} -q_j\Nth(a_N^{-1}\z+\Q))
\]
in the covariance calculations of the neutral PIM case (a, i)  
%\eqref{inf_var_equilibrium_fluctuations_feller} 
converges to $Q_i (\delta_{ij}-Q_j)>0$.
Therefore, in the first case, the equilibrium point on the boundary introduces a factor $a_N$ instead of $a_N^2$ in the infinitesimal covariance and thus requires a different scaling. 

It is now clear that our convergence result in Proposition \ref{thm:diffusion_limits}(c)  concerning the WF diffusion is analogous to a convergence result for the WF model with the following characteristics: 
\begin{itemize}[noitemsep]
\item
selection on the first allele is stronger than selection on other alleles and on mutation:  \\
$
\epsilon_1\Nth \succ
\epsilon\Nth, \epsilon_i\Nth,\, i=2,\dots,d, \; \sigma_1=1
$;
\item time scaling: $\alpha_N^{-1}= \epsilon_1\Nth $;
\item 
selection on the first allele is stronger than genetic drift: % (thus we get a deterministic limit):
$\epsilon_1\Nth \succ \frac{1}{N}$;
\item  spatial scaling: 
$a_N= N\epsilon_1\Nth$;
\item  mutation strength is comparable to genetic drift:  
$\epsilon\Nth=\frac{1}{N}$.
\end{itemize}
In fact, under the assumptions above, the sequence of rescaled WF models $\Z\Nth$
%(\lfloor \alpha_N t \rfloor)=a_N (\X\Nth(\lfloor \alpha_N t \rfloor)-\e_1 )$ 
converges to a diffusion
$ \{\Z( t )\}_{t \geq 0}$ 
with diffusion coefficients 
$z_i \delta_{ij},$ $i,j=2,\dots,d$, and 
drift coefficients 
$-\frac{z_i}{2} + \frac{\theta}{2} P_{1i}$, $i=2,\dots,d$, 
which is precisely the solution to the SDE \eqref{eq:SDE_CBI} and corresponds to 
the limit of the sequence of WF diffusions $ \sigma ( \X^{(\sigma)}(\sigma^{-1} t ) - \e_1 )$, $\sigma\to\infty$, as considered  in Proposition \ref{thm:diffusion_limits}(c). An analogous diffusion is also derived in \cite{Ethier1988} and \cite{Nagylaki1990}
% eq (38) Nagylaki1990
in a similar setting for rare alleles.

%\comm{
%\subsubsection*{Examples}
%(a, i, PIM, neutral) 
%$N^{-1}\prec \alpha_N^{-1} =  \epsilon\Nth = N^{-1/2}, \quad a_N= (N\epsilon\Nth)^{1/2}=N^{1/4} $
%\\
%(a, ii) 
%$ N^{-1}\prec  \alpha_N^{-1} =  \epsilon_1\Nth = N^{-1/2} \succ \epsilon_i\Nth, i\neq 1, \epsilon\Nth = N^{-1}, \quad a_N= N\epsilon_1\Nth=N^{1/2}$
%}

%%%%%%%%%%%%%%%%%%%%%%%%%%%%%%%%%%%%%%%%%%%%%%%%%%%%%%%%%%%%%%%%%%%
%%                                                               %%
%% You have reached the end of your document.                    %%
%%                                                               %%
%%%%%%%%%%%%%%%%%%%%%%%%%%%%%%%%%%%%%%%%%%%%%%%%%%%%%%%%%%%%%%%%%%%

\end{document}